\documentclass[11pt]{amsart}

\usepackage{parskip}
\allowdisplaybreaks

\usepackage{geometry} 
\usepackage{framed} 

\usepackage{tikz}
\tikzset{every picture/.style={line width=1pt}} 

\usepackage{setspace}
\setdisplayskipstretch{2.5}

\usepackage{amsthm} 
\usepackage{amsmath}
\usepackage{amssymb}
\usepackage{mathrsfs}

\usepackage{hyperref}

\usepackage{amsfonts} 

\newcommand\FB{\mathfrak{B}}

\newcommand\CC{\mathcal{C}}
\newcommand\FC{\mathfrak{C}}
\newcommand\BC{\mathbf{C}}

\newcommand\BD{\mathbf{D}}
\newcommand\Dbf{\mathbf{D}}
\newcommand\D{\mathcal{D}}

\newcommand\hfrak{\mathfrak{h}}
\renewcommand\H{\mathcal{H}}

\def\P{\mathcal{P}}
\newcommand\pfrak{\mathfrak{p}}

\newcommand\R{\mathbb{R}}

\renewcommand\S{\mathcal{S}}
\newcommand\sfrak{\mathfrak{s}}

\newcommand\V{\mathcal{V}}

\newcommand\Z{\mathbb{Z}}

\newcommand\w{\omega}
\newcommand\vphi{\varphi}
\renewcommand\phi{\vphi}
\newcommand\eps{\varepsilon}

\usepackage{stmaryrd} 

\newcommand\id{\textnormal{id}}
\newcommand\sing{\textnormal{sing}}
\newcommand\reg{\textnormal{reg}}
\newcommand\spt{\textnormal{spt}}

\newcommand\dist{\textnormal{dist}}
\newcommand\graph{\textnormal{graph}}
\newcommand\ext{\mathrm{d}}
\newcommand\del{\partial}

\newcommand{\res}{\mathbin{\hspace{0.1em}\vrule height 1.3ex depth 0pt width 0.13ex\vrule height 0.13ex depth 0pt width 1.0ex}} 

\newcommand{\weakly}{\rightharpoonup}

\makeatletter
\def\@tocline#1#2#3#4#5#6#7{\relax
  \ifnum #1>\c@tocdepth 
  \else
    \par \addpenalty\@secpenalty\addvspace{#2}%
    \begingroup \hyphenpenalty\@M
    \@ifempty{#4}{%
      \@tempdima\csname r@tocindent\number#1\endcsname\relax
    }{%
      \@tempdima#4\relax
    }%
    \parindent\z@ \leftskip#3\relax \advance\leftskip\@tempdima\relax
    \rightskip\@pnumwidth plus4em \parfillskip-\@pnumwidth
    #5\leavevmode\hskip-\@tempdima
      \ifcase #1
       \or\or \hskip 1em \or \hskip 2em \else \hskip 3em \fi%
      #6\nobreak\relax
    \dotfill\hbox to\@pnumwidth{\@tocpagenum{#7}}\par
    \nobreak
    \endgroup
  \fi}
\makeatother

\newtheoremstyle{newtheoremstyle}
{3pt}
{3pt}
{\itshape}
{\parindent}
{\bfseries}
{.}
{0.5em}
{} 

\newtheoremstyle{newtheoremstyledefn}
{3pt}
{3pt}
{}
{\parindent}
{\bfseries}
{.}
{0.5em}
{}

\theoremstyle{newtheoremstyle}
\newtheorem{theorem}{Theorem}
\newtheorem*{theorem*}{Theorem}
\newtheorem{lemma}[theorem]{Lemma}

\newtheorem{corollary}[theorem]{Corollary}

\newtheorem{thmx}{Theorem}

\theoremstyle{newtheoremstyledefn}
\newtheorem{defn}[theorem]{Definition}

\newtheorem{remark}[theorem]{Remark}

\numberwithin{equation}{section} 
\numberwithin{theorem}{section}

\usepackage{fancyhdr}
\begin{document}

\title{An optimal regularity theory for immersed stable minimal hypersurfaces with small singular set}

\author{
	Paul Minter
        \and 
        Zhengyi Xiao
}

\address{\textnormal{Department of Pure Mathematics and Mathematical Statistics, University of Cambridge}}
\email{pdtwm2@cam.ac.uk, zx273@cam.ac.uk}

\begin{abstract}
	We show that if $M^n$ is a properly immersed, two-sided, stable minimal hypersurface in $B^{n+1}_1(0)\setminus S$, where $S$ is closed with $\H^{n-2}(S)=0$, then $\dim_\H\sing(M)\leq n-7$, namely $\overline{M}\cap B^{n+1}_1(0)$ is represented by a smooth minimal immersion outside a closed set of generally unavoidable singularities which has Hausdorff dimension at most $n-7$. This provides the optimal a priori size assumption on the non-immersed singular set in order to guarantee optimal regularity. Consequently, such objects form a compact class under mass upper bounds.
\end{abstract}

\maketitle

\tableofcontents

\part{Introduction and Main Theorems}

The regularity theory of stable minimal hypersurfaces has been extensively studied over the last 50 years. To appropriately frame our results we briefly summarise the developments on this problem, starting with the case of \emph{embedded} stable minimal hypersurfaces.

The renowned regularity theory of Schoen--Simon \cite{SS81} shows that if an $n$-dimensional stable minimal hypersurface is smoothly embedded outside a closed $\H^{n-2}$-null set, then the hypersurface can in fact be extended to one which is smoothly embedded away from a closed set of generally unavoidable singularities which has Hausdorff dimension at most $n-7$ (this dimension bound is optimal in light of standard singularity models). Furthermore, this class is compact in the varifold topology under uniform mass bounds. In fact, the arguments in \cite{SS81} readily extend to the case where one assumes embeddedness away from a set of finite $\H^{n-2}$-measure or, more generally, a set of vanishing $2$-capacity. Around 30 years later, Wickramasekera \cite{Wic14} showed that the results of Schoen--Simon remain true under the significantly weaker (and optimal) assumption that the stable minimal hypersurface is smoothly embedded outside a $\H^{n-1}$-null set. In fact, Wickramasekera shows that this assumption is equivalent to a geometric structural condition (which is a priori much weaker), namely that the associated varifold has no so-called \emph{classical singularities}; these include, for instance, immersed singular points. It is in this way that one can view Wickramasekera's work as fully resolving the regularity and compactness question for \emph{embedded} stable minimal hypersurfaces.

The corresponding questions for properly \emph{immersed} two-sided stable minimal hypersurfaces remained very much open until recent years. Initially, the work of Schoen--Simon--Yau \cite{SSY75} provided a priori curvature estimates for stable minimal hypersurfaces smoothly immersed in Riemannian manifolds of dimension at most $6$. More recently, Bellettini \cite{Bel23} extended the work of Schoen--Simon--Yau to Euclidean space of dimension at most $7$ (the optimal dimension range for such estimates), and also extended the key sheeting theorem in the aforementioned work of Schoen--Simon to stable minimal hypersurfaces which are smoothly \emph{immersed} outside a closed set of vanishing $2$-capacity. Consequently, Bellettini is able to \emph{rule out} branch point singularities occurring if the hypersurface is known to be smoothly immersed outside a $\H^{n-2}$-null set. Shortly after, the work of Hong--Li--Wang \cite{HLW24} adapted Bellettini's methods to provide a regularity and compactness theorem for stable minimal hypersurfaces which are smoothly immersed outside a closed set of Hausdorff dimension \emph{strictly smaller than} $n-4+\frac{4}{n}$ (in particular, branch points are ruled out by Bellettini's work), in particular showing that such objects must be smoothly immersed outside a closed set of Hausdorff dimension at most $n-7$. We also mention that earlier work of Wickramasekera \cite{Wic08} already established these results in the multiplicity $2$ setting (see also \cite{MW24}).

In the present article, we complete part of this theory by establishing that an $n$-dimensional stable minimal hypersurface which is two-sided and properly immersed outside a closed $\H^{n-2}$-null set necessarily has a singular set of Hausdorff dimension $\leq n-7$ (which is again optimal). This provides the generalisation of the Schoen--Simon regularity theory from the embedded case to the immersed (non-branched) setting. We stress that, unlike in the embedded setting, in the immersed setting this a priori size assumption on the singular set is optimal for the conclusion. Indeed, if instead $\H^{n-2}(\sing(M))>0$ then the result need not hold as, for instance, $\overline{M}$ can then contain branch points (cf.~\cite{SW07, FMM26}).

Our main theorem is thus the following regularity and compactness theorem. Here, if $\iota: M\to B^{n+1}_2(0)$ is a smooth, proper, minimal immersion, we identify $M$ with the stationary integral varifold associated to the pushforward $\iota_\#M$ in the standard fashion, namely $V = \mathbf{v}(\iota(M),\theta)$ there $\theta(x):= \#\{y:\iota(y) = x\}$.

\begin{thmx}\label{thm:main-compactness}
    Let $n\geq 2$. Suppose that $(M_j)_{j=1}^\infty$ is a sequence of two-sided, stable minimal hypersurfaces smoothly properly immersed in $B^{n+1}_1(0)$ with:
    \begin{itemize}
        \item $0\in \overline{M}_j$ for all $j\geq 1$;
        \item $\sup_j \H^n(M_j\cap B^{n+1}_1(0))<\infty$;
        \item $\H^{n-2}(\sing(M_j))=0$, where $\sing(M_j):= B^{n+1}_1(0)\cap (\overline{M}_j\setminus M_j).$
    \end{itemize}
    Then, there exists a subsequence $\{j^\prime\}\subset\{j\}$, a stationary integral varifold $V$ in $B^{n+1}_1(0)$, and a relatively closed subset $S\subset \spt\|V\|\cap B^{n+1}_{1}(0)$ where $\dim_\H(S)\leq n-7$ such that:
    \begin{enumerate}
        \item [\textnormal{(i)}] $V\res (B^{n+1}_1(0)\setminus S)$ is represented by a proper, two-sided, stable minimal immersion $\iota:M\to B^{n+1}_1(0)$;
        \item [\textnormal{(ii)}] $M_{j^\prime}$ converges to $V$ locally in $B^{n+1}_1(0)$ in the varifold topology;
        \item [\textnormal{(iii)}] $M_{j^\prime}$ converges to $M$ locally smoothly in $B^{n+1}_1(0)\setminus S$.
    \end{enumerate}
    In particular, if $M$ is a two-sided stable minimal hypersurface which is properly smoothly immersed in $B^{n+1}_1(0)$ with $\H^{n-2}(\sing(M))=0$, then in fact $\overline{M}\cap B^{n+1}_1(0)$ is represented by a proper, two-sided, smooth, stable minimal immersion $M^\prime$ with $\dim_\H(\sing(M^\prime))\leq n-7$.
\end{thmx}

We will not distinguish between $M$ and its extension $M^\prime$, and will simply write $M$ for both.

\begin{remark}
As usual, if $2\leq n\leq 7$ and we write $\dim_\H(A)\leq n-7$ for a set $A$ we mean that $A=\emptyset$ if $2\leq n\leq 6$ and $A$ is discrete if $n=7$.
\end{remark}
\begin{remark}
	In Theorem \ref{thm:main-compactness} we will in fact show that $S\subset \S_{n-7}$, where $\S_{n-7}$ is the $(n-7)$-stratum of the singular set of $V$. Hence, by Naber--Valtorta \cite{NV20}, $S$ (and hence $\sing(M)$) is countably $(n-7)$-rectifiable. As we have an $\eps$-regularity theorem (cf.~Theorem \ref{thm:main-reg-basic}), one should also get uniform $(n-7)$-dimensional Minkowski estimates on $\sing(M)$.
\end{remark}

\medskip

Theorem \ref{thm:main-compactness} improves on the work of Hong--Li--Wang \cite{HLW24} which, as mentioned, proved the same result assuming the stronger assumption $\dim_\H(\sing(M)) < n-4+\frac{4}{n}$ (for $n=2$, we take this to mean $\sing(M)=\emptyset$) in place of $\H^{n-2}(\sing(M))=0$. It also answers the regularity and compactness question of Bellettini \cite{Bel23} in the \emph{non-branched} setting.

The corresponding statement to Theorem \ref{thm:main-compactness} when $\H^{n-2}(\sing(M))>0$, namely when \emph{allowing} branch points, remains an interesting open problem.

Nonetheless, our proof of Theorem \ref{thm:main-compactness} \emph{relies} on both the regularity results of \cite{Bel23, HLW24}. Indeed, as mentioned in \cite[Part II]{Bel23}, after establishing the sheeting theorem in \cite[Theorem 5]{Bel23} the missing ingredient is an $\eps$-regularity theorem for $M$ as in Theorem \ref{thm:main-compactness} near cones formed of \emph{unions} of half-hyperplanes (which includes unions of full hyperplanes). The main technical work of the present paper is proving this $\eps$-regularity theorem when $\H^{n-2}(\sing(M))=0$.

To describe the main difficulties in proving such an $\eps$-regularity theorem, the first point to note is that, unlike in \cite{Bel23, HLW24}, our method of proof is \emph{not} based on intrinsic PDE analysis, but instead on methods within geometric measure theory. This difference in approach is reminiscent of how the results of Wickramasekera \cite{Wic14} are, at present, only known to be provable using ideas from geometric measure theory. Indeed, the core analytical difficulty in both cases with a PDE approach is that the a priori size assumption on the singular set is too large to allow one to extend key integral identities (typically involving powers of the norm of the second fundamental form) across the singular set\footnote{In the sheeting theorems (near hyperplanes) of Schoen--Simon and Bellettini, the singular set having $2$-capacity zero is sufficient for this. In \cite{HLW24}, for the $\eps$-regularity near unions of (half-)hyperplanes, one needs to extend suitable identities which involve powers of the norm of the second fundamental form of the hypersurface, which requires a smaller size assumption on the singular set, which is the reason their result requires the singular set to have Hausdorff dimension $<n-4+\frac{4}{n}$.}. In our approach, like in \cite{Wic14}, we work inductively on the density, using regularity information obtained at singular points of lower density to understand those of higher density. It is within this inductive bootstrapping of regularity where we utilise the regularity theories of \cite{Bel23, HLW24} to understand regions with sufficient regularity. This should be compared to how the Schoen--Simon theory is used in \cite{Wic14}.

However, compared to \cite{Wic14} (and the original work of Simon \cite{Sim93}) the immersed case presents significant new difficulties in that it allows the possibility of so-called \emph{density gaps}, i.e.~regions in a varifold $V$ which is close to a cone $\BC$ where there are no points in $V$ which have density at least as large as the density of the vertex in $\BC$. An illustrative example is when $\BC$ is a union of 3 planes in $\R^3$ which all intersect along a common line, yet $V$ is a nearby union of $3$ planes which only intersect at $0$.

To resolve this issue, we draw inspiration from the work of Becker-Kahn \cite{BK17} and the recent work \cite{BKMW25}, where regularity theorems were established in multiplicity $2$ situations where density gaps could occur. It is in understanding density gaps that we utilise the regularity theory of \cite{HLW24}. However, we must also extend these ideas to higher multiplicity settings, which itself presents further difficulties. We overcome these by taking inspiration from the works \cite{MW24, Min21a}, where one inductively establishes $\eps$-regularity theorems of successive ``fineness'' in order to prove the final regularity theorem, which is done via an excess decay argument once one inductively understands the regularity of so-called ``blow-ups'' (which describe the linearised problem). Thus, our work uses and further develops the ideas in the works \cite{Sim93, Wic14, BK17, MW24, Min21a, BKMW25}, and ultimately presents a unified framework of these regularity techniques. We will describe this emerging philosophy more precisely in Section \ref{sec:prelim}.

We stress that our proof is therefore \emph{not} a reduction to a direct application of \cite{HLW24}; Theorem \ref{thm:main-compactness} and the associated regularity theorem (cf.~Theorem \ref{thm:main-reg-basic}) are proved, in full, using different methods. We utilise \cite{HLW24} only in order to understand the structure of the hypersurface in density gaps, with the applicability of \cite{HLW24} being possible via our inductive scheme.

In a follow-up work, we plan to address a further generalisation of the results in the present paper in which the a priori size assumption on the singular set is replaced by (weaker) structural assumptions of a similar nature to those in \cite{Wic14}. Looking to the future with this as a motivation, we will often work with more general classes of varifolds one strictly needs for Theorem \ref{thm:main-compactness}.

\textbf{Acknowledgements.} This research was conducted during the period PM was a Clay Research Fellow. The authors would like to thank Neshan Wickramasekera for many helpful discussions.

\section{Preliminaries and Reductions}\label{sec:prelim}

The overarching structure of the proof of Theorem \ref{thm:main-compactness} follows the same blueprint as that in \cite{Wic14} (and in turn the simpler \cite{SS81}): one first passes to a weak limit in the varifold topology, and then analyses the possible tangent cones at singular points in the limit as well as the convergence. As in \cite{Wic14}, we do this by working inductively on the density of the singular point, establishing quantitative $\eps$-regularity theorems near flat cones (namely, those supported on unions of half-hyperplanes) as we go. The inductive nature of the argument allows one to leverage better regularity properties in regions of lower density than that of the cone which one is close to. Furthermore, these $\eps$-regularity theorems ultimately allow us to improve the convergence of the sequence away from the (non-immersed) singular set of the limit.

Hence, there are several key ingredients which we need to prove Theorem \ref{thm:main-compactness}. These are:
\begin{enumerate}
	\item [(i)] \emph{An $\eps$-regularity theorem near hyperplanes with multiplicity.} In the present context, this follows from the work of Bellettini \cite{Bel23}. We recall this in Theorem \ref{thm:bellettini}.
	\item [(ii)] \emph{An $\eps$-regularity theorem near cones formed of sums of half-hyperplanes} (with the sum not a single hyperplane with multiplicity). Ultimately, we will show that the only such cones which are possible tangent cones in the context of Theorem \ref{thm:main-compactness} are sums of \emph{full} hyperplanes rather than just half-hyperplanes, but in the proof we need to analyse this more general case. This is the main technical result of the present work (cf.~Theorem \ref{thm:main-reg}).
	\item [(iii)] \emph{A classification theorem for low-dimensional hypercones $\BC$ in $\R^{n+1}$ which are represented by two-sided smooth immersions which are stable (as an immersion) away from the origin.} We will show that, if $n\in\{2,3,4,5,6\}$, then such $\BC$ must be the union of (full) hyperplanes. This is the natural analogue in the immersed setting of Simons' classification for stable minimal hypercones which are \emph{embedded} away from $0$ (and indeed the proof is the same, see Theorem \ref{thm:main-Simons}).
\end{enumerate}
As previously mentioned, at present we do not know an alternative approach for proving the $\eps$-regularity theorem mentioned in (ii) above only under the $\H^{n-2}$-null assumption on the singular set, and indeed our proof will utilise tools from geometric measure theory to overcome the apparent difficulties with an intrinsic PDE approach. Again, if one assumes a priori that the size of the non-immersed singular set is smaller, namely having Hausdorff dimension \emph{strictly smaller} than $n-4+\frac{4}{n}$, then Hong--Li--Wang \cite{HLW24} have succeeded in using an intrinsic approach to proving this $\eps$-regularity theorem. In our inductive approach, we are able to gain leverage from the inductive information to apply the result of \cite{HLW24} in `density gaps'; this is done by ruling out (non-immersed) singular points in the high-dimensional strata of the singular set. If one does not work inductively, it is unclear whether the singular set could still be too large in such regions.

We begin by stating versions of our main results for $M$ as in Theorem \ref{thm:main-compactness}. We again denote by $V$ the natural varifold associated to the image of $M$, and often we will not distinguish between $M$ and $V$ (and sometimes not even with $\iota(M)$). Furthermore, for any cone $\BC$, we define the (\emph{one-sided}) $L^2$\emph{-height excess of $V$ relative to $\BC$} by:
$$E_{V,\BC}:=\left(\int_{B^{n+1}_1(0)}\dist^2(x,\BC)\, \ext\|V\|(x)\right)^{1/2}.$$
Here, $\BC$ is an integral varifold, and we have used the shorthand $\dist(x,\BC)\equiv \dist(x,\spt\|\BC\|)$.
\begin{thmx}[$\eps$-Regularity Theorem]\label{thm:main-reg-basic}
	Fix $\Lambda\in (0,\infty)$, and let $\BC = \bigcup_{i=1}^N P_i$, where $P_i$ are distinct hyperplanes in $\R^{n+1}$ passing through $0$. Then, there exists $\eps = \eps(\Lambda,\BC)\in (0,1)$ and $Q = Q(n,\Lambda)\in \{1,2,3,\dotsc\}$ such that the following holds. Suppose that $M$ is a two-sided properly immersed stable minimal hypersurface in $B^{n+1}_1(0)$ with $\H^n(M)\leq\Lambda$ and $\H^{n-2}(\sing(M))=0$. Then, if
	$$\dist_\H(M\cap B^{n+1}_1(0), \BC\cap B^{n+1}_1(0)) < \eps,$$
	then there exists $q\in \{1,2,\dotsc,Q\}$ and integers $i_1,\dotsc,i_q\in \{1,\dotsc,N\}$ such that for each $j\in\{1,\dotsc,q\}$ there is a $C^{2}$ function $u_j:B^{n+1}_{1/2}(0)\cap P_{i_j}\to P_{i_j}^\perp$ such that
	$$\overline{M}\cap B^{n+1}_{1/2}(0) = \bigcup^q_{j=1}\graph(u_j)\cap B^{n+1}_{1/2}(0),$$
	with $\|u_j\|_{C^{2}}\leq CE_{V,\BC}$. Furthermore:
	\begin{enumerate}
		\item [\textnormal{(i)}] each $u_j$ solves the minimal surface equation on their domain (hence, each $u_j$ is smooth);
		\item [\textnormal{(ii)}] for each $\ell=1,\dotsc,N$ we have $\{j:i_j = \ell\} \neq\emptyset$.
	\end{enumerate}	
	In particular, $\overline{M}\cap B^{n+1}_{1/2}(0)$ is a smoothly immersed minimal submanifold. Here, $C = C(n,\Lambda)\in (0,\infty)$.
\end{thmx}

\begin{remark}
	One could also phrase the conclusion of Theorem \ref{thm:main-reg-basic} as a bound on the size of the second fundamental form of $M$, similarly to that seen in the related works \cite{Bel23, HLW24}.
\end{remark}

\begin{remark}\label{remark:dependence}
	One should be able to remove the dependence of $\eps$ on $\BC$ in Theorem \ref{thm:main-reg-basic}, so that in fact $\eps$ only depends on $n,\Lambda$. The difficulty in doing this is ensuring that the estimate on the $u_j$ still holds, as the planes in $\BC$ might be very close to one another. We will not need this for Theorem \ref{thm:main-compactness}, so we do not pursue it.
\end{remark}

\medskip

For the second theorem, it is convenient to first fix a choice of metric $d$ which metrises the $n$-varifold topology on $B^{n+1}_{1}(0)$ (induced by its Fréchet structure); this can be chosen to depend only on $n$. We fix such a $d$ throughout the paper.

\begin{thmx}[Minimal Distance Theorem]\label{thm:minimum-distance}
	Let $\BC$ be a $n$-dimensional stationary cone in $\R^{n+1}$ such that $\BC$ can be written as a finite sum of $n$-dimensional half-hyperplanes all with the same $(n-1)$-dimensional boundary, yet $\BC$ cannot be written as a sum of hyperplanes. Then, there exists $\eps = \eps(\BC)\in (0,1)$ such that the following holds. Suppose $M$ is a two-sided stable minimal hypersurface smoothly properly immersed in $B^{n+1}_1(0)$ with $\H^{n-2}(\sing(M))=0$, and let $V$ be the natural varifold associated to $M$. Then:
	$$d(V\res B^{n+1}_1(0),\BC\res B^{n+1}_1(0))\geq\eps.$$
\end{thmx}

\begin{remark}
	Often when a constant depends on $n$, $Q:=\Theta_{\BC}(0)$, and $\BC$, as a shorthand we will simply write that it depends on $\BC$, as the information of $n$ and $Q$ is encoded within $\BC$. In fact, similar to Remark \ref{remark:dependence}, when a constant depends on $\BC$ it should in fact only depend on $n,Q$, and the value $\inf_P \dist_\H(\spt\|\BC\|\cap B_1, P\cap B_1)$, where the infimum is taken over all hyperplanes $P$.
\end{remark}

\medskip

The reader should compare Theorem \ref{thm:main-reg-basic} and Theorem \ref{thm:minimum-distance} with \cite[Theorems 3.3 \& 3.4]{Wic14}. The differences are that Theorem \ref{thm:main-reg-basic} applies to the cone being a \emph{sum} of hyperplanes rather than a single hyperplane, and the minimum distance theorem now only applies to \emph{non-immersed} unions of half-hyperplanes (accounting for multiplicity). We stress that in Theorem \ref{thm:minimum-distance} it is \emph{not} necessarily true that $\dist_\H(\spt\|V\|\cap B^{n+1}_1(0),\spt\|\BC\|\cap B^{n+1}_1(0))\geq \eps$, as it is possible for $\spt\|\BC\|$ to be the union of hyperplanes whilst $\BC$ is not the sum of hyperplanes\footnote{For instance, take $\BC = \BC_0 + |P_1|+|P_2|+|P_3|$, where $\BC_0$ is a triple junction cone formed of half-planes $H_1,H_2,H_3$ with $P_i$ being the hyperplane formed by extending $H_i$ for each $i\in \{1,2,3\}$.}. This is why the conclusion of Theorem \ref{thm:minimum-distance} must be stated in a way which accounts for multiplicity, which $d$ does.

\begin{remark}\label{remark:slicing}
    One should note that in \cite{SS81}, namely when the non-embedded singular set is $\H^{n-2}$-null, one does not need an $\eps$-regularity theorem near \emph{unions} of (half-)hyperplanes as one wishes to prove a minimum distance theorem near any such union (which is a consequence of stability and a simple slicing argument, as shown in \cite{SS81}). In fact, the slicing argument in \cite{SS81} can still be used to prove Theorem \ref{thm:minimum-distance}, as the smallness assumption on the (non-immersed) singular set of $M$ is sufficient for the argument (in fact, knowing $\H^{n-1}(\sing(M))=0$ would suffice). The reason this simple argument is not available to prove the Minimal Distance Theorem in \cite{Wic14} is due to there being no a priori size assumption on the singular set, only structural conditions.
\end{remark}

\medskip

The proof of Theorem \ref{thm:main-compactness} follows from Theorem \ref{thm:main-reg-basic} and Theorem \ref{thm:minimum-distance} when combined with the following extension of Simons' classification \cite{Sim68} of low-dimensional minimal cones with isolated singularities and which are two-sided and stable as immersions.

\begin{thmx}[Simons' Classification]\label{thm:main-Simons}
	Suppose $n\in \{2,3,4,5,6\}$ and $\BC$ is an $n$-dimensional minimal cone smoothly immersed in $\R^{n+1}\setminus\{0\}$ which is represented by a two-sided and stable immersion (with $\Theta_{\BC}(0)<\infty$). Then, $\BC$ is a finite sum of hyperplanes.
\end{thmx}

\begin{proof}
	The proof is identical to the embedded case by Simons \cite{Sim68}, instead using the stability inequality for immersions. The reader can consult \cite[Appendix B]{Sim83} for a proof in the embedded case which directly extends to the two-sided immersed case.
\end{proof}

Let us now detail how Theorem \ref{thm:main-compactness} follows from Theorem \ref{thm:main-reg-basic}, Theorem \ref{thm:minimum-distance}, and Theorem \ref{thm:main-Simons}. Here, for a stationary integral varifold $V$ we write $\reg_*(V)$ for the set of smoothly immersed points in $V$, i.e.~points in $\spt\|V\|$ locally about which $V$ is represented by a smooth proper immersion (or, equivalently, a finite sum of embedded minimal submanifolds). We also write $\sing_*(V):= \spt\|V\|\setminus \reg_*(V)$ for the set of non-immersed singular points. Notice that the usual regular set $\reg(V)$ of $V$, which consists of points in $\spt\|V\|$ locally about which $V$ is represented by a smooth embedding, satisfies $\reg(V)\subseteq \reg_*(V)$. Thus, the usual singular set $\sing(V):= \spt\|V\|\setminus \reg(V)$ satisfies $\sing_*(V)\subseteq \sing(V)$.

\begin{proof}[Proof of Theorem \ref{thm:main-compactness}]
	Let $V_j$ denote the natural stationary integral varifold associated to the immersion $M_j$. By Allard's compactness theorem for stationary integral varifolds with a mass upper bound \cite{All72}, we can find a subsequence $\{j^\prime\}\subset\{j\}$ with $V_{j^\prime}\to V$ as varifolds in $B^{n+1}_2(0)$, where $V$ is a stationary integral $n$-varifold in $B^{n+1}_1(0)$. Let $\BC$ be a tangent cone at a point $x\in \spt\|V\|$. If $\dim(S(\BC))=n$, i.e.~$\BC$ is a plane with some (integer) multiplicity, we can apply Theorem \ref{thm:main-reg-basic} to $V_{j^\prime}\res B_\rho(x)$ (for sufficiently small $\rho$ and all sufficiently large $j^\prime$) to see that $V_{j^\prime}\res B^{n+1}_{\rho/2}(x)$ is the sum of $\Theta_V(x)$-many minimal graphs, with estimates. Thus, the convergence of $V_{j^\prime}$ to $V$ in $B^{n+1}_{\rho/2}(x)$ is smooth, implying that $x\in \reg_*(V)$, with smooth convergence in $B_{\rho/2}(x)$.
	
	Now suppose $\dim(S(\BC))=n-1$. If $\BC$ was the sum of half-hyperplanes but not the sum of full hyperplanes, then for some sufficiently small $\rho>0$ and for all sufficiently large $j^\prime$ we can apply Theorem \ref{thm:minimum-distance} to $(\eta_{x,\rho})_\#V_{j^\prime}$ to get a contradiction (here, $\eta_{x,\rho}:\R^{n+1}\to \R^{n+1}$ is the map $\eta_{x,\rho}(y):= \rho^{-1}(y-x)$). Thus, $\BC$ must be a sum of full hyperplanes. But then we can apply Theorem \ref{thm:main-reg-basic} to conclude that, for $\rho>0$ sufficiently small, $V_{j^\prime}\res B_\rho(x)$ is the sum of $\Theta_V(x)$-many minimal graphs (not necessarily over the same planes) and thus we can again conclude that $V$ is smoothly immersed about $x$, so that $x\in \reg_*(V)$, and that the convergence of $V_{j^\prime}$ to $V$ is smooth in $B_{\rho}(x)$.
	
	Next consider the case $\dim(S(\BC))=n-2$. We claim that $\BC$ must consist of a sum of hyperplanes; once we have this, we can again apply Theorem \ref{thm:main-reg-basic} to see that $V$ is smoothly immersed about $x$ and that the convergence is locally smooth on some neighbourhood of $x$. Indeed, suppose $y\in \sing(\BC)\setminus\{0\}$. Taking a tangent cone to $\BC$ at $y$, from what we have just shown above (using Theorem \ref{thm:main-reg-basic} and Theorem \ref{thm:minimum-distance}) we see that $y$ must be a smoothly immersed point of $\BC$. Thus, we see that $\BC$ is smoothly immersed away from $0$. Moreover, $\BC$ must be two-sided\footnote{We stress here that this means that there is a two-sided minimal immersion whose associated varifold is $\BC$. This is not the same as saying that the support, $\spt\|\BC\|$, is two-sided as a (possibly immersed) submanifold of $\R^{n+1}$; multiplicity of the image of the immersion is important here. For instance, just as varifolds, one can have a sequence of tori immersed in $\R^3$ which converge to a Klein bottle with multiplicity $2$. In particular, if a regular piece has multiplicity $2$, stability as an immersion allows the two multiplicity one pieces forming the multiplicity $2$ piece to be deformed independently of one another.}
    and stable as an immersion away from $0$ as $\BC$ is a smooth limit of two-sided stable (minimal) hypersurfaces away from $0$, and thus the stability inequality descends to $\BC$. Hence, we can apply Theorem \ref{thm:main-Simons} to conclude that $\BC$ is the sum of finitely many hyperplanes\footnote{Alternatively, in this case one can argue without Theorem \ref{thm:main-Simons}, as one can look at the link of the cross-section of $\BC$ and see that it must be a sum of great circles in $S^2$, which implies that $\BC$ is a sum of finitely many hyperplanes.}. Hence, Theorem \ref{thm:main-reg-basic} once again applies to show $x\in \reg_*(V)$.
	
	We now repeat the above argument inductively when $\dim(S(\BC))=n-\ell$ for $\ell\in \{3,4,5,6\}$. At each stage of the induction, we consider a singular point $y\in \sing(\BC)\setminus\{0\}$ and take a tangent cone there, using the inductive information to conclude that $\BC$ must be smoothly immersed about $y$ and furthermore must be two-sided and stable as an immersion away from $0$. Then we can apply Theorem \ref{thm:main-Simons} to conclude that $\BC$ is a sum of hyperplanes, from which Theorem \ref{thm:main-reg-basic} gives the regularity of $V$ about $x$ and that the convergence is smooth on some neighbourhood of $x$.
	
	Thus, we see that any tangent cone to $V$ at a non-immersed point of $V$ must have $\dim(S(\BC))\leq n-7$, i.e.~the non-immersed singular set $\sing_*(V)$ must be contained within the $(n-7)$-stratum $\S_{n-7}$, which has Hausdorff dimension $\leq n-7$. Hence, $V$ is smoothly immersed outside a closed set $\sing_*(V)$ of Hausdorff dimension at most $n-7$ (it is closed by Theorem \ref{thm:main-reg-basic}). Moreover, the convergence to $V$ is locally smooth outside of $\sing_*(V)$ by Theorem \ref{thm:main-reg-basic}. This completes the proof of Theorem \ref{thm:main-compactness}.
\end{proof}

Thus, we are left with proving Theorem \ref{thm:main-reg-basic} and Theorem \ref{thm:minimum-distance}. In fact, from \cite[Theorem 5]{Bel23}, we already know the validity of Theorem \ref{thm:main-reg-basic} when $\BC$ is a single hyperplane with multiplicity (cf.~Theorem \ref{thm:bellettini}), and so we only need to focus on the case where the cone $\BC$ is not supported on a single hyperplane (in fact, \cite[Theorem 5]{Bel23} also holds when all the hyperplanes in the cone $\BC$ are sufficiently close to a single hyperplane in a specific quantified sense, although we won't use this). 

For the remainder of the paper we will move to a more general set-up, proving $\eps$-regularity theorems there, from which the desired results above follow. These theorems will allow for non-immersed classical singularities of ``top density''. This involves an inductive set-up which we detail in the following section. This more general setup will make the mechanisms of our proofs more transparent.

\subsection{Inductive set-up and main $\eps$-regularity theorem}

Fix $Q\in \{\frac{3}{2},2,\frac{5}{2},\dotsc\}$. Our language and notation will be standard; we refer the reader to \cite{MW24} for more details where necessary.

For the rest of the paper we work with a class $\mathcal{V}_Q$ of integral $n$-varifolds $V$ in $B^{n+1}_2(0)$ obeying:
\begin{enumerate}
	\item [$(\S1)$\ \ ] $V$ is stationary in $B^{n+1}_2(0)$;
	\item [$(\S2)_Q$] Suppose that $\Omega\subset B^{n+1}_2(0)$ is an open set with $\Omega\subseteq \{\Theta_V<Q\}$ and such that $\Omega\cap \sing_*(V)$ is a relatively closed subset of $\Omega$ of Hausdorff dimension $\leq n-7$. Then, $V\res (\Omega\setminus\sing_*(V))$ is represented by a proper, two-sided, smooth, stable minimal immersion.
	\item [$(\S3)_Q$] $V$ obeys the following two structural hypotheses:
		\begin{enumerate}
			\item [$(\S3\text{a})$] If $x\in \spt\|V\|$ is a classical singularity of $V$ with $\Theta_V(x)<Q$, then $V$ is smoothly immersed about $x$ (i.e.~$x\not\in \sing_*(V)$);
			\item [$(\S3\text{b})$] There exists a constant $\eps_*>0$ for which the following holds. If on a ball $B_\rho(x)\subset B^{n+1}_2(0)$ we have
			\begin{enumerate}
				\item [(i)] $(\w_n \rho^n)^{-1}\|V\|(B_\rho(x)) \leq Q+\frac{1}{4}$;
				\item [(ii)] there is a hyperplane $P$ for which $E_{V,P}<\eps_*$;
			\end{enumerate}
			then $V\res B_{\rho/2}(x)$ is equal to the sum of at most $Q$ smooth minimal graphs (over $P$).
		\end{enumerate}
\end{enumerate}
Recall that $\reg_*(V)$ denotes the immersed regular set of $V$, i.e.~the set of points in $\spt\|V\|$ about which $V$ is the sum of a finite number of smoothly embedded submanifolds, and $\sing_*(V)$ is its complement in $\spt\|V\|$. We stress that the constant $\eps_*$ in $(\S3\text{b})$ does not depend on the specific $V\in \V_Q$, but only on the class $\V_Q$. The reader should also note that $\V_Q$ is a compact class under mass bounds, which is one advantage of working with it.

\begin{remark}\label{remark:bellettini}
If $M$ is a two-sided stable minimal hypersurface which is smoothly immersed in $B^{n+1}_2(0)$ with $\H^{n-2}(\sing(M))=0$, then the natural varifold $V$ associated to $M$ obeys $V\in \V_Q$ for all $Q$, for suitable $\eps_* = \eps_*(n,Q)$ therein. As such, any theorems will can prove for $\mathcal{V}_Q$ will also apply to the $V$ associated to such $M$. Indeed, $(\S1)$ is clear. Property $(\S2)_Q$ follows as the singular set is small enough to extend both the two-sidedness and stability to $V\res (\Omega\setminus\sing_*(V))$ in such $\Omega$. Property $(\S3\text{a})$ is also clear, and $(\S3\text{b})$ follows from \cite[Theorem 5]{Bel23}, which we now recall. Here, recall that the $L^2$\emph{-height excess} $E_V$ of a rectifiable $n$-varifold $V$ in $B^{n+1}_2(0)$ relative to the plane $P_0 := \{0\}\times\R^n$ is defined by:
$$E_V:=  \left(\int_{\R^\times B^n_1(0)}|x^1|^2\, \ext\|V\|(x)\right)^{1/2}$$
where we note $|x^1| \equiv \dist(x,P_0)$. The only difference between $E_V$ and $E_{V,P_0}$ is the domain of integration.
\begin{theorem}[{\cite[Theorem 6]{Bel23}}]\label{thm:bellettini}
	Let $n\geq 2$. Then, there exists $\eps = \eps(n,\Lambda)\in (0,1)$ such that if $V$ is a stationary integral $n$-varifold in $B^{n+1}_2(0)$ obeying $0\in \spt\|V\|$, $\|V\|(B^{n+1}_2(0))\leq\Lambda$, and:
	\begin{itemize}
		\item $\H^{n-2}(\sing_*(V))=0$;
		\item $V\res\reg_*(V)$ is represented by a two-sided and stable immersion;
		\item $E_V<\eps$;
	\end{itemize}
	then we have
	$$V\res \left((-\tfrac{3}{4},\tfrac{3}{4})\times B^n_{1/2}(0)\right) = \sum^q_{i=1}\mathbf{v}(u_j)$$
	where $u_j:B^n_{1/2}(0)\to \R$ are smooth (single-valued) functions solving the minimal surface equation, for some $q\in \{1,2,\dotsc\}$. Moreover, for each $j\in \{1,\dotsc,q\}$ we have $\|u_j\|_{C^4}\leq CE_V$, where $C = C(n)\in (0,\infty)$, and $q\leq q_*$ for some $q_* = q_*(n,\Lambda)\in \Z_{\geq 1}$.
\end{theorem}
We will often just refer to Theorem \ref{thm:bellettini} when we are invoking $(\S3\text{b})$.
\end{remark}

\begin{remark}\label{remark:generalised}
	Assumption $(\S3\text{b})$ is actually stronger than what one needs to prove the result: it suffices to only assume such a sheeting theorem near hyperplanes of multiplicity $<Q$; let us write $\widetilde{\V}_Q$ for this slightly modified class, which clearly obeys $\V_Q\subset\widetilde{\V}_Q$. We then immediately have $\S_Q\subset\widetilde{\V}_Q$, where the class $\S_Q$ is as in \cite{MW24}. Thus, results for $\widetilde{\V}_Q$ generalise those for $\S_Q$, which is one motivation for looking at this more general setup. However, knowing a sheeting theorem also holds near hyperplanes of multiplicity $\leq Q$, as we do in the context of Theorem \ref{thm:main-compactness} by Theorem \ref{thm:bellettini}, will simplify some of our later arguments. Thus, for simplicity we assume this stronger version, and work with $\V_Q$.
\end{remark}

\begin{remark}\label{remark:inclusion}
	For $Q_1<Q_2$, up to reducing the value of $\eps_*$ in $(\S3\text{b})$ for $\V_{Q_1}$, we have $\mathcal{V}_{Q_2}\subset\mathcal{V}_{Q_1}$. Thus, when studying a given $\V_Q$, we may assume that the classes $\{\mathcal{V}_{Q^\prime}\}_{Q^\prime\leq Q}$ are naturally decreasing with respect to set inclusion, and hence any theorems we have proved for $\mathcal{V}_{Q_1}$ will also be applicable to $\mathcal{V}_{Q_2}$.
\end{remark}

\begin{remark}
	The reader may wonder why do we work with the more general class $\V_Q$ rather than simply the varifolds associated to $M$ as in Theorem \ref{thm:main-compactness}. As mentioned in Remark \ref{remark:generalised}, one reason is that it is a natural generalisation of those results proved in the proceeding works \cite{Wic14, MW24}. Another is that by working with this more general class, the mechanisms behind our proofs are somehow more clear. This we hope will pave the way for other future works, including follow-up work we have planned for the present paper.
\end{remark}

\medskip

Our main result for the class $\mathcal{V}_Q$, from which Theorem \ref{thm:main-reg-basic} and Theorem \ref{thm:minimum-distance} follow as corollaries, is the following. It concerns cones $\BC$ which are the sum of classical cones. Recall that a \emph{classical cone} $\BC^\prime$ is a stationary cone which is \emph{not} supported on a single hyperplane but still takes the form $\BC^\prime = \sum^N_{i=1}|H_i|$, where the $H_i$ are half-hyperplanes which all have the same set-theoretic boundary (in particular, one needs $N\geq 3$); note that the $H_i$ need not be distinct.

\begin{remark}
	Notice that the density of a classical cone is always in $\{\frac{3}{2},2,\frac{5}{2},3,\dotsc\}$.
\end{remark}

We give an initial, imprecise statement of this regularity theorem; we will make it precise momentarily after making some inductive observations (see Theorem \ref{thm:main-reg}).

\begin{thmx}[Conical $\eps$-Regularity Theorem]\label{thm:main-cone}
	Fix $Q\in\{\frac{3}{2},2,\frac{5}{2},\dotsc\}$. Fix also a (stationary) cone $\BC$ which is the sum of classical cones (at least one) and hyperplanes for which $\Theta_{\BC}(0) = Q$. Then, there exists $\eps = \eps(\BC)\in (0,1)$ such that the following is true. If $V\in \mathcal{V}_Q$ satisfies:
	\begin{enumerate}
		\item [\textnormal{(i)}] $(\w_n 2^n)^{-1}\|V\|(B^{n+1}_2(0))<Q+\frac{1}{2}$;
		\item [\textnormal{(ii)}] $Q-\frac{1}{2}\leq \w_n^{-1}\|V\|(B^{n+1}_1(0))\leq Q+\frac{1}{2}$;
	\end{enumerate}
	then one of the following conclusions must hold:
	\begin{enumerate}
		\item [(a)] $\dist_\H(\spt\|V\|\cap B^{n+1}_1(0),\spt\|\BC\|\cap B^{n+1}_1(0))\geq \eps$;
		\item [(b)] $V\res B^{n+1}_{1/2}(0)$ is a suitable $C^{1,\alpha}$ perturbation of $\BC$, with estimates. Here, $\alpha = \alpha(\BC)\in (0,1)$.
	\end{enumerate}
\end{thmx}

\begin{remark}
	Technically from now on all constants in our theorems depend on the class $\V_Q$, or more precisely on the constant $\eps_*$ in $(\S3\text{b})$. However, we will suppress this dependence in our notation. Note that in our application to Theorem \ref{thm:main-compactness} we have $\eps_* = \eps_*(n,Q)$ (cf.~Remark \ref{remark:bellettini}).
\end{remark}

Before discussing the proof of Theorem \ref{thm:main-cone}, let us first show it implies Theorem \ref{thm:main-reg-basic} and Theorem \ref{thm:minimum-distance}.

\begin{proof}[Proof of Theorem \ref{thm:main-reg-basic}]
	Suppose for contradiction that Theorem \ref{thm:main-reg-basic} failed. Then, we could find a sequence $(M_j)_j$ of two-sided stable minimal hypersurfaces properly smoothly immersed in $B^{n+1}_1(0)$ with $\H^n(M_j)\leq\Lambda<\infty$ for all $j$ and a cone $\BC = \sum^N_{i=1}|P_i|$ with $P_1,\dotsc,P_N$ distinct hyperplanes in $\R^{n+1}$ passing through $0$, such that
	$$\dist_\H(M_j\cap B_1^{n+1}(0),\BC\cap B_1^{n+1}(0)) \to 0.$$
	Let $V_j$ be the natural stationary integral varifold associated to $M_j$, so that $\|V_j\|(B^{n+1}_2(0))\leq \Lambda<\infty$. By Allard's compactness theorem \cite{All72}, we can pass to a subsequence to ensure that $V_j$ converge as varifolds locally in $B^{n+1}_2(0)$ to a stationary integral varifold $V$. By our assumption on the Hausdorff distance, we need $\spt\|V\|\cap B_1(0) = P_1\cup\cdots\cup P_N$. If $\BC$ is supported on a single hyperplane then we can readily check the hypotheses of Theorem \ref{thm:bellettini} for $V_j$ for all $j$ sufficiently large to deduce that $V_j$ satisfies the desired conclusions of Theorem \ref{thm:main-reg-basic}, giving a contradiction. If instead $\BC$ is not supported on a single hyperplane, then we can make an analogous argument using Theorem \ref{thm:main-cone}, with the cone $\widetilde{\BC}$ satisfying $\widetilde{\BC}\res B_1(0) = V$ in place of $\BC$ therein. Indeed, recalling that $V_j\in \V_Q$ for all $Q$ by Remark \ref{remark:bellettini}, clearly Theorem \ref{thm:main-cone}(a) must fail from the convergence to $\BC$, and so Theorem \ref{thm:main-cone}(b) must hold for all $j$ sufficiently large. Since $V_j$ contains no non-immersed classical singularities (as $\H^{n-2}(\sing_*(V_j))=0$) we see that $\widetilde{\BC}$ must be a sum of full hyperplanes (rather than just half-hyperplanes), and that $V_j$ is a sum of minimal graphs over those hyperplanes (with at least one over each hyperplane from the convergence in Hausdorff distance). The estimate on the constituent minimal graphs $u_i^j$ in terms of $E_{V_j,\BC_j}$ follows because for each such graph, on a large set we have $\dist((x,u_i^j(x)),\widetilde{\BC}) = |u_i^j(x)|$, and so one may apply $L^2$ and doubling estimates for each minimal graph. This completes the proof of Theorem \ref{thm:main-reg-basic}.
\end{proof}

\begin{proof}[Proof of Theorem \ref{thm:minimum-distance}]
	As explained in Remark \ref{remark:slicing}, one way of proving this result directly is adapting the slicing argument of Schoen--Simon \cite{SS81} to the present setting. One can also prove it from Theorem \ref{thm:main-cone} as follows.

	We argue by contradiction in an analogous manner to the proof of Theorem \ref{thm:main-reg-basic} above. Indeed, if Theorem \ref{thm:minimum-distance} were false, then we could find a sequence $(V_j)_j$, where each $V_j$ is the natural stationary integral varifold associated to a two-sided properly immersed stable minimal hypersurface $M_j$ with $\H^{n-2}(\sing(M_j))=0$, and a (fixed) cone $\BC$ as in the statement of Theorem \ref{thm:minimum-distance} for which
	$$d(V_j\res B^{n+1}_1(0),\BC\res B^{n+1}_1(0))\to 0.$$
	Again, using Allard's compactness theorem \cite{All72} we can pass to a subsequence to ensure that $V_j\res B^{n+1}_1(0)$ converge locally in $B^{n+1}_1(0)$ to a varifold $V$ in $B^{n+1}_1(0)$, which from the above must satisfy $V = \BC\res B^{n+1}_1(0)$. Then, as $V_j\in \mathcal{V}_Q$, where $Q = \Theta_{\BC}(0)$, for all $j$ sufficiently large we can readily check that the assumptions of Theorem \ref{thm:main-cone} hold for $V_j$ (for this $\BC$), and so we can apply Theorem \ref{thm:main-cone} to see that $V_j$ must be a small $C^{1,\alpha}$ perturbation of $\BC$ (as Theorem \ref{thm:main-cone}(a) must fail). In particular, as by assumption $\BC$ is not the sum of hyperplanes, this means that for all $j$ sufficiently large $V_j$ contains a non-immersed classical singularity, giving that $\H^{n-1}(\sing_*(V_j))>0$ and thus contradicting that $\H^{n-2}(\sing_*(V_j)) = 0$. This gives the desired contradiction and so proves Theorem \ref{thm:minimum-distance}.
\end{proof}
Thus, we just need to prove Theorem \ref{thm:main-cone}. Our proof of this will be by induction on $Q\in \{\frac{3}{2},2,\frac{5}{2},3,\dotsc\}$. In fact, the case when $Q=\frac{3}{2}$, where the cone therein must be the standard multiplicity one triple junction, is already known to hold in much greater generality by the foundational work of Simon \cite{Sim93}. Hence for a fixed $Q\in \{2,\frac{5}{2},3,\frac{7}{2},4,\frac{9}{2},\dotsc\}$, we may inductively assume the validity of Theorem \ref{thm:main-cone} for all $Q^\prime<Q$; we will make this inductive hypothesis henceforth. 

The main way in which we will use this inductive hypothesis is through the following observation:

\begin{lemma}\label{lemma:gaps}
	Fix $Q\in \{2,\frac{5}{2},3,\dotsc\}$ and suppose the validity of Theorem \ref{thm:main-cone} for all $Q^\prime<Q$. Then, if $V\in\mathcal{V}_Q$ we have that 
	$$\dim_\H(\sing_*(V)\cap \{\Theta_V<Q\})\leq n-7.$$
\end{lemma}

\begin{proof}
	The proof of this is by analysing the types of tangent cone which can occur in $V$ at points in $\{\Theta_V<Q\}$ in a manner analogous to that seen in the proof of Theorem \ref{thm:main-compactness}, ultimately showing that $\sing_*(V)\cap \{\Theta_V<Q\}\subseteq \S_{n-7}$, from which the result follows. The key point is that because these singular points have density $<Q$, we are able to invoke $(\S3\text{b}$) (cf.~Theorem \ref{thm:bellettini}) as well as inductively Theorem \ref{thm:main-cone} (recall Remark \ref{remark:bellettini}).
\end{proof}
Lemma \ref{lemma:gaps} understands the size of the non-immersed singular set of $V$ in $\{\Theta_V<Q\}$ (which is an open subset of $B^{n+1}_2(0)$, by upper semi-continuity of $\Theta_V$). Thus in this region we have the desired upper bound on the size of the singular set needed to invoke the regularity theorem in \cite{HLW24} (as two-sidedness and stability as an immersion is given by $(\S2)_Q$). We recall this special case of the result:

\begin{theorem}[{\cite[Theorem 3.1 \& Proposition 4.3]{HLW24}}]\label{thm:HLW}
	Let $n\geq 2$ and $\Lambda\in (0,\infty)$. Let $\BC$ be a (stationary) non-planar cone which is the finite sum of classical cones and hyperplanes. Then, there exists $\eps = \eps(\Lambda,\BC)\in (0,1)$ such that the following is true. Suppose $V$ is a stationary integral varifold in $B^{n+1}_2(0)$ obeying $\|V\|(B^{n+1}_2(0))\leq\Lambda$, and:
	\begin{itemize}
		\item $\dim_\H(\sing_*(V))\leq n-7$;
		\item $V\res \reg_*(V)$ is represented by a two-sided and stable immersion;
		\item $\dist_\H(\spt\|V\|\cap B^{n+1}_1(0),\spt\|\BC\|\cap B^{n+1}_1(0))<\eps$;
	\end{itemize}
	then
	\begin{enumerate}
		\item [\textnormal{(i)}] $\spt\|\BC\| = \bigcup^N_{i=1}P_i$ is the union of hyperplanes $P_i$ for some $N<\infty$;
		\item [\textnormal{(ii)}] for some $q\in \{1,2,\dotsc\}$ we have
		$$V\res B^{n+1}_{1/2}(0) = \sum^q_{i=1}\mathbf{v}(u_i)\res B^{n+1}_{1/2}(0),$$
		where for each $i\in \{1,\dotsc,q\}$ there is some $j_i\in\{1,\dotsc,N\}$ with $u_i:B^{n+1}_{1/2}(0)\cap P_{j_i}\to P_{j_i}^\perp$ is a smooth (single-valued) function solving the minimal surface equation, with the estimate $\|u_i\|_{C^4}\leq CE_{V,\BC}$.
		\item [\textnormal{(iii)}] for each $k\in \{1,\dotsc,N\}$ we have $\{i:j_i=k\}\neq\emptyset$.
	\end{enumerate}
	Furthermore, there is a constant $q_* = q_*(n,\Lambda)\in \{1,2,\dotsc\}$ with $q\leq q_*$. Here, $C = C(n)\in (0,\infty)$.
\end{theorem}

\begin{remark}
	The form of Theorem \ref{thm:HLW} stated above is not identical to the version given in \cite{HLW24}\footnote{In particular, in \cite{HLW24} the relevant results are only stated for $n\geq 3$. However, when $n=2$ (or even $2\leq n\leq 6$) the assumptions of Theorem \ref{thm:HLW} guarantee that $V$ is smoothly immersed in $B^{n+1}_1(0)$, meaning that we have uniform curvature estimates from \cite{SSY75, Bel23}. This is enough to deduce Theorem \ref{thm:HLW} in these cases. However, the $n=2$ case of Theorem \ref{thm:main-compactness} is already covered by Theorem \ref{thm:HLW}, as in this case the assumption in both cases is $\sing_*(V)=\emptyset$, so there is no harm in assuming $n\geq 3$ in our arguments.}, but it follows directly once one has control on the second fundamental form as in \cite{HLW24}. Note again that just because $\spt\|\BC\|$ can be written as the union of hyperplanes this does not imply that $\BC$ can be written as the sum of hyperplanes. A particular consequence of Theorem \ref{thm:HLW} is that if $\spt\|\BC\|$ cannot be written as the union of full hyperplanes, then $\spt\|V\|$ as above cannot be within $\eps$ of $\spt\|\BC\|$ in Hausdorff distance in $B^{n+1}_1(0)$.
\end{remark}

\medskip

Note now the following consequence of Lemma \ref{lemma:gaps} and Theorem \ref{thm:HLW} for Theorem \ref{thm:main-cone}. Assuming the set-up of Theorem \ref{thm:main-cone}, and writing $S(\BC)$ for the spine of $\BC$, if $\dim(S(\BC))\leq n-2$ and conclusion (a) therein fails for sufficiently small $\eps = \eps(\BC)\in (0,1)$, then Theorem \ref{thm:HLW} can be applied away from a neighbourhood of $S(\BC)$\footnote{Indeed, it is applicable to $V$ restricted to the complement of a tubular neighbourhood of $S(\BC)$ since by upper semi-continuity of the density function the density of $V$ must be $<Q$ here, provided $\eps$ is sufficiently small.} to deduce that $\BC$ must be smoothly immersed away from $S(\BC)$. In particular, $\BC$ must be equal to a sum of hyperplanes (rather than just half-hyperplanes). Hence we see that to prove Theorem \ref{thm:main-cone} it suffices to prove conclusion (b) holds when either:
\begin{enumerate}
	\item [(i)] $\dim(S(\BC))\leq n-2$ and $\BC$ is equal to the sum of finitely many hyperplanes;
	\item [(ii)] $\BC$ is a classical cone (i.e.~$\dim(S(\BC))=n-1$);
\end{enumerate}
Indeed, in all other cases one can verify that conclusion (a) of Theorem \ref{thm:main-cone} must hold for suitable $\eps = \eps(\BC)\in (0,1)$. Put differently, cones obeying (i) or (ii) are the only cones $\BC$ of the form in Theorem \ref{thm:main-cone} which are in the class $\V_Q$, and we are thus saying if $\BC\not\in \V_Q$, then Theorem \ref{thm:main-cone}(a) holds for suitable $\eps = \eps(\BC)$, whilst if $\BC\in \V_Q$ then $\BC$ must take one of the forms in (i), (ii), and we will prove an $\eps$-regularity theorem in this case.

Notice in particular that when $Q$ is a half-integer rather than an integer, the only case of Theorem \ref{thm:main-cone} which needs to be verified is when $\dim(S(\BC))=n-1$ (and this case turns out to be much simpler, cf.~Remark \ref{remark:cone-2}(3)).

The above observation means that we can reduce Theorem \ref{thm:main-cone} to proving an $\eps$-regularity theorem. Here, we will use the following notation (cf.~\cite[Theorem C]{MW24}). Write $\BC$ for a stationary cone with $\Theta_{\BC}(0) = Q$ which takes one of the two forms above, i.e., up to performing an ambient rotation to ensure $S(\BC)\subseteq\{0\}^2\times\R^{n-1}$,
\begin{enumerate}
	\item [(I)] $\BC = \sum^N_{i=1}q_i|P_i|$, where $N\geq 2$, $q_i\in \{1,2,\dotsc\}$, and the $P_i$ are distinct hyperplanes in $\R^{n+1}$ passing through the origin such that $\bigcap_{i=1}^N P_i \equiv S(\BC)$ is a subspace of dimension $\in \{0,1,\dotsc,n-2\}$;
	\item [(II)] $\BC = \sum^N_{i=1}q_i|H_i|$, where $N\geq 3$, $q_i\in\{1,2,\dotsc\}$, and the $H_i$ are distinct half-hyperplanes with $\del H_i = \{0\}^2\times\R^{n-1}$ for each $i=1,\dotsc,N$.
\end{enumerate}
Note that our assumptions imply in each case that $q_i<Q$ for all $i$; this is crucial for being able to apply our inductive assumptions (namely Theorem \ref{thm:HLW}), via Lemma \ref{lemma:gaps} and upper semi-continuity of density, away from $S(\BC)$. In (II), we may write $H_i = R_i\times\R^{n-1}$, where $R_i = \{t\mathbf{w}_i:t>0\}$ for (distinct) unit vectors $\mathbf{w}_1,\dotsc,\mathbf{w}_N$ in $\R^2$. In this case also write
$$\sigma_0:=\max\{\mathbf{w}_i\cdot\mathbf{w}_k: i\neq k\}$$
and let $N(H_i)$ denote the conical neighbourhood of $H_i$ defined by
$$N(H_i) := \left\{(x,y)\in \R^2\times\R^{n-1}: x\cdot\mathbf{w}_i>|x|\sqrt{\frac{1+\sigma_0}{2}}\right\}.$$
Write also $\widetilde{H}_i$ for the hyperplane containing $H_i$. In general for a cone $\BC$ we write $r_{\BC}(x):=\dist(x,S(\BC))$. We also note that case (II) can be divided into two subcases:
\begin{enumerate}
	\item [(II-a)] $\BC$ is as in (II), yet $\BC$ is not the sum of hyperplanes;
	\item [(II-b)] $\BC$ is as in (II) and is the sum of hyperplanes.
\end{enumerate}

The reduction of Theorem \ref{thm:main-cone} is the following. This is the main $\eps$-regularity theorem of the paper.

\begin{thmx}[Conical $\eps$-Regularity Theorem: Reduced Statement]\label{thm:main-reg}
	Fix $Q\in \{\frac{3}{2},2,\frac{5}{2},3,\dotsc\}$. Let $\BC^{(0)}$ be a cone taking one of the two types described in \textnormal{(I)} or \textnormal{(II)} above. Then, there is a constant $\eps = \eps(\BC^{(0)})\in (0,1)$ such that the following holds. Suppose $V\in\mathcal{V}_Q$ satisfies:
	\begin{enumerate}
		\item [\textnormal{(i)}] $(\w_n 2^n)^{-1}\|V\|(B^{n+1}_2(0))<Q+\frac{1}{4}$;
		\item [\textnormal{(ii)}] $Q-\frac{1}{4}\leq \w_n^{-1}\|V\|(B^{n+1}_1(0))<Q+\frac{1}{4}$;
		\item [\textnormal{(iii)}] Depending on the form of $\BC^{(0)}$, assume:
		\begin{enumerate}
			\item [\textnormal{(iiia)}] If $\BC^{(0)}$ is as in \textnormal{(I)}, assume that
			$$\dist_\H(\spt\|V\|\cap B^{n+1}_1(0),\spt\|\BC^{(0)}\|\cap B^{n+1}_1(0))<\eps;$$
			\item [\textnormal{(iiib)}] If $\BC^{(0)}$ is as in \textnormal{(II)}, assume that
			$$d(V\res B_1^{n+1}(0), \BC^{(0)}\res B_1^{n+1}(0)) < \eps;$$
			or, equivalently, that $E_{V,\BC}<\eps$ and, for each $i\in \{1,\dotsc,N\}$,
			$$\|V\|\left(\big(B^{n+1}_{1/2}(0)\setminus\{r_{\BC^{(0)}}<1/8\}\big)\cap N(H_i)\right) \geq \left(q_i-\tfrac{1}{4}\right)\H^n\left(\big(B^{n+1}_{1/2}(0)\setminus\{r_{\BC^{(0)}}<1/8\}\big)\cap H_i\right).$$
		\end{enumerate}
	\end{enumerate}
	Then, depending on the form of $\BC^{(0)}$, the corresponding conclusion below holds.
	
	\textnormal{\textbf{\underline{Case A:}}} $\BC^{(0)}$ is as in \textnormal{(I)}. Then,
	$$V\res B^{n+1}_{1/2}(0) = \sum^Q_{i=1}\mathbf{v}(u_i)\res B^{n+1}_{1/2}(0),$$
	where for each $i\in \{1,\dotsc,Q\}$ there is $j_i\in \{1,\dotsc,N\}$ with $u_i:B^{n+1}_{1/2}(0)\cap P_{j_i}\to P_{j_i}^\perp$ a smooth single-valued function solving the minimal surface equation which obeys $\|u_i\|_{C^4}\leq CE_{V,\BC^{(0)}}$, where $C = C(n)\in (0,\infty)$. Furthermore:
	\begin{enumerate}
		\item [\textnormal{(A1)}] For each $k\in \{1,\dotsc,N\}$ we have $\{i:j_i = k\}\neq\emptyset$;
		\item [\textnormal{(A2)}] $\bigcap^Q_{i=1}\graph(u_i)$ is contained within an $(n-2)$-dimensional smooth submanifold.
	\end{enumerate}
	\textnormal{\textbf{\underline{Case B:}}} $\BC^{(0)}$ is as in \textnormal{(II-a)}. Then, for each $i\in \{1,\dotsc,N\}$ there is a function
	$$\gamma_i\in C^{1,\alpha}\left(\overline{S(\BC^{(0)})\cap B^{n+1}_{1/2}(0)};\,\widetilde{H}_i\cap \{r_{\BC^{(0)}}<\tfrac{1}{16}\}\right)$$
	and functions $u_{i,j}:\Omega_i\to \widetilde{H}_i^\perp$ for $j=1,\dotsc,q_i$, where $\Omega_i$ is the connected component of $\widetilde{H}_i\cap B^{n+1}_{1/2}(0)\setminus\{x+\gamma_i(x):x\in S(\BC^{(0)})\cap B^{n+1}_{1/2}(0)\}$ with $\Omega_i\supset (H_i\setminus \{r_{\BC^{(0)}}<\tfrac{1}{16}\})\cap B^{n+1}_{1/2}(0)$, such that:
	\begin{enumerate}
		\item [\textnormal{(B1)}] $u_{i,j}\in C^{1,\alpha}(\overline{\Omega}_i;\widetilde{H}^\perp_i)$ and $u_{i,j}\cdot \nu_i$ solves the minimal surface equation on $\Omega_i$, where $\nu_i$ is (a choice of) the unit normal to $\widetilde{H}_i$;
		\item [\textnormal{(B2)}] Writing $\widetilde{u}_{i,j}(x):= x + u_{i,j}(x)$ for $x\in\overline{\Omega}_i$, we have for all $y\in \overline{S(\BC^{(0)})\cap B^{n+1}_{1/2}(0)}$, 
		$$\widetilde{u}_{i_1,j_1}(y+\gamma_{i_1}(y)) = \widetilde{u}_{i_2,j_2}(y+\gamma_{i_2}(y))$$
		for all values of indices $i_1,i_2\in \{1,\dotsc,N\}$ and $j_k\in \{1,\dotsc,q_{i_k}\}$ for $k=1,2$;
		\item [\textnormal{(B3)}] $V\res B^{n+1}_{1/2}(0) = \sum^N_{i=1}\sum^{q_i}_{j=1}\mathbf{v}(u_{i,j})\res B^{n+1}_{1/2}(0)$;
		\item [\textnormal{(B4)}] For each $i\in \{1,\dotsc,N\}$ and $j\in \{1,\dotsc,q_i\}$ we have $|u_{i,j}|_{1,\alpha;\Omega_i}\leq CE_{V,\BC^{(0)}}$, where $C = C(\BC^{(0)})\in (0,\infty)$ and $\alpha = \alpha(\BC^{(0)})\in (0,1)$, and
		$$\{z:\Theta_V(z)\geq Q\}\cap B^{n+1}_{1/2}(0) = \{z:\Theta_V(z) = Q\}\cap B^{n+1}_{1/2}(0) = \widetilde{u}_{i,j}(B^{n+1}_{1/2}(0)\cap\del\Omega_i).$$
	\end{enumerate}
	\textnormal{\textbf{\underline{Case C:}}} $\BC^{(0)}$ is as in \textnormal{(II-b)}. Then, the conclusions of Case B hold with the following modifications:
	\begin{enumerate}
		\item [\textnormal{(C1)}] \textnormal{(B2)} is no longer valid and is removed;
		\item [\textnormal{(C2)}] The final claim in \textnormal{(B4)} is replaced with: for each $i\in \{1,\dotsc,N\}$ and $j\in \{1,\dotsc,q_i\}$,
		$$\{z:\Theta_V(z)\geq Q\}\cap B^{n+1}_{1/2}(0) = \{z:\Theta_V(z) = Q\}\cap B^{n+1}_{1/2}(0) \subset \widetilde{u}_{i,j}(B^{n+1}_{1/2}(0)\cap\del\Omega_i).$$
	\end{enumerate}
\end{thmx}

\begin{remark}\label{remark:cone-2}
	The reader should note the following points regarding Theorem \ref{thm:main-reg}:
	\begin{enumerate}
		\item [(1)] Assumption (iiib) of course implies assumption (iiia), but not vice versa (as (iiib) ensures that $V$ is close to $\BC^{(0)}$ with the correct multiplicities on the half-hyperplanes in $\BC^{(0)}$). The reason we state Theorem \ref{thm:main-reg} with this stronger hypothesis when $\BC^{(0)}$ is of the form in (II) is again due to the possibility that it is possible for $\spt\|\BC^{(0)}\|$ to be written as a union of hyperplanes without $\BC^{(0)}$ being written as a sum of hyperplanes.
		\item [(2)] The conclusion in Case C is a mixture of those in Case A and Case B. Qualitatively, the conclusions in Case C say that the density $Q$ (or equivalently $\geq Q$) set is only \emph{contained} within a $(n-1)$-dimensional $C^{1,\alpha}$ submanifold which is graphical over part of $S(\BC^{(0)})$, and at each density $Q$ singular point there is a unique tangent cone equal to a cone which could be of either form described in (I) or (II). In Case B, the density $Q$ set is \emph{equal} to a $(n-1)$-dimensional $C^{1,\alpha}$ submanifold.
		\item [(3)] When $Q$ is a half-integer, then $\BC^{(0)}$ must be of the form in (II-a). In this case, the proof of Theorem \ref{thm:main-reg} is significantly simpler, because so-called ``density gaps'' cannot occur (this follows in an identical manner to that seen in, for instance, \cite[Lemma 16.5(a)]{Wic14}, which in turn is an adaptation of \cite[(6.12)]{SS81}, or alternatively by applying Lemma \ref{lemma:gaps} and Theorem \ref{thm:HLW}). Once density gaps have been ruled out, one can reduce our proof to that seen in \cite[Section 16]{Wic14} (cf.~\cite[Theorem C]{MW24}), using Theorem \ref{thm:bellettini} in place of \cite[Theorem 15.2]{Wic14} therein. In fact, the same is true even when $Q$ is an integer if $\BC^{(0)}$ is as in (II-a), and thus this establishes Case B of Theorem \ref{thm:main-reg}, leaving only Case A and Case C. However, to prove Case A and Case C (which we note forces $Q$ to be an integer), one needs to analyse Case B when $Q$ is an integer as well (the intuitive reason being that a cone in the form (II-a) can arise as a small perturbation of a cone in the form (II-b) and thus such cones can arise in excess decay arguments). The crucial difference with the proof in Case A and Case C is that density gaps \emph{can occur}, and dealing with these requires a significantly more refined argument. Nonetheless, we now only need to prove Theorem \ref{thm:main-reg} when $Q$ is an integer.
		\item [(4)] The hardest case of Theorem \ref{thm:main-reg} is Case C, i.e.~when the cone is of the form in (II-b). This is because, if one fixes a cone of the form in (I) or (II) at looks at all nearby cones which are of the form in either (I) or (II), if the initial cone is of the form in (II-b) then all other forms of cone are nearby to it. This is unlike when the cone is of the form in (I), as nearby cones must also be of the form in (I). The same is also true for cones of the form in (II-a). This will be made precise in Section \ref{sec:set-up}.
		\item [(5)] In the context of Theorem \ref{thm:main-compactness} (namely, Theorem \ref{thm:main-reg-basic}), as the associated $V$ has no non-immersed classical singularities at all, the cone $\BC^{(0)}$ can never be as in (II-a), so Case B never happens. This drastically reduces the complexity of Case C, giving that seen in Theorem \ref{thm:main-reg-basic}. If one only cares about proving Theorem \ref{thm:main-compactness}, one can then additionally assume that varifolds in $\V_Q$ have no non-immersed classical singularities of density $Q$ as well, leading to this simpler version. Nonetheless, one would still need to analyse the behaviour near cones of the form in (II-a) within the proof for cones of the form (II-b) (one one would need to prove a ``fine minimal distance theorem'', like in \cite[Lemma 14.1]{Wic14}).
	\end{enumerate}
\end{remark}
Thus, by Remark \ref{remark:cone-2}(3) we now only need to consider the case when $Q$ is an integer. From now on we therefore fix $Q\in \{2,3,\dotsc\}$. We note that in fact the $Q=2$ case of Theorem \ref{thm:main-reg} follows from \cite[Theorem C]{MW24}.

\subsection{Notation and set-up for Theorem \ref{thm:main-reg}}\label{sec:set-up}

Recall that we have now fixed $Q\in \{2,3,\dotsc\}$ and we want to prove Theorem \ref{thm:main-reg}. We start by setting up the relevant notation.

We divide the different types of cones we will be interested in into different classes. These depend on two parameters:
\begin{itemize}
	\item $\sfrak\in \{0,1,2,\dotsc,n-1\}$, the dimension of the spine of the cone;
	\item $\pfrak\in \{2,3,\dotsc,Q\}$, the number of distinct (half-)hyperplanes in the cone.
\end{itemize}
Here and throughout, the words ``hyperplane'' and ``plane'' will always refer to a subspace (of dimension $n$) rather than an affine plane, and similarly ``half-hyperplane'' and ``half-plane'' refer to the intersection of a hyperplane with a transverse half-space of $\R^{n+1}$ whose boundary contains $0$.

Now we define classes of cones as follows. To start with, define
\begin{align*}
	\mathcal{P}_{\sfrak}(\pfrak):= & \left\{\BC = \sum^{\pfrak}_{i=1}q_i|P^{(i)}|: P^{(i)}\text{ are distinct hyperplanes, }\dim(S(\BC))=\sfrak,\right.\\
	& \hspace{20em} \left.\text{and } q_i\in \Z_{\geq 1}\text{ obey }\sum^\pfrak_{i=1}q_i=Q\right\}.
\end{align*}
For a cone $\BC$ we write $S(\BC)$ for its \emph{spine}, referring to the maximal subspace under which $\BC$ is translation invariant. Thus for $\BC = \sum^{\pfrak}_{i=1}q_i|P^{(i)}|\in \mathcal{P}_{\sfrak}(\pfrak)$ we have $S(\BC) = \bigcap_i P^{(i)}$.
\begin{remark}
Since we are in codimension one, if $\sfrak+\pfrak\leq n$ then $\mathcal{P}_{\sfrak}(\pfrak) = \emptyset$. For instance, it is impossible for $\pfrak=2$ hyperplanes to intersect only along a subspace of dimension $\sfrak \leq n-2$.
\end{remark}
We are also interested in classical cones, i.e.~cones which are sums of half-hyperplanes, provided that all their boundaries coincide (in particular, the spine dimension of the resulting cone is $n-1$). For this, define:
\begin{align*}
    \H_{n-1}(\pfrak):= & \left\{\sum^{\pfrak}_{i=1}q_i|H^{(i)}|: H^{(i)}\text{ are distinct half-hyperplanes with }\right.\\
    & \hspace{5em}\left. \del H^{(i)} = \del H^{(j)}\text{ for all $i\neq j$, and }q_i\in\Z_{\geq 1}\text{ obey }\sum^{\pfrak}_{i=1}q_i=2Q\right\}\mathbin{\big\backslash} \mathcal{P}_{n-1}(\pfrak).
\end{align*}
All of these cones have origin density $Q$. We stress that the multiplicity on each (half-)hyperplane is important to keep track of, even though in several definitions it is only the support of the cone which ultimately matters (such as in notions of excess).

We also group certain families of cones into more convenient families:
\begin{itemize}
	\item $\mathcal{P}_\sfrak:= \bigcup^Q_{\pfrak=2}\mathcal{P}_{\sfrak}(\pfrak)$ is the set of sums of $Q$ planes with spine dimension exactly $\sfrak$;
	\item $\mathcal{P}_{\leq n-2}:= \bigcup^{n-2}_{\sfrak=0}\mathcal{P}_{\sfrak}$ is the set of sums of $Q$ planes with spine dimension at most $n-2$;
	\item $\mathcal{P}:= \bigcup^{n-1}_{\sfrak=0}\mathcal{P}_{\sfrak}$ is the set of all sums of $Q$ planes with at least two being distinct;
	\item $\H_{n-1}:=\cup_{\pfrak=3}^{2Q}\H_{n-1}(\pfrak)$ is the set of all possible sums of $2Q$ half-hyperplanes (which are not sums of $Q$ full hyperplanes) with a common boundary;
	 \item For $\hfrak\in \{3,4,\dotsc,2Q\}$ set
    $$\CC_{n-1}(\hfrak) := \begin{cases}
        \H_{n-1}(\hfrak) & \text{if }\hfrak\text{ is odd;}\\
        \H_{n-1}(\hfrak)\cup \mathcal{P}_{n-1}(\tfrac{\hfrak}{2}) & \text{if $\hfrak$ is even;}
    \end{cases}$$
     \item $\CC_{n-1}:= \mathcal{H}_{n-1}\cup\mathcal{P}_{n-1}$ is the set of all cones of interest which have spine dimension $n-1$;
    \item $\CC := \mathcal{P}\cup\mathcal{H}_{n-1}$ is the set of all cones of interest in this work.
\end{itemize}
These families of cones fall into a hierarchy which will be important for us.
\begin{defn}
	Suppose $\mathfrak{C}_1$ and $\mathfrak{C}_2$ are both one of the collections of cones from the list of $\mathcal{P}_{\sfrak}(\pfrak)$ and $\H_{n-1}(\hfrak)$. We say $\mathfrak{C}_2$ is \emph{finer} than $\mathfrak{C}_1$ if $\overline{\mathfrak{C}_1}\subsetneq \overline{\mathfrak{C}_2}$. In this case we also say $\mathfrak{C}_1$ is \emph{coarser} than $\mathfrak{C}_2$.
\end{defn}

Here, the closure is taken in the varifold topology, which here just means keeping track of the multiplicity of the (half-)hyperplanes. Intuitively, finer cones are ones with (strictly) more degrees of freedom when compared to another class of cones. Also, small perturbations of a cone cannot result in a coarser cone.
\begin{defn}
	For cones $\BC_1,\BC_2\in \CC$, we say that $\BC_2$ is \emph{finer} than $\BC_1$, and write $\BC_1\prec\BC_2$, if the collection of cones which $\BC_1$ belongs to is finer than that of $\BC_2$. We also say $\BC_1$ is \emph{coarser} than $\BC_2$ in this case.
	
	We also write $\BC_1\preceq \BC_2$ if either $\BC_1\prec \BC_2$ or if $\BC_1$ and $\BC_2$ belong to the same family of cones. The notations $\succ$ and $\succeq$ then have the obvious meanings.
\end{defn}

The notion of being finer induces a strict \emph{partial order} on the set $\CC$, which can be visualised using the diagram shown in Figure \ref{fig:0}.

\begin{figure}[h]
\centering
\begin{tikzpicture}
        \node at (3.5,1.6) {$\cdot$};
        \node at (3.6,1.7) {$\cdot$};
        \node at (3.7,1.8) {$\cdot$};
        \node at (3.8,1.9) {$\cdot$};
        \node at (3.9,2.0) {$\cdot$};
        \node at (4.0,2.1) {$\cdot$};

        \node at (1.3,0) {$\mathcal{P}_{\max\{0,n+1-Q\}}(Q)$};
        \draw [<-] (2.9,0) -- (3.5,0);
        \node at (3.85,0) {$\cdots$};
        \draw [<-] (4.2,0) -- (4.7,0);
        \node at (5.55,0) {$\mathcal{P}_{n-2}(Q)$};
        \draw [<-] (6.4,0) -- (6.9,0);
        \node at (7.75,0) {$\mathcal{P}_{n-1}(Q)$};
        \draw [->] (8.6,0) -- (10,0);
        \node at (11,0) {$\H_{n-1}(2Q)$};

        \draw [<-] (11,0.35) -- (11,0.85);
        \node at (11.35,1.2) {$\H_{n-1}(2Q-1)$};
        \draw [<-] (11,1.55) -- (11,2.05);
        \node at (11,2.45) {$\vdots$};
        \draw [<-] (11,2.7) -- (11,3.2);
        \node at (11,3.55) {$\H_{n-1}(6)$};
        \draw [<-] (11,3.9) -- (11,4.4);
        \node at (11,4.75) {$\H_{n-1}(5)$};
        \draw [<-] (11,5.1) -- (11,5.6);
        \node at (11,5.95) {$\H_{n-1}(4)$};
        \draw [<-] (11,6.2) -- (11,6.7);
        \node at (11,7.05) {$\H_{n-1}(3)$};

        \node at (5.55,3.55) {$\mathcal{P}_{n-2}(3)$};
        \draw [<-] (6.4,3.55) -- (6.9,3.55);
        \node at (7.75,3.55) {$\mathcal{P}_{n-1}(3)$};
        \draw [->] (8.6,3.55) -- (10,3.55);

        \node at (7.75,5.95) {$\mathcal{P}_{n-1}(2)$};
        \draw [->] (8.6,5.95) -- (10,5.95);
        

        \draw [->] (7.75, 5.6) -- (7.75, 3.9);

        \draw [->] (7.75, 3.2) -- (7.75, 2.1);
        \draw [->] (5.55, 3.2) -- (5.55, 2.1);

        \node at (7.75, 1.85) {$\vdots$};
        \node at (5.55, 1.85) {$\vdots$};

        \draw [->] (7.75, 1.4) -- (7.75, 0.3);
        \draw [->] (5.55, 1.4) -- (5.55, 0.3);
\end{tikzpicture}
\caption{\footnotesize Visualisation of which families of cones are finer than one another. An arrow indicates a finer family from a given family. Notice that there are two maximal elements, namely $\mathcal{P}_{\max\{0,n+1-Q\}}(Q)$ and $\H_{n-1}(2Q)$.}
\label{fig:0}
\end{figure}
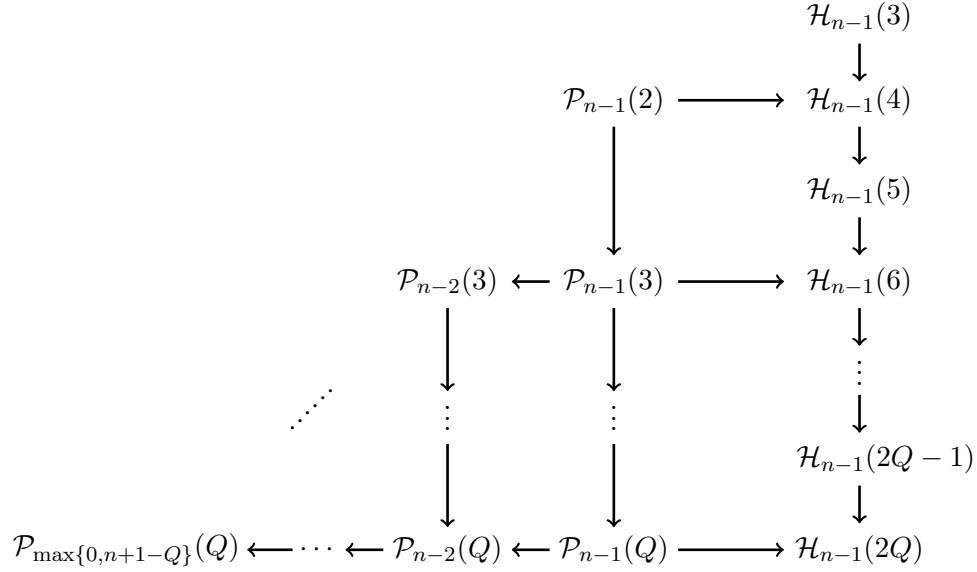

To elaborate on Figure \ref{fig:0} in relation to the present and other related works, the key point is that when one tries to prove regularity via a blow-up and excess decay argument, one hopes to prove an excess decay lemma in which the excess decays with respect to a new cone which is close to the one originally started at, but a suitable perturbation of it. One must therefore understand all the possible relevant nearby perturbations of a given cone (which is more challenging when the cone has multiplicity). Here, the relevant perturbations are those where:
\begin{itemize}
	\item if the cone has spine dimension $n-1$ (so is in $\CC_{n-1}$), then one is allowed to perturb the individual half-hyperplanes independently of one another, including those which coincide;
	\item if the cone has spine dimension $\leq n-2$ (so is in $\P_{\leq n-2}$) then one is allowed to perturb the individual hyperplanes independently of one another, including those which coincide.
\end{itemize}
In particular, one is not guaranteed to remain in the same family of cones as one starts in during the excess decay, as one can perturb to a finer cone, which corresponds to moving along the arrows in the (strict) partial order $\prec$ as shown in Figure \ref{fig:0}. Since in the present setting the process must eventually terminate (as $\mathcal{P}_{\max\{0,n+1-Q\}}(Q)$ and $\H_{n-1}(2Q)$ are the only maximal elements under $\prec$) one can ensure that the excess decay argument closes up. The hardest classes of cones to deal with are the minimal elements in the partial ordering, namely $\mathcal{P}_{n-1}(2)$ or $\H_{n-1}(3)$, as then there are the largest number of cases to analyse.

The reader may find it useful to cast other works in this light, viewing them as instances where smaller version of the diagram in Figure \ref{fig:0} needed to be analysed. For instance:
\begin{itemize}
    \item One could view the work of Allard \cite{All72} as only requiring a diagram with a single class of cones, namely planes with multiplicity one, i.e.~the diagram would simply be
    $$\mathcal{P}_n(1).$$
    Similarly, the $\eps$-regularity theorem of Simon \cite{Sim93} for the triple junction only requires a diagram with a single class of cones, namely cones formed of 3 half-planes with multiplicity one, i.e.
    $$\H_{n-1}(3).$$
    \item The work of Wickramasekera \cite{Wic14}, which first introduced the idea of using \emph{fine blow-ups} to understand what happens when multiplicity causes finer cones to arise in blow-up procedures, needed to analyse the chain
    $$\mathcal{P}_n(1) \to \H_{n-1}(4)\cup\mathcal{P}_{n-1}(2) \to \mathcal{H}_{n-1}(5)\to\cdots\to\mathcal{H}_{n-1}(2Q)\cup\mathcal{P}_{n-1}(Q)$$
    (in the setting of \cite{Wic14} it is more natural to group $\mathcal{H}_{n-1}(2\pfrak)$ and $\mathcal{P}_{n-1}(\pfrak)$ together). Similar chains to this were also needed in \cite{MW24, Min21a}. Notice that in these situations $\prec$ is actually a total order.
    \item The work of Becker--Kahn \cite{BK17} was the first work which investigated, in the context of $2$-valued stationary Lipschitz graphs (of arbitrary dimension and codimension), regularity and excess decay in a situation when the partial order took of the form
    $$\mathcal{P}_0(2)\leftarrow\cdots\leftarrow\mathcal{P}_{n-3}(2)\leftarrow\mathcal{P}_{n-2}(2)\leftarrow \mathcal{P}_{n-1}(2) \to \H_{n-1}(4).$$
    Indeed, in \cite{BK17} whenever the cone is formed of two planes, a perturbation of the cone could result in two planes intersecting along a smaller spine or, if the spine dimension is $n-1$, a cone formed of four half-planes with a common boundary (both of which are finer cones). Notice that a cone with spine dimension $n-1$ which is formed of $4$ distinct half-planes, but which is not two full planes, is far from two planes intersecting along an axis of spine dimension $\leq n-2$ (as one cannot be perturbed into the other). This is the cause for these classes of cones to be unrelated under $\prec$.
    \item The recent work of Becker--Kahn, the first author, and Wickramasekera \cite{BKMW25} (which is also in arbitrary dimension and codimension, and does not use stability or any a priori graph structure) was the first work to carry out this analysis where the multiplicity could split, the spine dimension of the cone could decrease, and two planes could perturb to $4$ half-planes, leading to the associated partial order diagram being
    \begin{align*}
        &\ \mathcal{P}_{n}(1)\\
        &\ \ \ \downarrow\\
        \mathcal{P}_0(2)\leftarrow\cdots\leftarrow\mathcal{P}_{n-3}(2)\leftarrow\mathcal{P}_{n-2}(2)\leftarrow\ & \mathcal{P}_{n-1}(2) \to \H_{n-1}(4).
    \end{align*}
\end{itemize}
The key difficulty and technical contribution in the present work can therefore be viewed as establishing and working with an excess decay lemma when the diagram is the more complicated one shown in Figure \ref{fig:0}. Here, the multiplicity can split, the spine dimension can decrease, and hyperplanes can split into half-hyperplanes.\footnote{As a comment to the reader familiar with this body of work: the main technical difficulty here in comparison to other works comes in the classification of homogeneous degree one blow-ups. Here, it seems that one needs to inductively establish fine $\eps$-regularity theorems in order to achieve this, which is reminiscent of the proofs of the main regularity theorems in \cite{MW24, Min21a, BKMW25}. The presence of multiplicity prevents one from arguing as in \cite{BK17} or \cite[Theorem D]{BKMW25}, and the presence of density gaps prevents one from arguing as in \cite[Section 16]{Wic14}.}

We use various notions of \emph{excess} throughout this work. We use the following shorthand notations for varifolds $W,W_1,W_2$ in $\R^{n+1}$:
\begin{itemize}
    \item $\dist(x,\spt\|W\|)\equiv \dist(x,W)$ for $x\in\R^{n+1}$;
    \item $\dist_\H(\spt\|W_1\|\cap B,\spt\|W_2\|\cap B) \equiv \dist_\H(W_1\cap B, W_2\cap B)$ for $B\subset \R^{n+1}$.
\end{itemize}

For the rest of the paper, we fix once and for all a cone $\BC^{(0)}\in \CC$. We will only need to consider cones $\BC\succeq\BC^{(0)}$.

\begin{defn}
    For $V\in\V_Q$ and $\BC\in \CC$, the (\emph{one-sided}) \emph{height excess} of $V$ relative to $\BC$ is:
    $$E_{V,\BC}:=\left(\int_{B_1^{n+1}(0)}\dist^2(x,\BC)\, \ext V(x)\right)^{1/2}.$$
    The \emph{fine excess} of $V$ relative to $\BC$ is:
    $$F_{V,\BC}:= \left(\int_{B_1^{n+1}(0)}\dist^2(x,\BC)\, \ext\|V\|(x) + \int_{B_{1/2}^{n+1}(0)\setminus B_{1/16}(S(\BC^{(0)}))}\dist^2(x,V)\, \ext\|\BC\|(x)\right)^{1/2}.$$
\end{defn}

\begin{remark}
Simple contradiction arguments show that for any $\eps>0$ and $\sigma\in (0,1)$, there exists $\delta = \delta(\eps,\sigma,n,Q)\in (0,1)$ such that:
\begin{enumerate}
    \item [(a)] If $F_{V,\BC}<\delta$ then $\dist_\H(V\cap B_\sigma(0),\BC\cap B_\sigma(0))<\eps$;
    \item [(b)] If $\dist_\H(V\cap B_{1+\sigma}(0),\BC\cap B_{1+\sigma}(0))<\delta$ then $F_{V,\BC}<\eps$.
\end{enumerate}
It is useful to note that therefore smallness of $F_{V,\BC}$ does not imply that $V$ and $\BC$ are close as varifolds in $B_1$, since $\BC$ might not have the correct multiplicities on its (half-)hyperplanes.
\end{remark}

\medskip

Often we will restrict ourselves to cones whose spines are related in the following manner:
\begin{defn}
	We say cones $\BC_1,\BC_2\in \CC$ are \emph{aligned} if either $S(\BC_1)\subseteq S(\BC_2)$ or $S(\BC_2)\subseteq S(\BC_1)$.
\end{defn}
\begin{remark}
If $\BC_1\preceq \BC_2$ are aligned then one must have $S(\BC_2)\subseteq S(\BC_1)$.
\end{remark}

\subsection{Overview of the proof}

To prove Theorem \ref{thm:main-reg} we use blow-up and excess decay argument. This requires two delicate inductive procedures based on the partial order $\preceq$ for cones $\BC\in \CC$. These are:
\begin{enumerate}
	\item [(1)] First, we prove key $L^2$ estimates for $V$ relative to \emph{finer cones} $\BC$ to $\BC^{(0)}$ by working inductively \emph{down} the partial ordering $\preceq$. Thus, the base case is when $\BC$ is coarsest, namely in the same class as $\BC^{(0)}$ itself.
	\item [(2)] Second, we prove \emph{fine} $\eps$-regularity theorems by working inductively \emph{up} the partial ordering $\preceq$. Thus, the base case is when $\BC$ is finest, namely in $\mathcal{P}_{\max\{0,n+1-Q\}}(Q)$ or $\H_{n-1}(2Q)$.
\end{enumerate}
Intuitively, it is easier to represent $V\in \V_Q$ nearby $\BC^{(0)}$ by graphs over coarser cones, with the relevant estimates, as it is easier to identify the ``correct'' plane to graph over. This is why the induction for the $L^2$ estimates works down the partial ordering. On the other hand, when proving $\eps$-regularity theorems the easiest case is when the cone is finest, as then perturbations of the cone have to remain in the same class (as they cannot become coarser). When proving an $\eps$-regularity theorem for a certain class of cones it is therefore useful to have already proved $\eps$-regularity theorems near finer cones, as if one were to perturb to a finer cone in the excess decay argument, intuitively one has already proved the theorem in that case. (In fact, there is also a third induction happening, as we will also prove Theorem \ref{thm:main-reg} inductively over the \emph{level} of $\BC^{(0)}$ -- see below.)

It is worth pointing out that a key part of (2) above is proving a $C^{1,\alpha}$ regularity estimate for blow-ups. When the blow-up arises from blowing-up a sequence of varifolds in $\V_Q$ relative to cones in $\P_{\leq n-2}$ or $\H_{n-1}$, this step is much easier to prove. Indeed, in the case of $\P_{\leq n-2}$ this follows from the fact that a continuous function on a plane which is harmonic away from a subspace of codimension $\geq 2$ must be harmonic on the plane. In the case of $\H_{n-1}$, this is reduced to $C^{1,\alpha}$ boundary regularity of harmonic functions in an analogous manner to that seen in \cite[Section 16]{Wic14} (cf.~\cite{Sim93}). In particular, less inductive information is needed for (2) in these cases. However, in both these cases, it seems one cannot avoid working inductively to prove the $L^2$ estimates in (1), as one needs to show $L^2$ convergence along blow-up sequences (morally, this is to rule out both the fundamental solution of the Laplacian occurring in the blow-up, as well as for proving the excess decay statement). However, for $\P_{n-1}$ full inductive information is needed in both (1) and (2).

Due to $\preceq$ being only a partial ordering rather than a total ordering, it will be notationally convenient to collect all the classes $\mathcal{P}_{\sfrak}(\pfrak)$ and $\H_{n-1}(\hfrak)$ into a single ordered list
$$\mathfrak{C}_1,\mathfrak{C}_2,\dotsc,\mathfrak{C}_{N_0(n,Q)}$$
such that every class of cones appears exactly once and that coarser classes of cones always appear before finer ones, i.e.~if $\mathfrak{C}_i\prec\mathfrak{C}_j$ then $i<j$. This can clearly be done (cf.~Figure \ref{fig:0}), although the way of doing this is not unique (e.g.~$\H_{n-1}(3)$ and $\mathcal{P}_{n-1}(2)$ are both valid choices for $\mathfrak{C}_1$ and $\mathfrak{C}_2$ as they are incomparable and no class of cones is finer). Note that it won't be the case that if $i<j$ then $\mathfrak{C}_i\prec\mathfrak{C}_j$. For $\BC\in \CC$, we will refer to the index $i$ for which $\BC\in \mathfrak{C}_i$ as the \emph{level} of $\BC$, with $\mathfrak{C}_i$ being the \emph{class} which $\BC$ belongs to.

The remainder of the paper is split into two parts. In Part \ref{part:estimates}, we give the inductive proof of the aforementioned $L^2$ estimates. In Section \ref{sec:main-estimates-statements} we state the main technical theorem (Theorem \ref{thm:main-L^2-estimate}) which concerns both graphicality of $V\in \V_Q$ over cones $\BC$ nearby to $\BC^{(0)}$ together with the main $L^2$ estimates for $V$. These should be compared to the estimates established in the original work of Simon \cite[Theorem 3.1]{Sim93} and their many manifestations since, and they resemble a combination of those in \cite[Theorem 3.2 \& Corollary 3.3]{BK17} and \cite[Theorem 10.1 \& Theorem 16.2]{Wic14}; to-date no single situation captures that of the present work, as we must deal with non-flat cones which have both multiplicity and in which the spine dimension can decrease. As mentioned above, our proof of these estimates will be by induction on the level $i_0$ of the cone $\BC$, and thus we will be able to assume the result holds for all cones in $\bigcup_{i<i_0}\mathfrak{C}_{i}$.

In order to prove the main $L^2$ estimates for finer cones, we will utilise blow-ups off coarser cones. The ability to do this comes from the inductive validity of the main $L^2$ estimates. In order to perform the blow-up procedure, we need a mechanism for choosing coarser cones which are good representatives for the infimum of the (fine) excess relative to coarser cones. This is done in Section \ref{sec:replacement}, in which we prove a \emph{replacement lemma} which allows us to choose such good representatives.

In Section \ref{sec:blow-up}, we then show how the validity of Theorem \ref{thm:main-L^2-estimate} for $V\in \V_Q$ and $\BC\in \mathfrak{C}_i$ allows us to construct blow-ups. We also prove some of their initial properties. More specifically, given the validity of Theorem \ref{thm:main-L^2-estimate} for $V\in \V_Q$ and $\BC\in \mathfrak{C}_i$, if $(V_j)_j\subset \V_Q$ and $(\BC_j)_j\subset \mathfrak{C}_i$ are sequences which obey the assumptions of Theorem \ref{thm:main-L^2-estimate} with $\eps_j\downarrow 0$ and $\beta_j\downarrow 0$, blow-ups are constructed by a suitable linearisation procedure of $V_j$ relative to $\BC_j$. We remark that, unlike in \cite{BK17} (but similarly to \cite{Wic14, Min21a, BKMW25}), this blow-up procedure is performed under an assumption guaranteeing that the fine excess $F_{V_j,\BC_j}$ is significantly smaller than the fine excess of $V_j$ relative to coarser cones. At this stage we do not need to worry about the regularity of the blow-ups, which is a key focus in Part \ref{part:regularity}.

Finally, in Section \ref{sec:main-estimates} we complete the inductive proof of Theorem \ref{thm:main-L^2-estimate}. This completes the first of the two main inductive arguments in the proof of Theorem \ref{thm:main-reg}.

In Part \ref{part:regularity}, we move to the second inductive argument. The aim here is to prove \emph{fine} $\eps$-regularity theorems inductively from finer cones to coarser cones. By this, we mean $\eps$-regularity theorems again under the assumption that the fine excess $F_{V,\BC}$ is significantly smaller than the infimum of the fine excess relative to coarser cones. A first instance of this appeared in \cite[Theorem 3.1]{MW24}, and similar results have played key roles in \cite{Min21a, BKMW25}. We state the main theorem in Theorem \ref{thm:fine-reg}, which directly implies Theorem \ref{thm:main-reg}. Thus, the proof is reduced to proving Theorem \ref{thm:fine-reg}, which we will prove inductively on the level of the cone $\BC$.

In Section \ref{sec:reg}, we study the regularity properties of the blow-ups constructed in Section \ref{sec:blow-up}, aiming for $C^{1,\alpha}$ regularity. We start by establishing Hölder regularity of blow-ups (this does not require any new information, and could have been done in Section \ref{sec:blow-up}). In order to upgrade this regularity further, one needs to classify all blow-ups which are homogeneous of degree one. We will show that these are necessarily linear functions (possibly only on half-hyperplanes) which are generated by sequences of cones which are the same level or finer than the sequence of cones over which the blow-up was constructed -- this allows us to ``integrate out'' the degree one homogeneity in our blow-ups by suitably modifying the sequence of cones we blew-up relative to. Once we have this classification, combining it with a ``reverse Hardt--Simon'' argument allows us to establish quantitative $C^{1,\alpha}$ regularity of blow-ups. The inductive validity of the fine $\eps$-regularity theorem is used crucially here in order to run these ``reverse Hardt--Simon'' arguments, both in the classification of homogeneous degree one blow-ups and in the $C^{1,\alpha}$ regularity proof. The reason for this is if one were to need to integrate out linear functions arising from a sequence of \emph{finer} cones, one can instead apply the fine regularity theorem inductively to guarantee full regularity, with estimates.

Finally, in Section \ref{sec:excess-decay} we use the $C^{1,\alpha}$ regularity of blow-ups established in Section \ref{sec:reg} to prove an excess decay lemma at the varifold level. We can then iterate this to complete the inductive proof of Theorem \ref{thm:fine-reg}, which completes the proof of Theorem \ref{thm:main-reg} and hence of Theorem \ref{thm:main-compactness}.

The reader will note that our proof utilises ideas from the works \cite{Sim93, Wic14, BK17, MW24, Min21a, BKMW25}. For the sake of length, at numerous points in our presentation we will refer the reader to corresponding arguments in these. Unfortunately, it is not possible to cite exact statements, as the situation studied here is different to those before, but only analogous arguments which also work in the present situation with few (predominantly notational) changes. As such, from this point on we will assume familiarity with these works, and in certain places we will only sketch the details of the argument and cite where complete arguments can be found in corresponding situations.

\bigskip

\part{The Main $L^2$ Estimates}\label{part:estimates}

In this section we state the main collection of estimates, Theorem \ref{thm:main-L^2-estimate} and Corollary \ref{cor:main-L^2-estimate}. These will be proved within the rest of Part \ref{part:estimates}.

Recall that we have fixed a based cone $\BC^{(0)}\in \CC$ once and for all. We then have $\BC^{(0)}\in \mathfrak{C}_{i^{(0)}}$ for some $i^{(0)}$. We typically work with coordinates on $\R^{n+1}$ so that $\{e_1,\dotsc,e_{\sfrak^{(0)}}\}$ is an orthonormal basis of $S(\BC^{(0)})$, where $\{e_1,\dotsc,e_{n+1}\}$ is an orthonormal basis of $\R^{n+1}$; in particular, $\sfrak^{(0)} = \dim S(\BC^{(0)})$.

We consider cones $\BC\in \CC$ which obey:
\begin{itemize}
	\item $S(\BC)\subseteq S(\BC^{(0)})$, i.e.~$\BC$ and $\BC^{(0)}$ are aligned;
	\item $d(\BC\res B_1,\BC^{(0)}\res B_1)<\eps_0$ for appropriately small $\eps_0\in (0,1)$.
\end{itemize}
Again, here $d$ denotes a chosen metric induced by the varifold topology (in particular, it keeps track of both the position and the multiplicity of the (half-)hyperplanes in the cones). Notice that if $\BC\in \CC_{n-1}$ then $S(\BC) = S(\BC^{(0)}) = \R^{n-1}\times\{0\}^2$.

If $\eps_0 = \eps_0(n,Q,\BC^{(0)})$ is chosen sufficiently small, then $\BC^{(0)}\preceq \BC$, and hence we must have $\BC\in\mathfrak{C}_j$ with $j\geq i_0$. For the few possibilities that there are depending on the form of $\BC^{(0)}$, we choose notation as follows:
\begin{enumerate}
	\item [(a)] If $\BC^{(0)}\in \mathcal{P}$, write $\BC^{(0)} = \sum^{\pfrak_0}_{i=1}q_i|P^{(i)}|$, where $P^{(i)}$ are distinct hyperplanes. Then $\BC$ can have one of two forms:
	\begin{itemize}
		\item If $\BC\in \mathcal{P}$, we write $\BC = \sum^{\pfrak_0}_{i=1}\sum^{\pfrak_i}_{j=1}q_{i,j}|P^{(i,j)}|$, where $\sum^{\pfrak_i}_{j=1}q_{i,j}=q_i$ for each $i$ and the $P^{(i,j)}$ are distinct hyperplanes with the closest plane in $\BC^{(0)}$ to $P^{(i,j)}$ is $P^{(i)}$, i.e.
		$$\dist_\H(P^{(i,j)}\cap B_1,P^{(i)}\cap B_1) = \min_k \dist_\H(P^{(i,j)}\cap B_1, P^{(k)}\cap B_1);$$
		\item If $\BC\in \H_{n-1}$, then $\BC^{(0)}\in \mathcal{P}_{n-1}$. In this case, we write $\BC = \sum^{\pfrak_0}_{i=1}\sum^{\pfrak_i}_{k=1}q_{i,j}|H^{(i,j)}|$, where $\sum^{\pfrak_i}_{j=1}q_{i,j}=2q_i$ for each $i$ and $H^{(i,j)}$ are distinct half-hyperplanes such that, writing $\widetilde{H}^{(i,j)}$ for the hyperplane extending $H^{(i,j)}$:
		$$\dist_\H(\widetilde{H}^{(i,j)}\cap B_1,P^{(i)}\cap B_1) = \min_k \dist_\H(\widetilde{H}^{(i,j)}\cap B_1, P^{(k)}\cap B_1).$$
	\end{itemize}
	\item [(b)] If $\BC^{(0)}\in \H_{n-1}$, write $\BC^{(0)} = \sum^{\pfrak_0}_{i=1}q_i|H^{(i)}|$ where each $H^{(i)}$ is a distinct half-hyperplane. Then we must have $\BC\in \H_{n-1}$, and we write $\BC = \sum^{\pfrak_0}_{i=1}\sum^{\pfrak_i}_{j=1}q_{i,j}|H^{(i,j)}|$, where $\sum^{\pfrak_i}_{j=1}q_{i,j}=q_i$ for each $i$ and each $H^{(i,j)}$ is a distinct half-hyperplane with
		$$\dist_\H(H^{(i,j)}\cap B_1,H^{(i)}\cap B_1) = \min_k \dist_\H(H^{(i,j)}\cap B_1,H^{(k)}\cap B_1).$$
\end{enumerate}
We further remark that, by simple linear algebra, in each case we also have comparability of the Hausdorff distance between a (half-)hyperplane in $\BC$ to the closest (half-)hyperplane in $\BC^{(0)}$ by the total Hausdorff distance between $\BC^{(0)}$ and $\BC$. For instance, in case (a) of the first bullet point above we have
\begin{equation}\label{E:P-to-C}\tag{2.1}
\dist_\H(P^{(i,j)}\cap B_1, P^{(i)}\cap B_1)  \leq C(n,Q)\dist_\H(\BC\cap B_1,\BC^{(0)}\cap B_1).
\end{equation}
We will always assume that $\eps_0 = \eps_0(n,Q,\BC^{(0)})\in (0,1)$ is chosen small enough to ensure that the above conclusions hold, and we will adopt this notation throughout. At times it will be convenient to view a cone in $\CC_{n-1}$ as a sum of half-hyperplanes regardless of whether it is in $\H_{n-1}$ or in $\mathcal{P}_{n-1}$, and thus we will often use the notation corresponding to half-hyperplanes in this setting.

We also make the following definitions:
\begin{itemize}
	\item $\mathcal{E}_{V,\BC}:=\inf\{E_{V,\BC^\prime}:\BC^\prime\prec \BC\text{ are aligned}\}$;
	\item $\mathcal{F}_{V,\BC}:= \inf\{F_{V,\BC^\prime}:\BC^\prime \prec \BC\text{ are aligned}\}$.
\end{itemize}
Here, we make the convention that if $\BC\in \mathcal{P}_{n-1}(2)\cup \H_{n-1}(3)$ (where there are no coarser families of cones) the respectively infimum is taken over all hyperplanes $P$ obeying $S(\BC)\subset P$. Notice in this latter case that both quantities will be $\geq c$ for some suitable $c = c(\BC^{(0)})>0$.

Finally, for a subspace $S\subset\R^{n+1}$ we write $\pi_S:\R^{n+1}\to S$ for the corresponding orthogonal projection, as $\pi_{S}^\perp := \pi_{S^\perp} \equiv \id-\pi_S$, where $\id:\R^{n+1}\to \R^{n+1}$ is the identity map.

\section{Statements}\label{sec:main-estimates-statements}

Our main graphicality claim and $L^2$ estimates are the following.

\begin{theorem}\label{thm:main-L^2-estimate}
	Let $\tau\in (0,1)$. Then, there exist $\eps_0,\beta_0\in (0,1)$ depending only on $n,Q,\BC^{(0)},\tau$ such that if $V\in \V_Q$ and $\BC\in \CC$ satisfy:
	\begin{enumerate}
		\item [\textnormal{(i)}] $\Theta_V(0)\geq Q$;
		\item [\textnormal{(ii)}] $d(V\res B^{n+1}_2(0), \BC^{(0)}\res B^{n+1}_2(0)) < \eps_0$;
		\item [\textnormal{(iii)}] $d(\BC\res B^{n+1}_1(0),\BC^{(0)}\res B^{n+1}_1(0))<\eps_0$, and $\BC$ is aligned with $\BC^{(0)}$;
		\item [$(\dagger)$] $F_{V,\BC}<\beta_0\mathcal{F}_{V,\BC}$;
	\end{enumerate}
	then, up to replacing $\BC$ by another cone with the same support (for which we do not change our notation), the following conclusions hold:
	\begin{enumerate}
		\item [(1)] $\{x\in B^{n+1}_{31/32}(0): \Theta_V(x)\geq Q\}\subset B_{\tau}(S(\BC))$. Moreover, depending on the form of $\BC$, we have:
		\begin{itemize}
			\item If $\BC\in \mathcal{P}_{\leq n-2}$, then
			$$V\res (B^{n+1}_{31/32}(0)\setminus B_\tau(S(\BC))) = \sum^{\pfrak_0}_{i=1}\sum^{\pfrak_i}_{j=1}\sum^{q_{i,j}}_{k=1}|\graph(u^{(i,j,k)})|\res B^{n+1}_{31/32}(0),$$
			where $u^{(i,j,k)}:P^{(i,j)}\cap B^{n+1}_{31/32}(0)\setminus B_\tau(S(\BC))\to (P^{(i,j)})^\perp$ is smooth, satisfies the minimal surface equation, and obeys $\|u^{(i,j,k)}\|_{C^4}\leq C_* E_{V,\BC}$;
			\item If $\BC\in \CC_{n-1}$, then
			$$V\res (B^{n+1}_{31/32}(0)\setminus B_\tau(S(\BC))) = \sum^{\pfrak_0}_{i=1}\sum^{\pfrak_i}_{j=1}\sum^{q_{i,j}}_{k=1}|\graph(u^{(i,j,k)})|\res B^{n+1}_{31/32}(0),$$
			where $u^{(i,j,k)}: H^{(i,j)}\cap B^{n+1}_{31/32}(0)\setminus B_\tau(S(\BC)) \to (\widetilde{H}^{(i,j)})^\perp$ is smooth, satisfies the minimal surface equation, and obeys $\|u^{(i,j,k)}\|_{C^4}\leq C_* E_{V,\BC}$;
			\item Furthermore, if $y\in S(\BC)$ has $B^{n+1}_{4\tau}(y)\subset \{\Theta_V<Q\}$, then necessarily $\BC\not\in \mathcal{H}_{n-1}$ and
		$$V\res B^{n+1}_{2\tau}(y) = \sum^{\pfrak_0}_{i=1}\sum^{\pfrak_i}_{j=1}\sum^{q_{i,j}}_{k=1}|\graph(u^{(i,j,k)})|\res B^{n+1}_{2\tau}(y),$$
		where $u^{(i,j,k)}:P^{(i,j)}\cap B^{n+1}_{2\tau}(y)\to (P^{(i,j)})^\perp$ is smooth, satisfies the minimal surface equation, and obeys $\|u^{(i,j,k)}\|_{C^4}\leq C_* E_{V,\BC}$.
		\end{itemize}
		\item [(2)] $\displaystyle \int_{B^{n+1}_{7/8}(0)}\sum^\sfrak_{s=1}|\pi^\perp_{T_xV}(e_s)|^2\, \ext\|V\|(x) \leq CE_{V,\BC}^2$, where $\{e_1,\dotsc,e_\sfrak\}$ is the basis for $S(\BC)$;
		\item [(3)] $\displaystyle \int_{B^{n+1}_{7/8}(0)}\frac{\dist^2(x,\BC)}{|x|^{n+7/4}}\, \ext\|V\|(x) \leq CE_{V,\BC}^2;$
		\item [(4)] $\displaystyle \int_{B^{n+1}_{7/8}(0)}\frac{|\pi^\perp_{T_xV}(x)|^2}{|x|^{n+2}}\, \ext\|V\|(x) \leq CE_{V,\BC}^2;$
		\item [(5)] Writing $U^{(i,j)}$ as a shorthand for the domain of $u^{(i,j,k)}$ and $R(x) = |x|$, we have
		$$\sum_{i,j,k}\int_{U^{(i,j)}}R^{2-n}\left|\frac{\del}{\del R}\left(\frac{u^{(i,j,k)}}{R}\right)\right|^{2}\, \ext x \leq CE_{V,\BC}^2.$$
	\end{enumerate}
	Here, $C = C(n,Q)\in (0,\infty)$ and $C_* = C_*(n,Q,\tau)\in (0,\infty)$.
\end{theorem}

\begin{remark}
	We make the following comments regarding Theorem \ref{thm:main-L^2-estimate}.
	\begin{enumerate}
		\item [(a)] For any $\tilde{\eps}>0$, provided $\eps$ is sufficiently small depending on $n,Q,\BC^{(0)},\tilde{\eps}$, the assumption that $d(V\res B_2,\BC^{(0)}\res B_2)<\eps_0$ implies that $(\w_n 2^n)^{-1}\|V\|(B^{n+1}_2(0))<Q+\frac{1}{4}$ and $F_{V,\BC^{(0)}}< \tilde{\eps}$. In fact, it is much stronger, as it also ensures that the mass of $V$ near each (half-)hyperplane in $\BC^{(0)}$ is the same as its multiplicity in $\BC^{(0)}$. Whilst the same is true for the (half-)hyperplanes in $\BC$ when compared to $\BC^{(0)}$, this is not necessarily true when compared the (half-)hyperplanes in $\BC$ to the mass distribution of $V$. This is reason for potentially needing to change $\BC$ to a different cone with the same support.
		\item [(b)] It is not possible for cones of certain levels, i.e.~in certain classes $\mathfrak{C}_i$, to satisfy the assumptions of Theorem \ref{thm:main-L^2-estimate}, but only those in classes which are $\succeq \mathfrak{C}_{i^{(0)}}$, i.e.~finer than or in the same class as $\BC^{(0)}$. In particular, when the proof is by induction on the level $i$, the proof for certain levels is trivial as the assumptions cannot be satisfied for $\eps_0,\beta_0$ sufficiently small.
		\item [(c)] Several aspects of Theorem \ref{thm:main-L^2-estimate} allow for a degree of flexibility. For instance, the radius of the balls on which the conclusions hold (namely, $B^{n+1}_{31/32}(0)$ in (1) and $B^{n+1}_{7/8}(0)$ in (2)\,--\,(4)) can be replaced with a parameter $\sigma\in (1/2,1)$, provided we allow all the constants to depend on $\sigma$. Moreover, the assumption $\Theta_V(0)\geq Q$ can be removed. Indeed, either there is a point of density $\geq Q$ in $B^{n+1}_{1/4}(0)$ (in which case one can reduce to the case as stated in Theorem \ref{thm:main-L^2-estimate}, after translating and shrinking some radii), or $\Theta_V<Q$ on all of $B^{n+1}_{1/4}(0)$. In this latter case, Lemma \ref{lemma:gaps} tells us that the non-immersed singular set is small enough in order to apply the result in \cite{HLW24} in order to deduce the above conclusions, again on a smaller radius (one needs an additional argument to show that, even with fully regularity, one has the desired estimates, which will follow from our proof). Indeed, this is more or less already included in the final component of conclusion (1) above. A short covering argument applied to translations and rescalings of $V$ then give the same conclusions of Theorem \ref{thm:main-L^2-estimate} without the assumption $\Theta_V(0)\geq Q$.
		\item [(d)] The part of the final bullet point in Theorem \ref{thm:main-L^2-estimate}(1) claiming that $\BC\not\in \H_{n-1}$ will be shown in Lemma \ref{lemma:no-gaps}.
	\end{enumerate}
\end{remark}
Theorem \ref{thm:main-L^2-estimate} has the following important corollary, which in particular is used to show that the $L^2$ norm does not concentrate along $S(\BC)$ when performing a blow-up (see Lemma \ref{lemma:non-con}). To unify notation, if $\BC^{(0)}$ is comprised of half-hyperplanes $H^{(i)}$ let us write $P^{(i)}:= \widetilde{H}^{(i)}$ for the hyperplane containing $H^{(i)}$. We also write $\tau_z(x):= x-z$ for $x,z\in\R^{n+1}$.

\begin{corollary}\label{cor:main-L^2-estimate}
	Let $\rho\in (0,3/4)$. Then, there exist $\eps_0,\beta_0\in (0,1)$ depending only on $n,Q,\BC^{(0)},\rho$, such that the following holds. Suppose $V\in \V_Q$ and $\BC\in \CC$ satisfy the assumptions of Theorem \ref{thm:main-L^2-estimate} with these choices of $\eps_0,\beta_0$. Then, for any $z\in \{\Theta_V\geq Q\}\cap B^{n+1}_{3/4}(0)$, writing $\xi:= \pi_{S(\BC)}^\perp(z)$, we have:
	\begin{enumerate}
		\item [\textnormal{(1)}] $\displaystyle |\pi^\perp_{P^{(i)}}(\xi)| + E_{V,\BC^{(0)}}|\pi_{P^{(i)}}(\xi)| \leq CE_{V,\BC}$ for each $i=1,\dotsc,\pfrak_0$;
		\item [\textnormal{(2)}] $\displaystyle \int_{B^{n+1}_{5\rho/8}(z)}\frac{\dist^2(x,(\tau_z)_\#\BC)}{|x-z|^{n+7/4}}\, \ext\|V\|(x) \leq C\rho^{-n-7/4}\int_{B^{n+1}_\rho(z)}\dist^2(x,(\tau_z)_\#\BC)\, \ext\|V\|(x)$.
	\end{enumerate}
	Here, $C = C(n,Q,\BC^{(0)})\in (0,\infty)$ is independent of $\rho$.
\end{corollary}

\begin{remark}
	In the notation of Corollary \ref{cor:main-L^2-estimate}, $\xi$ is a vector with $n+1-\sfrak$ distinct components. The inequality Corollary \ref{cor:main-L^2-estimate}(1) bounds $n+1-\sfrak^{(0)}$ of these components by $CE_{V,\BC}$. In particular, when $\sfrak=\sfrak^{(0)}$ we in fact get $|\xi|\leq CE_{V,\BC}$. Otherwise, we must make do with the weaker bound from Corollary \ref{cor:main-L^2-estimate}(1) in the directions $S(\BC)^\perp\setminus S(\BC^{(0)})^\perp$. We mention that in Part \ref{part:regularity} when establishing the regularity of blow-ups, the hardest case is when $\sfrak = \sfrak^{(0)} = n-1$, and so there we do have $|\xi|\leq CE_{V,\BC}$. Nonetheless, we need the version above in order to prove $L^2$ non-concentration of excess in every situation.
\end{remark}
\medskip

As mentioned already, the proof of Theorem \ref{thm:main-L^2-estimate} and Corollary \ref{cor:main-L^2-estimate} is by induction on the ordering $\{\mathfrak{C}_i\}_i$, namely from coarser cones to finer cones. The base case is therefore when $\BC$ has the same level as $\BC^{(0)}$. As such, when establishing these results for a given $\BC$, we will be able to assume they hold for all cones coarser than $\BC$. This is key, as we need to perform blow-ups off coarser cones in our inductive proof of these results. These coarser cones will be provided by the ``replacement lemma'' in Section \ref{sec:replacement}. In comparison to \cite[Theorem 3.2]{BK17}, the additional assumption $(\dagger)$ in Theorem \ref{thm:main-L^2-estimate} is needed analogously to that seen in \cite[Theorem 10.1 \& Theorem 16.2]{Wic14} or \cite[Theorem 7.3]{BKMW25}, where the possibility of multiplicity in the cones requires finer information in order to determine which planes in $\BC$ one should be graphing over, and to prove the relevant estimate on the graph. It also is what guarantees the ``good density points'' to accumulate around the spine of $\BC$, as stated in Theorem \ref{thm:main-L^2-estimate}(1).

The proof of Corollary \ref{cor:main-L^2-estimate} is similar to the corresponding results in \cite{Wic14, BK17, BKMW25} under the correct inductive assumptions, as one wishes to apply Theorem \ref{thm:main-L^2-estimate} with $(\eta_{z,\rho})_\#V$ in place of $V$. To verify that the assumptions of Theorem \ref{thm:main-L^2-estimate} still hold for $(\eta_{z,\rho})_\#V$, one needs sufficient control on $\dist(z,S(\BC)) \equiv |\xi|$. This will follow again using the inductive assumption to perform a blow-up off a coarser sequence of cones.

We will also prove Theorem \ref{thm:main-L^2-estimate} working inductively over the level of the base cone $\BC^{(0)}$; this is useful in the proof of Lemma \ref{lemma:one-sided-2}.

To end this section, we note the following ``separation property'' for $\BC$ under suitable assumptions.

\begin{lemma}\label{lemma:separation}
	There exist $\eps_0,\beta_0\in (0,1)$ depending only on $n,Q,\BC^{(0)}$ such that if $V\in\V_Q$ and $\BC\in \CC$ satisfy:
	\begin{enumerate}
		\item [(i)] $(\w_n 2^n)^{-1}\|V\|(B^{n+1}_2(0)) <Q+\frac{1}{4};$
		\item [(ii)] $F_{V,\BC^{(0)}} + d(\BC\res B_1,\BC^{(0)}\res B_1)<\eps_0;$
		\item [(iii)] $E_{V,\BC}<\beta_0 \mathcal{E}_{V,\BC};$
	\end{enumerate}
	then there is a constant $c = c(n,Q,\BC^{(0)})>0$ for which the following hold:
	\begin{enumerate}
		\item [\textnormal{(a)}] If $\BC\in \mathcal{P}_{\sfrak}(\pfrak)$, then for any two distinct planes $P_1,P_2$ in $\BC$ we have
		$$\dist_\H(P_1\cap B_1,P_2\cap B_1) \geq c\mathcal{E}_{V,\BC}.$$
		Furthermore, if $\sfrak\leq n-2$ then for any $\widetilde{\BC}\in\mathcal{P}_{\sfrak+1}(\pfrak)$ aligned with $\BC$ we have:
		$$\dist_\H(\BC\cap B_1,\widetilde{\BC}\cap B_1)\geq c\mathcal{E}_{V,\BC}.$$
		\item [\textnormal{(b)}] If $\BC\in \H_{n-1}(\pfrak)$, then for any two distinct half-hyperplanes $H_1,H_2$ in $\BC$ we have
		$$\dist_\H(H_1\cap B_1,H_2\cap B_1)\geq c\mathcal{E}_{V,\BC}.$$
		Furthermore, if $\pfrak$ is even, then for any $\widetilde{\BC}\in\mathcal{P}_{n-1}(\pfrak/2)$ aligned with $\BC$ we have
		$$\dist_\H(\BC\cap B_1,\widetilde{\BC}\cap B_1)\geq c\mathcal{E}_{V,\BC}.$$
	\end{enumerate}
\end{lemma}

\begin{proof}
	The claimed properties are simple to verify and have similar proofs, so we only sketch the proof of (a). The key observation lies in that for any aligned $\widetilde{\BC}\prec\BC$, by the triangle inequality we have
	\begin{align*}
		\mathcal{E}_{V,\BC}^2 \leq E_{V,\widetilde{\BC}}^2 & = \int_{B_1}\dist^2(x,\widetilde{\BC})\, \ext\|V\|(x)\\
		& \leq 2\int_{B_1}\dist^2(x,\BC)\, \ext\|V\|(x) + 2\|V\|(B_1)\cdot\dist_\H^2(\BC\cap B_1,\widetilde{\BC}_1\cap B_1)\\
		& \leq 2E_{V,\BC}^2 + C(n,Q)\cdot\dist^2_\H(\BC\cap B_1,\widetilde{\BC}\cap B_1)\\
		& \leq 2\beta_0^2\mathcal{E}_{V,\BC}^2 + C(n,Q)\cdot\dist_\H^2(\BC\cap B_1,\widetilde{\BC}\cap B_1).
	\end{align*}
	So, if $\beta_0\in (0,1/2)$, we get
	$$c(n,Q)\mathcal{E}_{V,\BC} \leq \dist_\H(\BC\cap B_1,\widetilde{\BC}\cap B_1).$$
	This proves the second claim in both cases. To prove the first, one simply chooses $\widetilde{\BC}$ appropriately in the above. For instance, in (a) one takes $\widetilde{\BC}$ to be the cone formed when replacing both $P_1$ and $P_2$ in $\BC$ with their average (i.e.~as this are hyperplanes, this simply means taking the hyperplane corresponding to the midpoint of the geodesic connecting the corresponding elements of $\R P^n$).
\end{proof}

\section{The Replacement Lemma}\label{sec:replacement}

The purpose of the following section is to outline a general procedure which we will use to find ``good'' choices of cones which almost achieve the infimum $\mathcal{F}_{V,\BC}$. The results we prove will illustrate this idea, although often in applications we need to utilise the construction itself rather than the results to guarantee certain constants are fixed, depending only on $n,Q,\BC^{(0)}$.

The following hypothesis will be recurrent for $V\in \V_Q$ and $\BC\in \CC$. Here, $\eps_1>0$ is a small parameter.
\begin{itemize}
	\item \textbf{Hypothesis $H(\eps_1)$:} $V$ and $\BC$ satisfy:
	\begin{enumerate}
		\item [(i)] $(\w_n 2^n)^{-1}\|V\|(B^{n+1}_2(0))<Q+\frac{1}{4}$ and $\Theta_V(0)\geq Q$;
		\item [(ii)] $F_{V,\BC^{(0)}} + d(\BC\res B_1,\BC^{(0)}\res B_1)<\eps_1$.
	\end{enumerate}
\end{itemize}
Let us suppose $V,\BC$ satisfy Hypothesis $H(\eps_1)$ for suitably small $\eps_1\in (0,1)$. We then define:
\begin{defn}\label{defn:replacement}
	For $\eps,\beta\in (0,1)$, we say $\widetilde{\BC}$ is a $(\eps,\beta)$\emph{-replacement} for $\BC$ if:
	\begin{enumerate}
		\item [(a)] $\widetilde{\BC}\preceq\BC$ are aligned;
		\item [(b)] $V,\widetilde{\BC}$ satisfy Hypothesis $H(\eps$);
		\item [(c)] $F_{V,\widetilde{\BC}}\leq\beta\mathcal{F}_{V,\widetilde{\BC}}$;
		\item [(d)] $F_{V,\widetilde{\BC}}\leq CF_{V,\BC}$.
	\end{enumerate}
	Here, $C = C(n,Q,\BC^{(0)},\beta)\in (0,\infty)$ is a suitable constant. We will colloquially refer to a $(\eps,\beta)$-replacement simply as a \emph{replacement}.
\end{defn}
When using replacements, the constant $\beta$ should ultimately be chosen to only depend only on $n,Q,\BC^{(0)}$, which fixes the constant $C$ in the definition. We also note that if $\widetilde{\BC}$ were in the same class as $\BC^{(0)}$, then $\mathcal{F}_{V,\widetilde{\BC}}\geq c$ for some fixed $c = c(n,Q,\BC^{(0)})>0$, and so condition (c) in Definition \ref{defn:replacement} is trivially satisfied as soon as $\eps_1$ is sufficiently small.

The following lemma illustrates a general selection principle for replacements.

\begin{lemma}[Replacement Lemma]\label{lemma:replacement}
	Let $\tilde{\eps},\tilde{\beta}\in (0,1)$. Then, there exists $\eps = \eps(n,Q,\BC^{(0)},\tilde{\eps},\tilde{\beta})\in (0,1)$ such that if $V,\BC$ satisfy Hypothesis $H(\eps)$, then there exists a $(\tilde{\eps},\tilde{\beta})$-replacement for $\BC$.
\end{lemma}

\begin{remark}
The type of situation one should bear in mind for why we need replacements is when $V$ itself is in $\CC$. Indeed, say $V\in \mathcal{P}_{n-1}(2)$ and $\BC$ is an arbitrarily small perturbation of $V$ such that $\BC\in\mathcal{P}_{\sfrak}(\pfrak)$ for some $s<n-1$. Then a condition of the form $F_{V,\BC}\leq \beta\mathcal{E}_{V,\BC}$ will always fail (as the right-hand side is zero and the left-hand side is non-zero), but one can find a replacement for which the corresponding inequality holds (e.g.~$\widetilde{\BC} = V$). Furthermore, condition (d) in Definition \ref{defn:replacement} guarantees that the height excess for the new cone has increased by at most a constant factor.
\end{remark}

\begin{proof}
	If $F_{V,\BC}\leq \tilde{\beta} \mathcal{F}_{V,\BC}$, then we can simply take $\eps = \tilde{\eps}$ and then $\widetilde{\BC} := \BC$ is the desired replacement (note that this always happens in the ``base case'' when $\BC$ has the same level, provided $\tilde{\eps}$ is sufficiently small, as then $\mathcal{F}_{V,\BC}\geq c = c(\BC^{(0)})>0$).
	So we may assume
	\begin{equation}\label{E:replacement-a}
		F_{V,\BC}\geq \tilde{\beta}\mathcal{F}_{V,\BC}.
	\end{equation}
	We now set up the following inductive procedure. Set $\BD_0 := \BC$. Then, given $\BD_{p-1}$, choose $\BD_p$ inductively to satisfy:
	\begin{itemize}
		\item $\BD_p\prec \BD_{p-1}$ are aligned;
		\item $\BC^{(0)}\preceq \BD_p$;
		\item $F_{V,\BD_p} \leq\frac{3}{2}\mathcal{F}_{V,\BD_{p-1}}$.
	\end{itemize}
	At each step, if either $\BC^{(0)}$ and $\BD_p$ are in the same class of cones, or $F_{V,\BD_p} \leq \tilde{\beta}\mathcal{F}_{V,\BD_p}$, then we stop the process and set $\widetilde{\BC}:= \BD_p$. Notice that the process must terminate in at most a finite number $\leq N_0 = N_0(n,Q)$ of steps, as the level of a strictly decreasing sequence under $\prec$ is bounded.
	
	We claim that $\widetilde{\BC}$ is the desired replacement. Indeed, by construction we definitely have $\widetilde{\BC}\preceq \BC$ are aligned. Clearly as $V,\BC$ converge to $\BC^{(0)}$ as $\eps\to 0$ (namely, their supports converge in Hausdorff distance), the same must be true for $\widetilde{\BC}$, which guarantees that $V,\widetilde{\BC}$ will satisfy Hypothesis $H(\tilde{\eps})$ if $\eps$ is chosen sufficiently small (up to possibly changing multiplicities on the (half-)hyperplanes in $\widetilde{\BC}$). Now, if the process terminates before $\BD_p$ and $\BC^{(0)}$ are in the same class of cones, then $F_{V,\widetilde{\BC}} \leq \tilde{\beta} \mathcal{F}_{V,\widetilde{\BC}}$ holds. If instead $\BD_p$ and $\BC^{(0)}$ are in the same class, then using the universal lower bound $\mathcal{F}_{V,\widetilde{\BC}}\geq c(n,Q,\BC^{(0)})$, we have $F_{V,\widetilde{\BC}} \leq c\tilde{\beta} \leq \tilde{\beta}\mathcal{F}_{V,\widetilde{\BC}}$ automatically holds for $\eps = \eps(n,Q,\BC^{(0)},\tilde{\beta})$ sufficiently small, as the left-hand side $\to 0$ as $\eps\to 0$.
	
	Finally, by construction and \eqref{E:replacement-a} we have
	\begin{align*}
	F_{V,\BC} \geq \tilde{\beta}\mathcal{F}_{V,\BD_0} \geq \frac{2}{3}\tilde{\beta} F_{V,\BD_1} \geq \frac{2}{3}(\tilde{\beta})^2\mathcal{F}_{V,\BD_1} \geq \cdots & \geq \left(\frac{2}{3}\tilde{\beta}\right)^{N_0(n,Q)}F_{V,\BD_p} = C(n,Q,\BC^{(0)},\tilde{\beta})F_{V,\widetilde{\BC}}.
	\end{align*}
	This shows $V,\widetilde{\BC}$ satisfy all the required properties, completing the proof.
\end{proof}

Next we prove the existence of cones which are ``good'' representatives for $\mathcal{F}_{V,\BC}$. This notion will be key in our arguments when proving Theorem \ref{thm:main-L^2-estimate}.

\begin{lemma}\label{lemma:good-representative}
	Let $\tilde{\eps},\tilde{\beta}\in (0,1)$. Then, there exists $\eps = \eps(n,Q,\BC^{(0)},\tilde{\eps},\tilde{\beta})\in (0,1)$ such that if $V,\BC$ satisfy Hypothesis $H(\eps)$ and $\BC^{(0)}\prec \BC$, then there exists an aligned cone $\widetilde{\BC}\prec\BC$ satisfying:
	\begin{enumerate}
		\item [\textnormal{(i)}] $V,\widetilde{\BC}$ satisfy Hypothesis $H(\tilde{\eps})$;
		\item [\textnormal{(ii)}] $F_{V,\widetilde{\BC}} \leq \tilde{\beta}\mathcal{F}_{V,\widetilde{\BC}}$;
		\item [\textnormal{(iii)}] $F_{V,\widetilde{\BC}}\leq C\mathcal{F}_{V,\BC}$, for some $C = C(n,Q,\BC^{(0)},\tilde{\beta})\in (0,\infty)$.
	\end{enumerate}
\end{lemma}
We call a cone $\widetilde{\BC}$ satisfying the conclusions of Lemma \ref{lemma:good-representative} a $(\tilde{\eps},\tilde{\beta})$\emph{-good representative} for $\mathcal{F}_{V,\BC}$.

The only reason one assumes additionally that $\BC^{(0)}\prec \BC$ in Lemma \ref{lemma:good-representative} is because if $\BC^{(0)}$ and $\BC$ were the same level, one could not choose a cone $\widetilde{\BC}\in \BC$ which was coarser than $\BC$ and was still close to $V$.

\begin{proof}
	First choose aligned $\BC^\prime\prec\BC$ with $F_{V,\BC^\prime}\leq \frac{3}{2}\mathcal{F}_{V,\BC}$. Again, up to possibly changing the multiplicity on some of the (half)-planes in $\BC^{\prime}$, we have that $V,\BC^\prime$ necessarily satisfy Hypothesis $H(\eps^\prime)$ for any suitably small $\eps^\prime$, provided we choose $\eps$ sufficiently small. We may therefore apply Lemma \ref{lemma:replacement} to find an $(\tilde{\eps},\tilde{\beta})$-replacement $\widetilde{\BC}$ for $\BC^\prime$, which one sees satisfies the conclusions of the present lemma.
\end{proof}

\begin{remark}\label{remark:replacements}
	The reader will readily note that one can prove similar versions of Lemma \ref{lemma:replacement} and Lemma \ref{lemma:good-representative} for one-sided excess quantities rather than the two-sided excess. Hypothesis $H(\eps)$ is unchanged.
\end{remark}

\begin{remark}\label{remark:aligned}
	Note that there is no assumption that the cone $\widetilde{\BC}$ provided by Lemma \ref{lemma:replacement} or Lemma \ref{lemma:good-representative} is aligned with $\BC^{(0)}$. This was, however, an assumption on the cone in Theorem \ref{thm:main-L^2-estimate}. However, if $V$ and $\BC$ satisfy all the conditions of Theorem \ref{thm:main-L^2-estimate} except the alignment of $\BC$ and $\BC^{(0)}$, and $S(\BC)$ is close to $S(\BC^{(0)})$, one can apply a small rotation to $V$ and $\BC$ to ensure the alignment, then apply Theorem \ref{thm:main-L^2-estimate}, and rotate back, in order to get the conclusions of Theorem \ref{thm:main-L^2-estimate}. When needed, we will do this without further comment.
\end{remark}

\section{The Blow-Up Class}\label{sec:blow-up}

In this section, under the assumption of the validity of the results in Section \ref{sec:main-estimates-statements}, we will construct classes of blow-ups of suitable sequences in $\V_Q$ relative to cones in $\CC$ and study their basic properties. Thus, throughout this section we assume the following inductive hypothesis for some fixed $i_0\geq 2$:

\hspace{2em}\textbf{Hypothesis I:} Theorem \ref{thm:main-L^2-estimate} holds for any $V \in \V_Q$ and $\BC \in \cup_{i=1}^{i_0-1} \FC_i$.

For sequences $\eps_j \downarrow 0$ and $\beta_j\downarrow 0$, the first goal of this section is to construct a (\emph{fine}) \emph{blow-up} $v_\infty= (v_\infty^{(i,k,s)})_{i,k,s}$ from sequences $(V_j)_j\subset\V_Q$ and $(\BC_j)_j\subset \cup_{i=1}^{i_0-1}\FC_i$ which satisfy the following assumptions:
\begin{itemize}
	\item Hypothesis $H(\eps_j)$ (see Section \ref{sec:replacement});
	\item Hypothesis $(\dagger)$: $F_{V_j,\BC_j}<\beta_j\mathcal{F}_{V_j,\BC_j}$.
\end{itemize}
By passing to a subsequence, we may assume that $\BC_j\in \mathfrak{C}_{i_*}$ for a fixed $i_*\leq i_0-1$ and $S(\BC_j)\to S$ for a subspace $S$. In particular, the corresponding parameters $\sfrak,\pfrak$ are independent of $j$.

We will define a blow-up class $\FB_{i_*}$ which depends on the level the sequence $\BC_j$. This will consist of tuples of functions defined on the (half-)hyperplanes of $\BC^{(0)}$ which, up to rotating, we may therefore assume are defined on a fixed (half-)hyperplane. Nonetheless, it will be convenient to keep track of the (half-)hyperplanes within the sequence $\BC_j$ which the functions come from; we describe this momentarily.

We define the set of \emph{good density points} of $V_j$ by:
$$\mathcal{D}_j:= \{z\in \spt\|V_j\|\cap B_1^{n+1}(0): \Theta_{V_j}(z)\geq Q\}.$$
We call $z\in \mathcal{D}_j$ a \emph{good density point} of $V_j$. Notice that $\mathcal{D}_j$ is a relatively closed subset of $B_1$ by upper semi-continuity of density. We refer to open subsets of $B^{n+1}_1(0)\setminus \mathcal{D}_j$ as \emph{density gaps} of $V_j$.

In the blow-up process it is important to keep track of the sets $\mathcal{D}_j$. By passing to a subsequence, we can assume that there exists a relatively closed set $\mathcal{D}\subset B_1$ for which $\mathcal{D}_j\to \mathcal{D}$ locally in Hausdorff distance in $B_1$. We have the characterisation
$$\mathcal{D} = \{z\in B_1: z = \lim_{j\to\infty}z_j\text{ for some }z_j\in\mathcal{D}_j\}.$$
Notice that we always have $\mathcal{D}\subseteq S(\BC^{(0)})\cap B_1^{n+1}(0)$ by upper semi-continuity of the density. In fact, by Theorem \ref{thm:main-L^2-estimate} we know $\D\subseteq S\cap B^{n+1}_1(0)$.
\begin{remark}
	Notice that $\mathcal{D}$ depends on the choice of sequence $(V_j)_j$, even though the notation does not reflect this. In fact, it will just depend on the blow-up formed.
\end{remark}

Our first lemma says that if $\BC_j\in\H_{n-1}$, then there can be no density gaps of ``significant size''. We will use this to reduce to two separate situations for constructing blow-ups.

\begin{lemma}\label{lemma:no-gaps}
	Fix $\delta_0>0$ and suppose $\BC_j\in \H_{n-1}$. Then each $y\in S(\BC^{(0)})\cap B_1$ we must have
	$$B_{\delta_0}(y) \cap \D_j\neq \emptyset \qquad \text{for all $j$ sufficiently large.}$$
\end{lemma}

\begin{remark}
	In fact, instead of working with sequences $V_j,\BC_j$ with parameters $\eps_j,\beta_j\downarrow 0$, the lemma only requires $V,\BC$ to satisfy Hypothesis $H(\eps)$ and Hypothesis $(\dagger)$ with $\eps,\beta$ therein sufficiently small depending on $n,Q,\BC^{(0)}$.
\end{remark}

\begin{proof}
	The proof is similar to \cite[Lemma 7.1]{BK17}. We argue by contradiction. Thus, we may assume that there exists $\delta>0$ and $y\in S(\BC^{(0)})\cap B_1$ for which
	$$B_{\delta}(y)\cap \D_j = \emptyset \qquad \text{for all $j$ sufficiently large.}$$
	Notice that as $\BC_j\in \H_{n-1}$ and $\BC^{(0)}\preceq \BC_j$ are aligned, we have $S(\BC^{(0)}) = S(\BC_j)$.
	
	If $\BC^{(0)}\in \H_{n-1}$, then the lemma trivially follows from Lemma \ref{lemma:gaps} and Theorem \ref{thm:HLW}. Indeed, by assumption we know $\Theta_{V_j}<Q$ in $B_{\delta}(y)$ for all $j$ sufficiently large, and thus by Lemma \ref{lemma:gaps} we know $\dim_\H(\sing_*(V_j)\cap B_{\delta}(y))\leq n-7$. Thus, since $V_j\to \BC^{(0)}$ as varifolds in $B_1$, by Theorem \ref{thm:HLW} we would need $\BC^{(0)}\in \mathcal{P}$, a contradiction. (Notice that the fact that $\beta_j\downarrow 0$ is actually irrelevant in this case, as all one needs is $\eps_j\downarrow 0$.)
	
	So now assume $\BC^{(0)}\in \mathcal{P}$; as $\sfrak=n-1$ and $\BC^{(0)}\preceq \BC_j$ this implies $\BC^{(0)}\in \mathcal{P}_{n-1}$. In particular, as $\BC_j\in \H_{n-1}$ we know $\BC^{(0)}\prec \BC_j$. Notice that similar to the above case, we always have $\Theta_{V_j}<Q$ in $B_{\delta}(y)$ for all $j$ sufficiently large, so again by Lemma \ref{lemma:gaps} and Theorem \ref{thm:HLW}, we can assume that in $V_j\res B_{\delta/2}(y)$ is expressible as a sum of smooth minimal graphs over suitable regions of the hyperplanes $P^{(i)}$ (cf.~Theorem \ref{thm:main-L^2-estimate}(1)). In particular, these regions include $S(\BC^{(0)})$.
	
	We now work inductively over the class $\mathfrak{C}_{i_*}$ which the $\BC_j$ belong to. In the coarsest situation relative to $\BC^{(0)}$ (which is the base case), one can perform a blow-up (cf.~Section \ref{sec:first-blow-up}) of $V_j$ relative to $\widetilde{\BC}_j$, where $\widetilde{\BC}_j\prec \BC_j$ is aligned and satisfies $E_{V_j,\widetilde{\BC}_j}\leq \frac{3}{2}\mathcal{E}_{V_j,\BC_j}$ (in particular, $\widetilde{\BC}_j$ must belong to the same level of cones as $\BC^{(0)}$). Using Lemma \ref{lemma:separation}, the fact that two-sided and one-sided excesses of $V_j$ relative to $\BC_j$ are comparable by Theorem \ref{thm:main-L^2-estimate}, and $F_{V_j,\BC_j}<\beta_j\mathcal{F}_{V_j,\BC_j}$, we see that the resulting blow-up $v$ must consist of linear functions defined on the half-hyperplanes $H^{(i)}$, yet it can \emph{not} be written as a collection of linear functions defined on the full hyperplanes $P^{(i)}$. However, for each $i$ the blow-up in the region $B_{\delta}(y)\cap P^{(i)}$ must be $C^1$ from the above application of Theorem \ref{thm:HLW}. These two properties of $v$ are in direct contradiction to one another, providing the desired contradiction.
	
	For the inductive step, we wish to derive the contradiction in the analogous manner, except now by blowing-up off a sequence of $(\tilde{\eps}_j,\tilde{\beta}_j)$-good representatives $\widetilde{\BC}_j$ of $\mathcal{F}_{V_j,\BC_j}$, for suitable sequences $\tilde{\eps}_j,\tilde{\beta}_j\downarrow 0$; these are found in an identical manner to that seen in the proof of \eqref{E:concentration-repeated}. The one caveat is whether $\widetilde{\BC}_j\in \P$ or not. If $\widetilde{\BC}_j\in \P$ for infinitely many $j$, the contradiction is reached in the same manner as above. If $\widetilde{\BC}_j\in \H_{n-1}$, then one can simply apply the present lemma inductively with $V_j,\widetilde{\BC}_j$ in place of $V_j,\BC_j$. This completes the proof.
\end{proof}

\subsection{Constructing blow-ups}\label{sec:first-blow-up}

Fix $i_*\leq i_0-1$. We construct blow-ups assuming:
\begin{itemize}
	\item Hypothesis I;
	\item $V_j, \BC_j$ satisfy Hypothesis $H(\eps_j)$ with $\eps_j \downarrow 0$;
	\item $V_j, \BC_j$ satisfy Hypothesis ($\dagger$) with $\beta_j\downarrow 0$, namely $F_{V_j,\BC_j}<\beta_j\mathcal{F}_{V_j,\BC_j}$;
	\item $\BC_j\in \mathfrak{C}_{i_*}$ for all $j$.
\end{itemize}
Depending on the exact form of the sequences, we construct blow-ups with domains being half-hyperplanes or full hyperplanes as follows.

\textbf{Case 1:} Suppose that either:
\begin{itemize}
	\item $\dim S(\BC^{(0)}) \leq n-2$; or
	\item $\BC^{(0)}\in \P_{n-1}$ and $\D\subsetneq S(\BC^{(0)})\cap B_1(0)$.
\end{itemize}
In either case $\BC^{(0)}\in \P$ and furthermore Lemma \ref{lemma:no-gaps} implies that $\BC_j\in \P$.

Thus, $\BC^{(0)} = \sum^{\pfrak_0}_{i=1}q_i|P^{(i)}|\in \P_{\sfrak_0}(\pfrak_0)$ and $\BC_j = \sum^{\pfrak_0}_{i=1}\sum^{\pfrak_i}_{k=1}q_{i,k}|P_j^{(i,k)}|\in \P_{\sfrak}(\pfrak)$ (we can assume that $q_{i,k}$ are independent of $j$ by passing to a subsequence). Taking any sequence $(\tau_j)_j$ with $\tau_j\downarrow 0$ sufficiently slowly, by Theorem \ref{thm:main-L^2-estimate} (cf.~Remark \ref{remark:aligned}) we know that (up to changing the $q_{i,k}$)
\begin{equation}\label{E:graphing-blow-up-1}
V_j \res \big(B_{1-\tau_j}^{n+1}(0)\setminus B_{\tau_j}(\mathcal{D})\big) = \sum_{i=1}^{\pfrak_0}\sum^{\pfrak_i}_{k=1}\sum^{q_{i,k}}_{s=1}|\graph(u_j^{(i,k,s)})|
\end{equation}
where $u_j^{(i,k,s)}:P_j^{(i,k)}\cap (B_{1-\tau_j}(0)\setminus B_{\tau_j}(\D))\to (P_j^{(i,k)})^\perp$ is smooth, minimal, and obeys  for any $\tau>2\tau_j$,
\begin{equation}\label{E:blow-up-estimate-1}
	\|u_j^{(i,k,s)}\|_{C^4(P_j^{(i,k)}\cap (B_{1-\tau}(0)\setminus B_{\tau}(\D)))}\leq C(n,Q,\BC^{(0)},\tau)E_{V_j,\BC_j}.
\end{equation}
Up to identifying $P^{(i,k)}_j$ with a fixed hyperplane $P_*$ (which we may do by rotating, keeping the spine fixed) we can view each $u_j^{(i,k,s)}$ as a function $P_*\cap (B_{1-\tau_j}(0)\setminus B_{\tau_j}(\D))\to P_*^\perp$. We may extend $u_j^{(i,k,s)}$ by $0$ to $P_*\cap B_1$. Then, by the $L^2$ non-concentration (Lemma \ref{lemma:non-con}) and Arzelà--Ascoli, by \eqref{E:blow-up-estimate-1} we get that, up to passing to a subsequence, for each $i,k,s$ there exists a function $v_\infty^{(i,k,s)}:P_*\cap B_1\to P_*^\perp$ with
$$\lim_{j\to\infty}E_{V_j,\BC_j}^{-1} u_j^{(i,k,s)} =: v_\infty^{(i,k,s)}$$
where the convergence is in $C^3_{\text{loc}}(P_*\cap B_1\setminus \D;P_*^\perp)$ and in $L^2_{\text{loc}}(P_*\cap B_1)$. Notice also that, as $\mathcal{D}$ is closed, our assumptions give that $P_*\setminus \D$ is connected.

\textbf{Case 2:} If the assumptions of Case 1 fail, then either:
\begin{itemize}
	\item $\BC^{(0)}\in \H_{n-1}$; or
	\item $\BC^{(0)} \in \P_{n-1}$ and $\D = S(\BC^{(0)})\cap B_1(0)$.
\end{itemize}
This is equivalent to saying that $\BC^{(0)}\in \CC_{n-1}$ and $\D = S(\BC^{(0)})\cap B_1(0)$ (the latter is guaranteed if $\BC^{(0)}\in \H_{n-1}$ by Lemma \ref{lemma:no-gaps}). In this situation, we know $\BC_j\in \CC_{n-1}$.

To unify the two cases in this situation, and because the best we can do is graph over half-hyperplanes, we will always use half-hyperplane notation. Thus, we write $\BC^{(0)} = \sum^{\pfrak_0}_{i=1}q_i|H^{(i)}|$ and $\BC_j = \sum^{\pfrak_0}_{i=1}\sum^{\pfrak_i}_{k=1}q_{i,k}|H_j^{(i,k)}|$, where as usual $q_i = \sum^{\pfrak_i}_{k=1}q_{i,k}$ and for each $j,k$, $H^{(i)}$ is the closest half-hyperplane in $\BC^{(0)}$ to $H^{(i,k)}_j$. As in Case 1, taking any sequence $(\tau_j)_j$ with $\tau_j\downarrow 0$ sufficiently slowly, by Theorem \ref{thm:main-L^2-estimate} we have (up to changing $q_{i,k}$)
$$V_j \res \big(B_{1-\tau_j}^{n+1}(0)\setminus B_{\tau_j}(\mathcal{D})\big) = \sum_{i=1}^{\pfrak_0}\sum^{\pfrak_i}_{k=1}\sum^{q_{i,k}}_{s=1}|\graph(u_j^{(i,k,s)})|,$$
where $u_j^{(i,k,s)}:H_j^{(i,k)}\cap (B_{1-\tau_j}(0)\setminus B_{\tau_j}(\D))\to (P_j^{(i,k)})^\perp$ is smooth, minimal, and obeys for any $\tau>2\tau_j$,
\begin{equation}\label{E:blow-up-estimate-2}
	\|u_j^{(i,k,s)}\|_{C^4(H_j^{(i,k)}\cap (B_{1-\tau}(0)\setminus B_{\tau}(\D)))}\leq C(n,Q,\BC^{(0)},\tau)E_{V_j,\BC_j}.
\end{equation}
Here, $P_j^{(i,k)}$ is the hyperplane extending $H_j^{(i,k)}$. Similarly to Case 1, we can remove the dependence of the domain on $j$. Thus, rotating each half-hyperplane $H^{(i,k)}_j$ to a fixed half-hyperplane $H_*$ and extending the functions $u_j^{(i,k,s)}$ by $0$ to $H_*\cap B_1$, we may view all the functions as functions $H_*\cap B_1\to P_*^\perp$, for $P_*$ the full hyperplane extending $H_*$. Then again by Lemma \ref{lemma:non-con} and Arzelà--Ascoli (using \eqref{E:blow-up-estimate-2}), up to passing to a subsequence we have
$$\lim_{j\to\infty}E_{V_j,\BC_j}^{-1}u_j^{(i,k,s)}=: v_\infty^{(i,k,s)}$$
exists for each $(i,k,s)$, where the convergence is in $C^3_{\text{loc}}(H_* \cap B_1(0)\setminus S(\BC^{(0)}))$ and in $L^2(\overline{H}_*\cap B_1)$.
\begin{defn}
	In both Case 1 \& Case 2, we say that $v_\infty:= (v_\infty^{(i,k,s)})_{i,k,s}$ is the \emph{blow-up} of $V_j$ relative to $\BC_j$. We write $\mathfrak{B}_{i_*}$ for the set of all such blow-ups.
\end{defn}
To unify notation between the two cases, we denote the domain of $v_\infty^{(i,k,s)}$ in each case, before rotating, by $\Omega^{(i,k)}$. In Case 1 we therefore have $\Omega^{(i,k)} = P^{(i)}\cap B_1(0)\setminus \D$ whilst in Case 2 we have $\Omega^{(i,k)} = H^{(i)}\cap B_1(0)\setminus S(\BC^{(0)})$. We also write $\Omega_*$ for $P_*$ or $H_*$ depending on the case.

We stress that it will be important to keep track of the tuples of functions which arise from a given (half-)hyperplane, namely the families $(v^{(i,k,s)}_\infty)_{s}$ for fixed $i,k$.

An important subclass of blow-ups which we will need to consider are those generated by blowing-up a sequence of \emph{cones}. These necessarily have to be at least as fine as $\BC_j$, so will belong to some class $\mathfrak{C}_i$ with $i\geq i_*$.

\begin{defn}\label{defn:frak-L}
	Denote by $\mathfrak{L}_{i}\subset\FB_{i_*}$ the set of all blow-ups for which $V_j\in \mathfrak{C}_i$ for all $j$. Also write $\mathfrak{L} := \cup_{i\geq i_*}\mathfrak{L}_i$.
\end{defn}

We introduce two useful pieces of notation for later.
\begin{itemize}
	\item Let $\ell_j^{(i,k)}:P^{(i)}\to (P^{(i)})^\perp$ be the linear function whose graph is $P_j^{(i,k)}$. We define $m^{(i,k)}:= \lim_{j\to\infty} E_{V_j,\BC^{(0)}}^{-1}D\ell_j^{(i,k)}$ for the limiting gradient of the hyperplanes $P_j^{(i,k)}$ relative to $P^{(i)}$; this limit exists by \eqref{E:P-to-C}, as by Theorem \ref{thm:main-L^2-estimate} we know $q_{i,k}\geq 1$ for each $i,k$, and so a simple triangle inequality argument gives $\dist_\H(\BC_j\cap B_1,\BC^{(0)}\cap B_1)\leq CE_{V_j,\BC^{(0)}}$. Notice that because $P^{(i)}$ and $P_j^{(i,k)}$ intersect along an $(n-1)$-dimensional subspace, we may identify $m^{(i,k)}$ with a real number. Similarly, one can do the same in the half-hyperplane situation in Case 2.
	\item For each $z\in \D\cap B_1$, we can choose a sequence $z_j\in \D_j$ with $z_j\to z$. For each $i,k$, define $S^{(i,k)}_j:= P^{(i)}\cap P_j^{(i,k)}$, and set $\xi_j^{(i,k)}:= \pi^\perp_{S_j^{(i,k)}}(z_j)$. Then, Corollary \ref{cor:main-L^2-estimate}(1) gives that
	$$E_{V_j,\BC_j}^{-1}\pi^\perp_{P^{(i)}}(\xi_j^{(i,k)})\to \lambda^{(i,k)}(z)^{\perp_{P^{(i)}}}\qquad \text{and}\qquad E_{V_j,\BC_j}^{-1}E_{V_j,\BC^{(0)}}\pi_{P^{(i)}}(\xi_j^{(i,k)}) \to \lambda^{(i,k)}(z)^{\top_{P^{(i)}}}$$
	for some constant $\lambda^{(i,k)}(z) \in (S^{(i,k)})^\perp$, where $S^{(i,k)} = \lim_{j\to\infty} S^{(i,k)}_j$. In particular,
	\begin{equation}\label{E:tilde-lambda}
	E_{V_j,\BC_j}^{-1}\pi^\perp_{P^{(i,k)}_j}(\xi_j^{(i,k)}) \to \lambda^{(i,k)}(z)^{\perp_{P^{(i)}}} - m^{(i,k)}\lambda^{(i,k)}(z)^{\top_{P^{(i)}}} =: \widetilde{\lambda}^{(i,k)}.
	\end{equation}
	Notice that since $\pi^\perp_{P^{(i)}}\circ \pi^\perp_{S^{(i,k)}_j}\equiv \pi^{\perp}_{P^{(i)}}$, in fact $\lambda^{(i,k)}(z)^{\perp_{P^{(i)}}}$ is independent of $k$ as it is simply $\lim_{j\to\infty}E^{-1}_{V_j,\BC_j}\pi^\perp_{P^{(i)}}(z_j)$. Furthermore, when $\sfrak=n-1$, then $S_j^{(i,k)} \equiv S(\BC_j)$ is independent of $i,k$, and in Corollary \ref{cor:main-L^2-estimate} we simply get $|\xi_j|\leq CE_{V_j,\BC_j}$, where $\xi_j:= \pi_{S(\BC_j)}^\perp(z_j)$.
\end{itemize}

\begin{remark}[Comparison with \cite{BK17} and \cite{BKMW25}]
	In \cite{BK17} and \cite[Theorem D]{BKMW25}, one enlarges the blow-up class by allowing an additional flexibility, namely that there is a \emph{second} sequence of cones $(\widetilde{\BC}_j)_j$ where each $\widetilde{\BC}_j$ is a small perturbation of the cone $\BC_j$ by an amount which is controlled by $E_{V_j,\BC_j}$, \emph{and thus could be finer} (in \cite{BK17} and \cite[Theorem D]{BKMW25}, this would mean the spine dimension decreases), yet one does \emph{not} assume that $V_j,\widetilde{\BC}_j$ obey the analogue of Hypothesis $(\dagger)$ with some constant $\widetilde{\beta}_j$. This additional flexibility ultimately allows one to directly ``integrate out'' homogeneous degree one parts of the blow-up, just by perturbing the cone, and directly prove an excess decay lemma. One is not able to do the same in the present setting precisely because the cone could have multiplicity on some (half-)hyperplanes, and in order to ``integrate'' the homogeneous degree one parts of the blow-up, one needs the flexibility to independently rotate the different parts of a (half-)hyperplane with multiplicity, but our assumptions do not allow one to graph over such a `split' cone without some assumption like Hypothesis $(\dagger)$ being satisfied. As such, multiplicity forces us to use an approach based on proving a fine $\eps$-regularity theorem, such as in \cite{Min21a} (and also in the proofs of the main planar $\eps$-regularity theorems in \cite{MW24, BKMW25}). What we are able to do is integrate out the homogeneous degree one parts of the blow-ups which correspond to perturbing the sequences of cones $\BC_j$ \emph{whilst remaining in the same class of cones} (see Theorem \ref{thm:blow-up-properties}(4) below). Since these still satisfy the assumptions used in the present section, one does not need to enlarge the classes $\mathfrak{B}_{i_*}$ any further.
\end{remark}

\subsection{Properties of the blow-up class}\label{sec:blow-up-properties}

In this section we give the basic properties of blow-ups. For simplicity, we denote elements of $\FB_{\ell}$ by $v = (v^{(i,k,s)})_{i,k,s}$, which suppresses the exact index sets, namely the parameters $\pfrak_i,q_{i,k}$, used for the blow-up (we stress these are not constant in $\FB_\ell$). We write $\|v\|_{L^2(\Omega_*)} := \sqrt{\sum_{i,k,s}\|v^{(i,k,s)}\|_{L^2(\Omega_*)}^2}$. Furthermore, for each $(i,k,s)$ recall that $S_j^{(i,k)}:= P^{(i)}\cap P_j^{(i,k)}$ and $S_j^{(i,k)}\to S^{(i,k)}$. Given $\kappa\in S^\perp$ set $\kappa^{(i,k)}:= \pi^\perp_{S^{(i,k)}}(\kappa)$.

The main theorem of this section is the following. Here, we fix $\ell\geq i^{(0)}$.
\begin{theorem}\label{thm:blow-up-properties}
	The class $\FB_\ell$ obeys the following properties.
	\begin{enumerate}
		\item [(1)] If $v \in \mathfrak{B}_\ell$, then $v^{(i,k,s)}\in C^3(\Omega_*;P_*^\perp)$ is harmonic for all $i,k,s$.
		\item [(2)] $\FB_\ell$ is closed under translations parallel to $S$ and rescalings of the domain. More precisely, if $v\in \FB_\ell$, $y\in S\cap B_{1/2}(0)$, and $\rho\in (0,1/4)$, then provided $v\not\equiv 0$ in $B_\rho(y)$, we have $v_{y,\rho}\in \FB_\ell$, where
		$$v_{y,\rho}^{(i,k,s)} := \|v(y+\rho(\cdot))\|_{L^2(B_1)}^{-1}v^{(i,k,s)}(y+\rho(\cdot)).$$
		\item [(3)] $\FB_\ell$ is closed under translations orthogonal to $S$. More precisely, if $v\in \FB_\ell$ and $\kappa\in S^\perp$ with $v^{(i,k,s)} \not\equiv(\kappa^{(i,k)})^{\perp_{P^{(i)}}}-m^{(i,k)}(\kappa^{(i,k)})^{\top_{P^{(i)}}} := \widetilde{\kappa}^{(i,k)}$ for some $i,k,s$, then we have $\hat{v}\in \FB_\ell$, where
		$$\hat{v}^{(i,k,s)} := \frac{v^{(i,k,s)} - \widetilde{\kappa}^{(i,k)}}{\left(\sum_{i^\prime,k^\prime,s^\prime}\|v^{(i^\prime,k^\prime,s^\prime)} - \widetilde{\kappa}^{(i^\prime,k^\prime)}\|_{L^2(\Omega_*)}^2\right)^{1/2}}.$$
		\item [(4)] $\FB_\ell$ is closed under subtraction of linear functions produced by blowing up cones of level $\ell$. More precisely, suppose $v\in \mathfrak{B}_\ell$ is the blow-up of $(V_j)_j$ relative to $(\BC_j)_j\subset\mathfrak{C}_\ell$, and suppose $(\widetilde{\BC}_j)_j$ is another sequences of cones with:
	\begin{itemize}
		\item $\widetilde{\BC}_j\in \mathfrak{C}_\ell$ for all $j$; 
		\item $\dist_\H(\BC_j\res B_1,\widetilde{\BC}_j\res B_1)\leq CE_{V_j,\BC_j}$, for some $C = C(n,Q,\BC^{(0)})\in (0,\infty)$.
	\end{itemize}
	Let $\pi = (\pi^{(i,k,s)})_{i,k,s}\subset\mathfrak{L}_\ell$ be the blow-up of $\widetilde{\BC}_j$ relative to $\BC_j$; note that for each $i,k$ we have $\pi^{(i,k,s)}\equiv \pi^{(i,k)}$ is independent of $s$. Then, provided we have $v^{(i,k,s)}\not\equiv \pi^{(i,k)}$ for some $i,k,s$, we have $\bar{v}\in\mathfrak{B}_\ell$, where
	$$\bar{v}^{(i,k,s)} := \frac{v^{(i,k,s)}-\pi^{(i,k)}}{\left(\sum_{i^\prime,k^\prime,s^\prime}\|v^{(i^\prime,k^\prime,s^\prime)} - \pi^{(i^\prime,k^\prime)}\|^2_{L^2(\Omega_*)}\right)^{1/2}}.$$
	\item [(5)] $\FB_\ell$ is compact in $C^3_{\textnormal{loc}}(B_1(0)\setminus \D)$. More precisely, suppose $(v_j)_j\subset \FB_\ell$, and pass to a (unrelabelled) subsequence to ensure that the parameter spaces where the indices $(i,k,s)$ lie for each $v_j = (v_j^{(i,k,s)})_{i,k,s}$ are constant in $j$. Then, there exists a subsequence $\{v_{j^\prime}\}_{j^\prime}$ and $v\in \FB_\ell$ (with the same parameter spaces for its indices as $v_j$) such that $v_{j^\prime}^{(i,k,s)} \to v^{(i,k,s)}$ in $C^3_{\textnormal{loc}}(\Omega_*)$ and in $L^2_{\text{loc}}(\overline{\Omega}_*\cap B_1)$.
	\item [(6)] \textnormal{(Hardt--Simon Inequality for $z\in \D$.)} Let $v\in \mathfrak{B}_\ell$ and $z\in \D\cap B_{1/2}(0)$ and $\rho \in (0,1/4)$. Then we have
	\begin{align*}
	\sum_{i,k,s}\int_{\Omega_*\cap B_{\rho/2}(z)}&R_z(x)^{2-n}\left|\frac{\del}{\del R_z}\left(\frac{v^{(i,k,s)} - \widetilde{\lambda}^{(i,k)}}{R_z(x)}\right)\right|^2\, \ext x  \leq C\sum_{i,k,s}\int_{\Omega_*\cap B_\rho(z)}|v^{(i,k,s)} - \widetilde{\lambda}^{(i,k)}|^2,
	\end{align*}
	where $R_z(x):= |x-z|$ and $C = C(n,Q,\BC^{(0)})\in (0,\infty)$. Here, $\widetilde{\lambda}^{(i,k)}$ is in \eqref{E:tilde-lambda}.
	\end{enumerate}
\end{theorem}

\begin{remark}
	In Theorem \ref{thm:blow-up-properties}(4), the linear function which is subtracted from the $v^{(i,k,s)}$ must be the same for fixed $i,k$; this corresponds to \emph{not} splitting the multiplicity of the (half-)hyperplanes in the cones along the blow-up sequence.
\end{remark}

\begin{proof}
	Property (1) is immediate from the blow-up procedure, as one is linearising the minimal surface equation. Properties (2) and (3) follow from modifying the blow-up procedure by an appropriate translations and rescalings of the sequence $V_j$ generating it and determining how the blow-up changes, in much the same way as related works (cf.~\cite[Theorem 4.3]{BK17}, \cite[Section 12]{BKMW25}); as such we do not repeat the details. Property (4) follows by blowing $V_j$ up over the sequence $\widetilde{\BC}_j$ instead. Indeed, one can check Hypothesis $H(\tilde{\eps}_j)$ holds for $V_j$ and $\widetilde{\BC}_j$ for suitable $\tilde{\eps}_j\downarrow 0$, and Hypothesis $(\dagger)$ holds for an appropriate sequence $\tilde{\beta}_j\downarrow 0$ by using the assumption on the Hausdorff distance between $\BC_j$ and $\widetilde{\BC}_j$. Indeed, each (half-)hyperplane in $\widetilde{\BC}_j$ can be written as the graph of a linear function $\tilde{\pi}^{(i,k)}_j$ over the corresponding (half-)hyperplane in $\BC_j$ (namely, the closest one to it) and $\|\tilde{\pi}^{(i,k)}\|_{C^1}\leq CE_{V_j,\BC_j}$, and so the blow-up of $\widetilde{\BC}_j$ relative to $\BC_j$ is simply $\pi^{(i,k)}:= \lim_{j\to\infty}E_{V_j,\BC_j}^{-1}\tilde{\pi}^{(i,k)}_j$; in particular, the blow-up is independent of the last index $s$. Property (5) is a simple compactness result based on a diagonal argument and on the uniform bounds one has when blowing-up. Finally, Property (6) follows from blowing-up Theorem \ref{thm:main-L^2-estimate}(5) with $(\eta_{z_j,\rho})_\#V_j$ in place of $V_j$, where $z_j\to z$ obey $z_j\in\D_j$; one needs to verify that the assumptions of Theorem \ref{thm:main-L^2-estimate} hold for this modified sequence, but again this is by now standard (this was also needed in verifying Property (2), for instance), and we refer the reader to similar arguments in other places (e.g.~\cite[Proposition 10.5]{Wic14}).
\end{proof}

We stress that Theorem \ref{thm:blow-up-properties}(4) only allows us to subtract off linear functions arising from blow-ups of cones of the same level as the blow-up class $\mathfrak{B}_\ell$; it does not allow us to freely subtract different linear functions for a given $i,k$, and it also does not allow us to subtract off linear functions which do not arise from the blow-up of a cone. In order to study regularity properties of blow-ups in Section \ref{sec:reg}, we will need to study a situation where an element $v\in \mathfrak{B}_{\ell}$ is sufficiently close to a blow-up which arises from blowing-up a sequence of \emph{finer} cones, i.e.~cones with level $>\ell$. This is where we will use our inductive fine $\eps$-regularity theorem. Notice that for the finest class of cones, where every (half-)hyperplane has multiplicity one, Theorem \ref{thm:blow-up-properties}(4) already gives us all the desired freedom.

\section{Proof of the Main Estimates}\label{sec:main-estimates}

In this section we prove Theorem \ref{thm:main-L^2-estimate} and Corollary \ref{cor:main-L^2-estimate}. We will again assume the following inductive hypothesis:

\hspace{2em}\textbf{Hypothesis I:} Theorem \ref{thm:main-L^2-estimate} (and thus Corollary \ref{cor:main-L^2-estimate}) holds for any $\BC\in \cup^{i_0-1}_{j=1}\mathfrak{C}_j$.

As we saw in Section \ref{sec:blow-up}, Hypothesis I allows us to perform blow-ups over cones in $\cup^{i_0-1}_{j=1}\mathfrak{C}_j$.

\subsection{Concentration of good density points}\label{sec:good-density-points}

We first prove the concentration of good density points to $S(\BC)$, as claimed in Theorem \ref{thm:main-L^2-estimate}(1). Namely, we show
\begin{equation}\label{E:concentration-repeated}
	\{x\in B^{n+1}_{31/32}(0):\Theta_V(x)\geq Q\}\subset B_\tau(S(\BC)).
\end{equation}
\begin{proof}[Proof of \eqref{E:concentration-repeated}]
	Suppose for contradiction that \eqref{E:concentration-repeated} were false. Then, we could find $\tau_0>0$ and sequences $V_j,\BC_j$ which satisfy the assumptions of Theorem \ref{thm:main-L^2-estimate} for some sequences $\eps_j,\beta_j\downarrow 0$, i.e.
	\begin{itemize}
		\item Hypothesis $H(\eps_j)$ holds for $V_j$ and $\BC_j$;
		\item $F_{V_j,\BC_j}\leq\beta_j\mathcal{F}_{V_j,\BC_j}$;
	\end{itemize}
	yet we have
	$$\{x\in B^{n+1}_{31/32}(0):\Theta_{V_j}(x)\geq Q\}\not\subset B_{\tau_0}(S(\BC_j)).$$
	By passing to a subsequence, we may assume $V_j\weakly V_\infty$ for some stationary integral varifold $V_\infty$, $\dim S(\BC_j)\equiv \sfrak$ is independent of $j$, and that $S(\BC_j)$ converges to an $\sfrak$-dimensional subspace, which we denote by $S_\infty$, in Hausdorff distance in $B_1$. As a result of Hypothesis $H(\eps_j)$, we know that $\BC_j\weakly \BC^{(0)}$, and thus $S_\infty\subseteq S(\BC^{(0)})$. Also from Hypothesis $H(\eps_j)$ we know $\spt\|V_\infty\|\cap B_1 = \spt\|\BC^{(0)}\|\cap B_1$.
	
	Our contradiction assumption gives that there exist $Z_j\in \spt\|V_j\|\cap B^{n+1}_{31/32}(0)$ with $\Theta_{V_j}(Z_j)\geq Q$ yet $Z_j\not\in B_{\tau_0}(S(\BC_j))$. By passing to another subsequence, we may assume $Z_j\to Z_\infty\in B^{n+1}_{31/32}(0)\cap (S(\BC^{(0)})\setminus B_{\tau_0/2}(S_\infty))$. The fact that $Z_\infty$ must lie in $S(\BC^{(0)})$ follows from the fact that $V_\infty$ is stationary and has the same support as $\BC^{(0)}$. Indeed, if $Z_\infty\not\in S(\BC^{(0)})$, then $V_\infty$ contains a point of density $\geq Q$ away from $S(\BC^{(0)})$, which from the stationary, mass upper bound, and form of $\BC^{(0)}$ would imply that $V_\infty$ is a multiplicity $Q$ hyperplane, which contradicts the form of $\BC^{(0)}$. 
	
	Coincidentally, notice that if $\sfrak = \dim(S(\BC^{(0)}))$, then necessarily we have $S_\infty = S(\BC^{(0)})$, and so we have already reached the desired contradiction based on looking at where the limit point $Z_\infty$ lies. Thus the proof is complete in this case.
	
	We may therefore assume from now on that $s<\dim(S(\BC^{(0)}))$. In particular, $\sfrak\leq n-2$ and so $\BC_j\in\mathcal{P}$. Hence, as $\BC^{(0)}\preceq \BC_j$ (in fact $\BC^{(0)}\prec \BC_j$) we also have $\BC^{(0)}\in \P$.

	We now claim that, up to passing to a subsequence, there exist sequences $\tilde{\eps}_j,\tilde{\beta}_j\downarrow 0$ and $(\tilde{\eps}_j,\tilde{\beta}_j)$-good representatives $\widetilde{\BC}_j$ for $\mathcal{F}_{V_j,\BC_j}$, where the constant $C$ therein the definition only depends on $n,Q,\BC^{(0)}$ (this is key in \eqref{E:cone-rotation-3} to ensure separation in the blow-up we take, and is why we cannot simply apply Lemma \ref{lemma:good-representative} for an arbitrary sequence $\tilde{\beta}_j$). Indeed, the process for finding such $\widetilde{\BC}_j$ is analogous to that seen in the proof of Lemma \ref{lemma:good-representative}  (cf.~Lemma \ref{lemma:replacement}) but without specifying the constants $\eps,\beta$ therein. First, for each $j$ choose $\Dbf^{(1)}_j\prec \BC_j$ aligned such that
	$$F_{V_j,\Dbf^{(1)}_j} \leq \frac{3}{2}\mathcal{F}_{V_j,\BC_j}.$$
	If $\mathcal{F}_{V_{j^\prime},\Dbf^{(1)}_{j^\prime}}^{-1}F_{V_{j^\prime},\Dbf^{(1)}_{j^\prime}}\to 0$ for some subsequence $(j^\prime)$ of $(j)$, then we can set $\widetilde{\BC}_j:= \Dbf_j^{(1)}$ and pass to the subsequence $(j^\prime)$ and we have found the claimed sequence $(\widetilde{\BC}_j)_j$. Otherwise, for all sufficiently large $j$ we have
	$$F_{V_j,\Dbf_j^{(1)}} \geq C\mathcal{F}_{V_j,\Dbf_j^{(1)}}$$
	for some fixed $C = C(n,Q,\BC^{(0)})$. Now, we repeat: given $\Dbf_j^{(p-1)}$ for some $p\geq 2$, if $\Dbf_j^{(p-1)}$ is not the same level as $\BC^{(0)}$, we choose $\Dbf_j^{(p)}\prec \Dbf_j^{(p-1)}$ aligned such that
	$$F_{V_j,\Dbf_j^{(p)}} \leq \frac{3}{2}\mathcal{F}_{V_j,\Dbf_j^{(p-1)}}.$$
    Then if $\mathcal{F}_{V_{j^\prime},\Dbf_{j^\prime}^{(p)}}^{-1}F_{V_{j^\prime},\Dbf_{j^\prime}^{(p)}}\to 0$ for some subsequence $(j^\prime)$ of $(j)$, then we can pass to the subsequence $(j^\prime)$ and set $\widetilde{\BC}_j:= \Dbf_j^{(p)}$; notice that we must have $F_{V_j,\widetilde{\BC}_j} \leq C\mathcal{F}_{V_j,\BC_j}$ for some constant $C = C(n,Q,\BC^{(0)})$ by construction. Otherwise, we must have for all $j$ sufficiently large,
	$$F_{V_j,\Dbf_j^{(p)}}\geq C\mathcal{F}_{V_j,\Dbf_j^{(p)}}$$
	for some fixed $C = C(n,Q,\BC^{(0)})\in (0,\infty)$. We repeat this process until either $\Dbf_j^{(p)}$ is the same level as $\BC^{(0)}$ or $\mathcal{F}_{V_{j^\prime},\Dbf_{j^\prime}^{(p)}}F_{V_{j^\prime},\Dbf_{j^\prime}^{(p)}}\to 0$ for some subsequence $(j^\prime)$; in the latter case we have already defined $\widetilde{\BC}_j$, and in the form case we can simply set $\widetilde{\BC}_j:= \Dbf_j^{(p)}$. In either case, $\widetilde{\BC}_j$ is the claimed $(\tilde{\eps}_j,\tilde{\beta}_j)$-replacement of $\BC_j$ for $V_j$.

	Passing to another subsequence, we may assume $\widetilde{\BC}_j\in \mathcal{P}_{\tilde{\sfrak}}(\tilde{\pfrak})$ for suitable $\tilde{\sfrak},\tilde{\pfrak}$ independent of $j$. Note that if $\widetilde{\sfrak} = \sfrak$, the proof of this claim is in fact complete by applying its inductive validity to (suitable rotations of) $V_j$ and $\widetilde{\BC}_j$, so we can assume $\tilde{\sfrak}>\sfrak$, and so $S(\BC_j)\subsetneq S(\widetilde{\BC}_j)$. We may also assume that $S(\widetilde{\BC}_j)\to \widetilde{S}$, where $S_\infty\subsetneq\widetilde{S}\subseteq S(\BC^{(0)})$.
	
	Next we modify the sequence of cones $\widetilde{\BC}_j$ by applying rotations to ensure that the $Z_j$ are actually contained with $S(\widetilde{\BC}_j)$. Indeed, for each $j$ choose a rotation $R_j:\R^{n+1}\to \R^{n+1}$ which minimises $\|R-\id\|$ over all such rotations $R$ satisfying
	$$R|_{S(\BC_j)} = \id_{S(\BC_j)} \quad \text{and}\quad R(Z_j)\in S(\widetilde{\BC}_j).$$
	Our aim now is to construct a blow-up of $(R_j)_\#V_j$ relative to $\widetilde{\BC}_j$, as described in Section \ref{sec:blow-up}. Notice that $\widetilde{\BC}_j$ is coarser than $\BC_j$, and so we are in the realm of our inductive assumption (Hypothesis I).
	
	First, note that by elementary geometry (cf.~\cite[(6.12)--(6.16), (3.4)--(3.5)]{BK17}) we have
	\begin{align*}
		\dist_\H(\widetilde{\BC}_j\res B_1, (R_j)_\#\widetilde{\BC}_j \res B_1) & \leq C\dist_\H(\widetilde{\BC}_j\res B_1, (\tau_{Z_j})_\#\widetilde{\BC}_j\res B_1)\\
		& \leq C\max_{1\leq i \leq \pfrak_0 }\big\{|\pi_{P^{(i)}}^\perp(\xi_j)| + E_{V_j,\BC^{(0)}}|\pi_{P^{(i)}}(\xi_j)|\big\},
	\end{align*}
	where $\xi_j:= \pi^\perp_{S(\widetilde{\BC}_j)}(Z_j)$ and $C = C(n,Q,\BC^{(0)})>0$. Note also that $V_j,\widetilde{\BC}_j$ satisfy the assumptions of Theorem \ref{thm:main-L^2-estimate}, and thus Corollary \ref{cor:main-L^2-estimate}, with $\tilde{\eps}_j$ and $\tilde{\beta}_j$ replacing $\eps_0,\beta_0$ therein, and our inductive assumption therefore allows us to apply them. In particular, we may use Corollary \ref{cor:main-L^2-estimate}(1) to get
	\begin{equation}\label{E:cone-rotation-1}
		\dist_\H(\widetilde{\BC}_j\res B_1,(R_j)_\#\widetilde{\BC}_j\res B_1)\leq CE_{V_j,\widetilde{\BC}_j}.
	\end{equation}
	In particular, we know that $(R_j)_\#\widetilde{\BC}_j\weakly \BC^{(0)}$.
	
	We also now choose a second rotation $\Gamma_j$ which aligns $S(\widetilde{\BC}_j)$ with $S(\BC^{(0)})$. Indeed, choose a rotation $\Gamma_j:\R^{n+1}\to \R^{n+1}$ which minimises $\|\Gamma-\id\|$ over all rotations $\Gamma$ satisfying
	$$\Gamma|_{S(\BC_j)}=\id_{S(\BC_j)} \quad \text{and} \quad S(\Gamma_\#\widetilde{\BC}_j) \subset S(\BC^{(0)}).$$
	We now blow-up $(\Gamma_j\circ R_j)_\#V_j$ off $(\Gamma_j)_\#\widetilde{\BC}_j$. Of course, one must first verify that these sequences obey the necessary assumptions from Section \ref{sec:first-blow-up}; this is done analogously to similar situations, using the inductive information provided by Theorem \ref{thm:main-L^2-estimate} and Corollary \ref{cor:main-L^2-estimate} (cf.~\cite[Step 3 in Proof of Theorem 7.3]{BKMW25}). For simplicity of notation (as there is no significant change here in the argument if we rotate $\BC^{(0)}$ by a small amount) let us suppose $\Gamma_j = \id$, and so we drop it from our notation.
	
	To recall the notation from the blow-up procedure in Section \ref{sec:first-blow-up}, write $\widetilde{\BC}_j:= \sum^{\pfrak_0}_{i=1}\sum^{\pfrak_i}_{k=1}q_{i,k}\widetilde{P}_j^{(i,k)}$, where $\widetilde{P}_j^{(i,k)}$ are distinct planes with the closest plane in $\BC^{(0)}$ to $\widetilde{P}^{(i,k)}$ is $P^{(i)}$. As we may apply Theorem \ref{thm:main-L^2-estimate} to $V_j$ and $\widetilde{\BC}_j$, taking $\tau_j\downarrow 0$ sufficiently slowly, up to a slight change of domain we can represent $(R_j)_\#V_j$ as graphs over $\widetilde{\BC}_j$. Thus, for all $j$ sufficiently large, we have
	$$(R_j)_\#V_j\res (B^{n+1}_{63/64}(0)\setminus B_{\tau_j}(S(\widetilde{\BC}_j))) = \sum^{\pfrak_0}_{i=1}\sum^{\pfrak_i}_{k=1}\sum^{q_{i,k}}_{s=1}|\graph(\widetilde{u}_j^{(i,k,s)})|,$$
	where $\widetilde{u}_j^{(i,k,s)}:\widetilde{P}_j^{(i,k)}\cap (B^{n+1}_{63/64}(0)\setminus B_{\tau_j/2}(S(\widetilde{\BC}_j)))\to (\widetilde{P}_j^{(i,k)})^\perp$, and for any $\tau>2\tau_j$ satisfies the estimate
	$$\|\widetilde{u}_j^{(i,k,s)}\|_{C^4(B^{n+1}_{125/128}(0)\setminus B_{\tau}(S(\widetilde{\BC}_j)))}\leq C(n,Q,\BC^{(0)},\tau)E_{V_j,\widetilde{\BC}_j}.$$
	We may therefore construct a blow-up via the limit of $v_j^{(i,k,s)}:= E_{V_j,\widetilde{\BC}_j}^{-1}\widetilde{u}_j^{(i,k,s)}$, namely
	$$v_\infty^{(i,k,s)} := \lim_{j\to\infty}v_j^{(i,k,s)},$$
	where the convergence is in $C^3_{\text{loc}}$ away from $S(\widetilde{\BC}_j)$ and strongly in $L^2$ by Lemma \ref{lemma:non-con}. 
	Now we wish to blow-up $(R_j)_\#\BC_j$ off $\widetilde{\BC}_j$ in the same manner. Indeed, we may write $(R_j)_\#\BC_j = \sum^{\pfrak_0}_{i=1}\sum^{\pfrak_i}_{k=1}\sum^{\pfrak_{i,k}}_{s=1}q_{i,k,s}|R_j(P^{(i,k,s)}_j)|$, where each $P_j^{(i,k,s)}$ is distinct and has its closest plane in $\widetilde{\BC}_j$ being $\widetilde{P}_j^{(i,k)}$; in particular, $\lim_{j\to\infty}R_j(P_j^{(i,k,s)}) = P^{(i)}$. Now write $\ell_j^{(i,k,s)}:\widetilde{P}_j^{(i,k)}\to (\widetilde{P}_j^{(i,k)})^\perp$ for the linear function whose graph coincides with $P_j^{(i,k,s)}$, and define the corresponding rescaled function by
	$$\overline{\ell}_j^{(i,k,s)}:= E_{V_j,\widetilde{\BC}_j}^{-1}\ell_j^{(i,k,s)}.$$
	We would like to now show that $\overline{\ell}_j^{(i,k,s)}$ is bounded, so we may pass it to a limit. This follows because the bound $F_{V_j,\BC_j}\leq \beta_j \mathcal{F}_{V_j,\BC_j}$ implies
	\begin{align}
    \sum^{\pfrak_0}_{i=1} \sum^{\pfrak_i}_{k=1} &\sum^{q_{i,k}}_{s=1}\int_{\widetilde{P}_j^{(i,k)}\cap B_{31/32}\setminus B_{1/16}(S(\BC^{(0)}))} \min_{1\leq s^\prime\leq \pfrak_{i,k}}\inf_{y\in \widetilde{P}_j^{(i,k)}}\big|v_j^{(i,k,s)}(x)-\overline{\ell}_j^{(i,k,s^\prime)}(y)\big|^2\, \ext x\nonumber\\
    & + \sum_{i=1}^{\pfrak_0}\sum_{k=1}^{\pfrak_i}\sum_{s=1}^{\pfrak_{i,k}}\int_{\widetilde{P}_j^{(i,k)}\cap B_{1/2}\setminus B_{1/16}(S(\BC^{(0)}))}\min_{1\leq s^\prime\leq q_{i,k}}\inf_{y\in \widetilde{P}_j^{(i,k)}}\big|\overline{\ell}_j^{(i,k,s)}(x) -v_j^{(i,k,s^\prime)}(y)\big|^2\, \ext x \leq C\beta_j^2.\label{E:cone-rotation-2}
\end{align}
Since each $v_j^{(i,k,s)}$ is bounded in $C^4$ independently of $j$ on $\widetilde{P}_j^{(i,k)}\cap B_{1/2}\setminus B_{1/16}(S(\BC^{(0)}))$, this shows that $\overline{\ell}_j^{(i,k,s)}$ is bounded uniformly in $L^2$ on all $\widetilde{P}_{j}^{(i,k)}\cap B_1$, which as $\overline{\ell}_j^{(i,k,s)}$ is linear is sufficient to show that $\overline{\ell}_j^{(i,k,s)}$ is bounded. Thus, each linear function $\overline{\ell}_j^{(i,k,s)}$ converges to a finite, linear, limit, i.e.~we have
$$\|\ell_j^{(i,k,s)]}\|_{C^1}\leq CE_{V_j,\widetilde{\BC}_j} \qquad \text{and}\qquad \lim_{j\to\infty}\overline{\ell}_j^{(i,k,s)} = \overline{\ell}_\infty^{(i,k,s)}$$
where the convergence is, say, in $C^1$ on $B_1$.
For the limit of the first integral in \eqref{E:cone-rotation-2} to be zero, we must have that for each $x\in P^{(i)}\cap B_{31/32}\setminus B_{1/16}(S(\BC^{(0)}))$ there exists $s^\prime = s^\prime(x,s)$ which satisfies
$$v^{(i,k,s)}_\infty(x) = \overline{\ell}_\infty^{(i,k,s^\prime(x,s))}(x).$$
However, as all the $\overline{\ell}_\infty^{(i,k,s)}$ are linear functions and $v^{(i,k,s)}_\infty$ are $C^1$ away from $\widetilde{S}$, we must have that $s^\prime(x,s)$ can be taken independent of $x$, i.e.~there is $s^\prime = s^\prime(s)$ for which
$$v_\infty^{(i,k,s)} \equiv \overline{\ell}_\infty^{(i,k,s^\prime(s))}.$$
Thus, as the $v_\infty^{(i,k,s)}$ are harmonic, unique continuation tells us that the same holds in all of $P^{(i)}\cap B_1$, and thus the $v_\infty^{(i,k,s)}$ are linear. With this in mind, we now investigate the information contained within the limit of the second integral of \eqref{E:cone-rotation-2} as $j\to\infty$. Indeed, as both $\overline{\ell}_j^{(i,k,s)}$ and $v_j^{(i,k,s^\prime)}$ are converging in (say) $C^3$ to linear limits, it follows that for each $(i,k,s)$ there exists $s^{\prime\prime} = s^{\prime\prime}(s)$ such that
$$\ell_\infty^{(i,k,s)} \equiv v_\infty^{(i,k,s^{\prime\prime}(s))}.$$
In particular, each plane in $\BC_j$ blows up to the same limit as one of the functions defining $V_j$.

Arguing analogously to that in Lemma \ref{lemma:separation}, except for the two-sided excess, and using that both $F_{V_j,\widetilde{\BC}_j}\leq C\mathcal{F}_{V_j,\BC_j}$ by construction and by Theorem \ref{thm:main-L^2-estimate} we know $F_{V_j,\widetilde{\BC}_j} \leq CE_{V_j,\widetilde{\BC}_j}$, we have for all sufficiently large $j$,
\begin{equation}\label{E:cone-rotation-3}
\|\ell_j^{(i,k,s)}-\ell_j^{(i,k,s^\prime)}\|_{C^1(B_1)}\geq cE_{V_j,\widetilde{\BC}_j} \qquad \text{for any }s\neq s^\prime
\end{equation}
and thus we see that $\overline{\ell}_\infty^{(i,k,s)}\not\equiv \overline{\ell}_\infty^{(i,k,s^\prime)}$ for any $s\neq s^\prime$. This means that the number of distinct planes which the blow-up creates is the same as the number of distinct planes in $\BC_j$. Furthermore, by the Hardt--Simon inequality for blow-ups and because $Z_j\in S(\widetilde{\BC}_j)$ for all $j$ (see Theorem \ref{thm:blow-up-properties}), we necessarily have
$$\overline{\ell}^{(i,k,s)}_\infty \equiv 0 \quad \text{on }S_\infty\cup \{Z_\infty\} \quad \text{for all }i,k,s.$$
As $\overline{\ell}^{(i,k,s)}_\infty$ is linear and by construction $Z_\infty\not\in S_\infty$, this implies that $\overline{\ell}^{(i,k,s)}_\infty$ vanishes on a subspace of dimension $>\sfrak$. Now construct a new cone based on this blow-up by defining $\widehat{P}_j^{(i,k,s)}:= \graph(E_{V_j,\widetilde{\BC}_j}(\overline{\ell}^{(i,k,s)}_\infty + \widetilde{L}_j^{(i,k)}))$ and
$$\widehat{\BC}_j := \sum^{\pfrak_0}_{i=1}\sum^{\pfrak_i}_{k=1}\sum^{\pfrak_{i,k}}_{s=1}q_{i,k,s}|\widehat{P}_j^{(i,k,s)}|.$$
Then $\widehat{\BC}_j\prec \BC_j$ are aligned (indeed, being coarser is seen from the increase in spine dimension), and by construction we have
$$F_{V_j,\widehat{\BC}_j}/E_{V_j,\widetilde{\BC}_j} \to 0 \qquad \text{as }j\to\infty$$
(one can argue analogously to \cite[Proof of Theorem 7.3, Step 5]{BKMW25} to get $L^2$ convergence on all of $B_1$ during the blow-up process). Since $E_{V_j,\widetilde{\BC}_j} \leq F_{V_j,\widetilde{\BC}_j} \leq C\mathcal{F}_{V_j,\BC_j} \leq CF_{V_j,\widehat{\BC}_j}$ gives a contradiction. This completes the proof.
\end{proof}

\subsection{Graphical representation}\label{sec:graphical-rep-proof}

Next we prove the second half of Theorem \ref{thm:main-L^2-estimate}(1) which claims the relevant graphical representation with estimates in terms of the excess $E_{V,\BC}$. The additional complication with this step compared to previous works is that $\BC$ need not be cylindrical, i.e.~the planes in $\BC$ can intersect away from $S(\BC)$. This requires the use of both our $\eps$-regularity theorems in gaps (as then the singular set is small), namely near planes with multiplicity and unions of hyperplanes with density $<Q$ (Theorem \ref{thm:bellettini} and Theorem \ref{thm:HLW} respectively).

Here, we write $(x,y)$ for the coordinates in the natural decomposition $S(\BC^{(0)})^\perp\times S(\BC^{(0)})\cong \R^{n+1}$. As a warm-up, we first prove a weaker version of the result by taking functions over $\BC^{(0)}$ to illustrate how one can deal with singularities away from the spine of the cone, i.e.~the specific case of the result when $\BC = \BC^{(0)}$.
\begin{lemma}[Coarse Representation]\label{lemma:coarse-rep}
	Fix $\tau_0\in (0,1)$. Then, there exist constants $\eps_0,\beta_0$ only depending on $n,Q,\BC^{(0)},\tau_0$ such that the following holds. Suppose $V\in \V_Q$ and $\BC\in \CC$ satisfy the assumptions of Theorem \ref{thm:main-L^2-estimate} with $\eps_0,\beta_0$ therein, and that $\BC = \BC^{(0)}$.
	
	Then, for each $1\leq i\leq \pfrak_0$ and $1\leq k\leq q_i$, there exist smooth functions $u^{(i,k)}:\Omega^{(i)}\cap B^{n+1}_{63/64}(0)\setminus B_{\tau_0}(S(\BC^{(0)}))\to (P^{(i)})^\perp$ which solve the minimal surface equation, where $\Omega^{(i)} = P^{(i)}$ if $\BC^{(0)}\in \P$ and $\Omega^{(i)} = H^{(i)}$ if $\BC^{(0)}\in \H_{n-1}$, such that
	\begin{itemize}
		\item $\displaystyle V\res (B^{n+1}_{63/64}(0)\setminus B_{\tau_0}(S(\BC^{(0)}))) = \sum^{\pfrak_0}_{i=1}\sum^{q_i}_{k=1}|\graph(u^{(i,k)})|\res B^{n+1}_{63/64}(0);$
		\item $\|u^{(i,k)}\|_{C^4}\leq CE_{V,\BC^{(0)}}$, where $C = C(n,Q,\BC^{(0)},\tau_0)>0$.
	\end{itemize}
	Furthermore, the analogous conclusion holds on a ball $B_{4\tau_0}(y)\subset \{\Theta_V<Q\}$ with $y\in S(\BC^{(0)})$.
\end{lemma}
For the ``Furthermore'' part of the above lemma, one should refer to the corresponding statement in Theorem \ref{thm:main-L^2-estimate} for the analogous claim.

\begin{proof}
	For notational simplicity, set $U:= B^{n+1}_{127/128}(0)\setminus B_{\tau_0/2}(S(\BC^{(0)}))$. First note that if $\dim(S(\BC^{(0)}))=n-1$ then the result is simple, as then all planes in $\BC^{(0)}$ are disjoint in $U$. The result then follows simply from Theorem \ref{thm:bellettini} and a short compactness argument (or, more precisely, from $(\S3)_Q$, as necessarily $q_i<Q$ for all $i$).
	
	Thus we may assume that $\dim(S(\BC^{(0)}))\leq n-2$, which implies $\BC^{(0)}\in \mathcal{P}$. The additional difficulty here is that the planes in $\BC^{(0)}$ intersect away from $S(\BC^{(0)})$. However, we have already shown \eqref{E:concentration-repeated}, from which it follows that $U\subset\{\Theta_V<Q\}$ if $\eps_0,\beta_0$ are sufficiently small. Thus by Lemma \ref{lemma:gaps}, we have $\dim_\H(\sing_*(V)\cap U)\leq n-7$.
	
	We would now like to apply the $\eps$-regularity theorems (Theorem \ref{thm:bellettini} and Theorem \ref{thm:HLW}) in $U$. For this purpose, it is convenient to define the function $\rho:\spt\|\BC^{(0)}\|\to \R_{> 0}\cup \{+\infty\}$ by
	$$\rho(x):=\sup\{r>0:(\eta_{x,r})_\#\BC^{(0)}\res B_1 = T_x\BC^{(0)}\res B_1\};$$
	i.e.~$\rho(x)$ is the maximum radius $\BC^{(0)}\res B_{\rho(x)}(x)$ is a sum of affine hyperplanes which each contain $x$. Observe that $\rho$ is homogeneous of degree one, $\rho(x) = +\infty$ if and only if $x\in S(\BC^{(0)})$, and, for any $\tau>0$, there is a constant $\rho_0 = \rho_0(\BC^{(0)},\tau)>0$ such that
	$$\rho|_{B_1\setminus B_\tau(S(\BC^{(0)}))} \in [\rho_0, 1].$$
	By containment of sets we have $E_{V,\BC^{(0)}}(B_{\rho_0}(x)) \leq \rho_{0}^{-n-2}E_{V,\BC^{(0)}}$ for any $B_\rho(x)\subset B_1(0)$. Thus, provided $\eps_0$ is small enough, we may apply either Theorem \ref{thm:bellettini} or Theorem \ref{thm:HLW} in each ball $B_\rho(x)$ with $x\in \spt\|\BC^{(0)}\|\cap U\setminus B_{\tau_0}(S(\BC^{(0)}))$, with $\rho = \min\{\rho_0(\BC^{(0)},\tau_0),1/256\}$. Applying a simple covering argument then gives the conclusion of the lemma on $B^{n+1}_{63/64}(0)\setminus B_{\tau_0}(S(\BC^{(0)}))$. Note that since the $u_i$ are minimal, to prove the desired estimate it suffices to prove it on some open ball with a lower bound on its diameter, with the global estimate depending on the diameter lower bound, which here we can ensure is uniformly large (this follows from doubling estimates and unique continuation). Indeed, since $\sing(\BC^{(0)})$ is a codimension set set contained within a finite number of $(n-1)$-dimensional subspaces, the claimed estimate holds away from a tubular neighbourhood of $\sing(\BC^{(0)})$, as there we have $\dist((x,u^{(i,k)}(x)),\BC^{(0)}) = |u^{(i,k)}|(x)$.
	
	Finally, we need to extend the conclusion of the lemma to any $B^{n+1}_{2\tau_0}(y)$ for any $y\in S(\BC^{(0)})$ with $B^{n+1}_{4\tau_0}(y)\subset\{\Theta_V<Q\}$. Notice that we can still apply Theorem \ref{thm:HLW} on such balls, such we still know the non-immersed singular set is codimension $\geq 7$ here by Lemma \ref{lemma:gaps}. To check the estimate, one notes that each $u^{(i,k)}$ obtained here must coincide on a large set with one of the functions defined on $B^{n+1}_{63/64}(0)\setminus B_{\tau_0/2}(S(\BC^{(0)}))$ obtained above, and so must inherit the estimate via doubling estimates for the minimal surface equation. This completes the proof.
\end{proof}

We now complete the proof of Theorem \ref{thm:main-L^2-estimate}(1) by proving the following.

\begin{lemma}\label{lemma:fine-rep}
	Fix $\tau_0\in (0,1)$. Then, there exist constants $\eps_0,\beta_0\in (0,1)$ which depend only on $n,Q,\BC^{(0)},\tau_0$ such that the following holds. Suppose $V\in \V_Q$ and $\BC\in \mathfrak{C}_{i_0}$ satisfy the assumptions of Theorem \ref{thm:main-L^2-estimate} with $\eps_0,\beta_0$. Then, up to changing $\BC$ to a different cone with the same support, the following is true. We may write
	$$\BC = \sum^{\pfrak_0}_{i=1}\sum^{\pfrak_i}_{k=1}q_{i,k}|P^{(i,k)}|,$$
	where for each $i$ we have that $P^{(i)}$ is the closest plane in $\BC^{(0)}$ to $P^{(i,k)}$, and furthermore there exist smooth functions $u^{(i,k,s)}:P^{(i,k)}\cap B^{n+1}_{31/32}(0)\setminus B_{\tau_0}(S(\BC))\to (P^{(i,k)})^\perp$ which are minimal and such that
	\begin{itemize}
		\item $\displaystyle V\res (B^{n+1}_{31/32}(0)\setminus B_{\tau_0}(S(\BC))) = \sum^{\pfrak_0}_{i=1}\sum^{\pfrak_i}_{k=1}\sum^{q_{i,k}}_{s=1}|\graph(u^{(i,k,s)})|\res (B^{n+1}_{31/32}(0)\setminus B_{\tau_0}(S(\BC)))$;
		\item $\|u^{(i,k,s)}\|_{C^4}\leq CE_{V,\BC}$, where $C = C(n,Q,\BC^{(0)},\tau_0)>0$.
	\end{itemize}
	Furthermore, the analogous conclusion holds on a ball $B_{4\tau_0}(y)\subset \{\Theta_V<Q\}$ with $y\in S(\BC)$.
\end{lemma}

\begin{remark}
	The conclusions of Lemma \ref{lemma:fine-rep} are not immediate from Theorem \ref{thm:HLW}, as one needs to improve the estimate on the function over the fixed cone $\BC^{(0)}$ under a smallness assumption on $F_{V,\BC^{(0)}}$ to an estimate over a cone $\BC\succeq \BC^{(0)}$ under the assumptions of Theorem \ref{thm:main-L^2-estimate} in which the constants do not depend on $\BC$ but only on $\BC^{(0)}$. 
\end{remark}

\begin{proof}
The replacement of $\BC$ by a cone with the same support but with the multiplicities which match the mass of $V$ around each (half-)hyperplane will follow from the proof. To begin with, one can ensure that the total multiplicities of (half-)hyperplanes in $\BC$ agree with the multiplicity of the closest (half-)hyperplane in $\BC^{(0)}$. 

For the main claim, we argue by contradiction. Thus, we take sequences $(V_j)_j\subset \V_Q$ and $(\BC_j)_j\subset\mathfrak{C}_{i_0}$ which satisfy the assumptions of the lemma with constants $\eps_j,\beta_j\downarrow 0$ and show that the conclusions must hold uniformly for all $j$ sufficiently large. We are therefore assuming the following:
\begin{enumerate}
	\item [(i)] $\Theta_{V_j}(0)\geq Q$;
	\item [(ii)] $d(V_j\res B_2(0),\BC^{(0)}\res B_2) < \eps_j$;
	\item [(iii)] $d(\BC_j\res B_1,\BC^{(0)}\res B_1)<\eps_j$ and $\BC_j$ and $\BC^{(0)}$ are aligned;
	\item [$(\dagger)$] $F_{V_j,\BC_j}<\beta_j\mathcal{F}_{V_j,\BC_j}$.
\end{enumerate}
We may pass to a subsequence to ensure that $\dim(S(\BC_j))\equiv \sfrak$ is constant in $j$. Note that if $\sfrak = n-1$, then the lemma follows analogously to \cite[Section 10]{Wic14} (we don't go into the full details here as this case is strictly simpler than the one we present below, as it uses that $\BC_j$ have no singular points away from their spines); in this case the ``Furthermore'' part of the claim follows from Theorem \ref{thm:HLW} analogously to that seen in the proof of Lemma \ref{lemma:coarse-rep}. Moreover, if $\BC^{(0)}\in \H_{n-1}$, then there can be no density gaps at the spine of fixed size by Lemma \ref{lemma:gaps} (cf.~Lemma \ref{lemma:no-gaps}).

Thus, we may assume that $\BC^{(0)}\in \mathcal{P}_{\sfrak_0}(\pfrak_0)$. In this case, if $\BC_j\in \H_{n-1}$, the ``Furthermore" part of the lemma follows from Lemma \ref{lemma:no-gaps}, as no such gaps are possible for sufficiently small $\eps_0,\beta_0$ (namely, sufficiently large $j$). Thus, by passing to a subsequence, we may also assume that $\BC_j\in \mathcal{P}_{\sfrak}(\pfrak)$. We will actually focus on proving the ``Furthermore'' part of the lemma, as once this has been established the main claim follows in an analogous fashion, using the same covering proof from Lemma \ref{lemma:coarse-rep} (as away from a neighbourhood of $S(\BC)$ we know from \eqref{E:concentration-repeated} that the density is $<Q$ everywhere). Let us therefore fix $y\in S(\BC)$ for which $B_{4\tau_0}(y_0)\subset\{\Theta_V<Q\}$.

\textbf{Base Case:} $\pfrak= \pfrak_0$ and $\sfrak = \sfrak_0$ (i.e.~$\BC_j$ and $\BC^{(0)}$ are in the same class of cones). This case is essentially the same as proved in Lemma \ref{lemma:coarse-rep}, changing the domains to now lie in $\BC_j$. We give a few details below.

Recall $\BC^{(0)} = \sum^{\pfrak_0}_{i=1}q_i|P^{(i)}|$, and write $\BC_j = \sum^{\pfrak_0}_{i=1}q_i|P_j^{(i)}|$ where $P_j^{(i)}\to P^{(i)}$ for each $i$. By Lemma \ref{lemma:coarse-rep}, we have
$$V_j\res B_{2\tau_0}(y) = \sum^{\pfrak_0}_{i=1}\sum^{q_i}_{k=1}|\graph(u_j^{(i,k)})|,$$
where $u_j^{(i,k)}:P^{(i)}\cap B_{2\tau_0}(y)\to (P^{(i)})^\perp$. For sufficiently large $j$, one can then find functions $\widetilde{u}_j^{(i,k)}:P_j^{(i)}\cap B_{2\tau_0}(y)\to (P_j^{(i)})^\perp$ for which $\text{graph}(\widetilde{u}_j^{(i,k)}) = \graph(u_j^{(i,k)})$ and $\dist((x,\widetilde{u}_j^{(i,k)}(x)),\BC_j) = |\widetilde{u}_j^{(i,k)}(x)|$ for $x\not\in B_{\tau_0}(S(\BC_j))$. Hence, elliptic and doubling estimates for the minimal surface equation give
$$\|\widetilde{u}_j^{(i,k)}\|_{C^4(P_j^{(i)}\cap B_{\tau_0}(y))} \leq C\|\widetilde{u}_j^{(i,k)}\|_{L^2(P^{(i)}_j\cap B_{2\tau_0}(y))} \leq C\|\widetilde{u}_j^{(i,k)}\|_{C^4(P_j^{(i)}\cap B_{2\tau_0}(y)\setminus B_{\tau_0}(S(\BC)))} \leq CE_{V_j,\BC_j}.$$
\textbf{Induction Step:} Now we assume $\BC^{(0)}\prec \BC_j$, i.e.~either $\pfrak>\pfrak_0$ or $\sfrak<\sfrak_0$. Arguing as in the proof of \eqref{E:concentration-repeated}, we can find $\tilde{\eps}_j,\tilde{\beta}_j\downarrow 0$ and good $(\tilde{\eps}_j,\tilde{\beta}_j)$-representatives $\widetilde{\BC}_j\in \cup^{i_0-1}_{j=1}\mathfrak{C}_i$ for $\mathcal{F}_{V_j,\BC_j}$ which obey
\begin{equation}\label{E:graph-rep-0-0}
F_{V_j,\widetilde{\BC}_j}\leq C\mathcal{F}_{V_j,\BC_j}
\end{equation}
for some constant $C = C(n,Q,\BC^{(0)})$ independent of $j$. Write $\widetilde{\BC}_j = \sum^{\pfrak_0}_{i=1} \sum^{\pfrak_i}_{k=1} q_{i,k}|\widetilde{P}_j^{(i,k)}|$ for suitable choices of $q_{i,k}$, so in particular $\sum^{\pfrak_0}_{i=1}\sum^{\pfrak_i}_{k=1}q_{i,k}=Q$ and $\lim_{j\to\infty}\widetilde{P}_j^{(i,k)} = P^{(i)}$ for each $i,k$.

For $j$ sufficiently large, by our inductive hypothesis (Hypothesis I), we may apply Theorem \ref{thm:main-L^2-estimate} to $V_j$ and $\widetilde{\BC}_j$, and thus we have
\begin{equation}\label{E:graph-rep-0}
	V_j\res B_{2\tau_0}(y) = \sum^{\pfrak_0}_{i=1}\sum^{\pfrak_i}_{k=1}\sum^{q_{i,k}}_{s=1}|\graph(\widetilde{u}_j^{(i,k,s)})|,
\end{equation}
where
\begin{equation}\label{E:graph-rep-1}
	\|\widetilde{u}_j^{(i,k,s)}\|_{C^4(\widetilde{P}_j^{(i,k)}\cap B_{2\tau_0}(y))}\leq CE_{V_j,\widetilde{\BC}_j}
\end{equation}
for suitable $C = C(n,Q,\BC^{(0)},\tau_0)\in (0,\infty)$. Notice that Theorem \ref{thm:main-L^2-estimate} also gives
\begin{equation}\label{E:graph-rep-1.5}
	F_{V_j,\widetilde{\BC}_j} \leq CE_{V_j,\widetilde{\BC}_j}.
\end{equation}
Write $\BC_j = \sum^{\pfrak_0}_{i=1}\sum^{\pfrak_i}_{k=1}\sum^{\pfrak_{i,k}}_{s=1}q_{i,k,s}|P_j^{(i,k,s)}|$ for appropriate integers $q_{i,k,s}$, where the closest plane in $\widetilde{\BC}_j$ to $P_j^{(i,k,s)}$ is $\widetilde{P}_{j}^{(i,k)}$. From $F_{V_j,\BC_j} \leq \beta_j F_{V_j,\widetilde{\BC}_j} \leq C\beta_j E_{V_j,\widetilde{\BC}_j}$ (the second inequality coming from \eqref{E:graph-rep-1.5}) and \eqref{E:graph-rep-1} it follows that
\begin{equation}\label{E:graph-rep-2}
	\dist_\H(P_j^{(i,k,s)}\cap B_1,\widetilde{P}_j^{(i,k)}\cap B_1)\leq CE_{V_j,\widetilde{\BC}_j}.
\end{equation}
We will need this bound shortly to blow-up $\BC_j$ off $\widetilde{\BC}_j$.
Write $\ell_j^{(i,k,s)}:\widetilde{P}_j^{(i,k)}\to (\widetilde{P}_j^{(i,k)})^\perp$ for the linear function with $P_j^{(i,k,s)} = \graph(\ell_j^{(i,k,s)})$. Now define the rescaled sequences
$$v_j^{(i,k,s)}:= E_{V_j,\widetilde{\BC}_j}^{-1}\widetilde{u}_j^{(i,k,s)} \qquad \text{and} \qquad \overline{\ell}_j^{(i,k,s)}:= E_{V_j,\widetilde{\BC}_j}^{-1}\ell_j^{(i,k,s)}.$$
By \eqref{E:graph-rep-1} and \eqref{E:graph-rep-2} we can pass to a subsequence so that we have ($C^2$) limits
$$v_\infty^{(i,k,s)}:= \lim_{j\to\infty}v_j^{(i,k,s)} \qquad \text{and} \qquad \overline{\ell}_\infty^{(i,k,s)} := \lim_{j\to\infty}\overline{\ell}_j^{(i,k,s)}.$$
Since $F_{V_j,\BC_j}<C\beta_j E_{V_j,\widetilde{\BC}_j}$, we get the following properties:
\begin{itemize}
	\item For each $(i,k,s)$ there exists $s^\prime = s^\prime(s)$ such that $v_\infty^{(i,k,s)} \equiv \overline{\ell}_\infty^{(i,k,s^\prime)}$;
	\item For each $(i,k,s)$ there exists $s^{\prime\prime} = s^{\prime\prime}(s)$ such that $\overline{\ell}_\infty^{(i,k,s)}\equiv v_\infty^{(i,k,s^{\prime\prime})}$.
\end{itemize}
We may therefore relabel the functions $(\widetilde{u}_j^{(i,k,s)})_{1\leq s\leq q_{i,k}}$ by grouping together those which blow-up to the same linear function. Indeed, relabel them by $\widetilde{u}_j^{(i,k,s,t)}$ for $1\leq s\leq \pfrak_{i,k}$ and $1\leq t \leq q_{i,k,s}$, where $v_\infty^{(i,k,s,t)} \equiv \overline{\ell}_\infty^{(i,k,s)}$, for $v_\infty^{(i,k,s,t)}$ the corresponding limit of $\widetilde{u}_j^{(i,k,s,t)}$. We can also reparameterise $\widetilde{u}_j^{(i,k,s,t)}$ from $\widetilde{P}_j^{(i,k)}$ to $P_j^{(i,k,s)}$, and so write $u_j^{(i,k,s,t)}:P_j^{(i,k,s)}\to (P_j^{(i,k,s)})^\perp$ for the function whose graph coincides with that of $\widetilde{u}_j^{(i,k,s,t)}$. From the convergence, we know that $\|u_j^{(i,k,s,t)}\|_{C^4}\leq C\beta_j E_{V_j,\widetilde{\BC}_j}$.

To complete the proof, it suffices to show that $\overline{\ell}_\infty^{(i,k,s)}\not\equiv \overline{\ell}_\infty^{(i,k,s^\prime)}$ for all $s\neq s^\prime$ and $i,k$. Indeed, this coupled with the uniform convergence to the blow-up tells us that $\dist((x,u^{(i,k,s,t)}_j(x)),\BC_j) = |u_j^{(i,k,s,t)}(x)|$ on a set of large measure, which, as we have seen before, tells us that the $L^2$ norm of $u_j^{(i,k,s,t)}$ can be controlled $E_{V_j,\BC_j}$, completing the proof.

It therefore suffices to prove the lower bound
$$|\ell_j^{(i,k,s)}-\ell_j^{(i,k,s^\prime)}|\geq cE_{V_j,\widetilde{\BC}_j}\qquad \text{on }\del B_1$$
for a fixed constant $c$. This is analogous to Lemma \ref{lemma:separation}, except now we only assume control on the two-sided excess. Indeed, this can be argued by contradiction, since if it were fail, then letting $\BC_j^*$ denote the cone obtained from deleting the plane $P_j^{(i,k,s^\prime)}$ from $\BC_j$, we would get for some $\delta_j\downarrow 0$,
$$F_{V_j,\BC_j^*} \leq F_{V_j,\BC_j} + \delta_j E_{V_j,\widetilde{\BC}_j} \leq (C\beta_j + \delta_j)E_{V_j,\widetilde{\BC}_j} \leq C(\beta_j + \delta_j)\mathcal{F}_{V_j,\BC_j} \leq (C\beta_j + \delta_j)F_{V_j,\BC_j^*}$$
which is a contradiction for sufficiently large $j$; here we have used \eqref{E:graph-rep-0-0} and \eqref{E:graph-rep-1.5}. This completes the proof.
\end{proof}

We have therefore now proved Theorem \ref{thm:main-L^2-estimate}(1). Proving the remaining estimates, Theorem \ref{thm:main-L^2-estimate}(2)--(5), is the topic of Section \ref{sec:main-estimate-proof}.

\subsection{Proof of the estimates in Theorem \ref{thm:main-L^2-estimate}}\label{sec:main-estimate-proof}

The proof here follows the general scheme in \cite{BK17, BKMW25}, which is inspired by those seen in the works \cite{Sim93, Wic14}, but with the additional difficulty that the spine dimension of $\BC$ can be strictly less than $\BC^{(0)}$. In fact, we just need to make some observations regarding the structure of the varifold in regions where the \emph{one-sided} height excess to $\BC$ is small, taking into account multiplicity, and the proofs after that point will be the same.

Again, we choose coordinates on $\R^{n+1}$ of the form $(x,y)\in S(\BC)^\perp\times S(\BC)\cong \R^{n+1}$. Given a subspace $A\subset\R^{n+1}$, $\zeta\in A$, $\rho\in (0,1)$, and $r\in (0,\rho)$, we define a corresponding \emph{toroidal region} by
$$T^A_{\rho,r}(\zeta):= \{(x,y)\in A^\perp\times A: (|x|-\rho)^2 + |y-\zeta|^2 < r^2\}.$$
For a cone $\BC\in \CC$ we define $T^{\BC}_{\rho,r}(\zeta) := T^{S(\BC)}_{\rho,r}(\zeta)$. Notice that we always have $T^{\BC}_{\rho,r}(\zeta)\cap S(\BC) = \emptyset$. More generally, for any set $S\subset\R^{n+1}$ we can define $T^{\BC}(S)$ to be the region of revolution formed by rotating $S$ about $S(\BC)$, i.e.~the set of $(x,y)\in S(\BC)^\perp\times S(\BC)$ where there is some $\w\in S(\BC)^\perp\cap S^n$ such that $(|x|\w,y)\in S$. We then define the \emph{one-sided excess} of $V\in \V_Q$ related to $\BC\in \CC$ in $T^{\BC}_{\rho,r}(\zeta)$ by
$$E_{V,\BC}(T^{\BC}_{\rho,r}(\zeta)) := \left(\frac{1}{r^{n+2}}\int_{T^{\BC}_{\rho,r}(\zeta)}\dist^2(x,\BC)\, \ext\|V\|(x)\right)^{1/2}.$$
The first results we need are concerned with the behaviour of varifolds which have small \emph{one-sided} excess in regions $T^{\BC}_{\rho,r}(\zeta)$. To aid the reader, the general structure of our presentation follows that in \cite[Section 6]{BK17} (cf.~\cite[Part 3]{BKMW25}), although with one significant simplification. Since $V$ is close to a cone with density $Q$, we only know that the mass of $V$ itself is close to that of $\BC$, and thus the mass ratio of $V$ on balls can be upper bounded by $Q+\eps$ for a small $\eps>0$. This means that, when assuming that only the one-sided excess of $V$ relative to $\BC$ is small, $V$ is close to a subset of the (half-)hyperplanes in $\BC$, and furthermore the mass of $V$ near each (half-)hyperplane is \emph{at most} $Q+\eps$ times the mass of the (half-)hyperplanes in the region. The important point here is having an upper bound of $Q+\frac{1}{4}$. Notice that one cannot ensure a mass bound of $Q-1$ here\footnote{This would be the ideal assumption for an inductive argument, if one only knows the validity of sheeting theorems, such as Theorem \ref{thm:bellettini}, near hyperplanes of multiplicity $<Q$ at this point of the argument.} as, for instance, $V$ could be close to a single hyperplane in $\BC$ with multiplicity $Q$. Here, precisely because Theorem \ref{thm:bellettini} holds for \emph{any} multiplicity (which is reflected in the form of $(\S3\text{b})$) we always have a sheeting theorem, regardless of the multiplicity of the (half-)hyperplane which $V$ is close to. This avoid the need to use, like in \cite[Section 6]{BK17}, Almgren's Lipschitz approximation, making the situation closer to \cite{Sim93} and \cite[Section 16]{Wic14}, with the only additional caveat arising from the spine dimension possibly being smaller.

We start with the simplest ``base case'' version of the estimates when only the one-sided excess is small, where $\BC$ is in the same class of cones as $\BC^{(0)}$.

\begin{lemma}\label{lemma:one-sided-1}
	There exists $\eps_0 = \eps_0(n,k,\BC^{(0)})\in (0,1)$ such that the following holds. Suppose that $V\in \V_Q$ and $\BC\in \CC$ satisfy:
	\begin{enumerate}
		\item [\textnormal{(i)}] $(\w_n 2^n)^{-1}\|V\|(B^{n+1}_2(0))\leq Q+\frac{1}{4}$;
		\item [\textnormal{(ii)}] $d(\BC\res B_1,\BC^{(0)}\res B_1)<\eps_0$ and $\BC$ lies in the same class $\mathfrak{C}_{i^{(0)}}$ as $\BC^{(0)}$;
		\item [\textnormal{(iii)}] $E_{V,\BC}(T^{\BC}_{1/2,7/16}(0)) < \eps_0$.
	\end{enumerate}
	Then there are two possible conclusions depending on the spine dimension $\sfrak_0$ of $\BC^{(0)}$.
	\begin{enumerate}
		\item [\textnormal{(A)}] $\sfrak_0 = n-1$. Then, working with half-hyperplane notation, so that $\BC = \sum^{\pfrak_0}_{i=1} q_i|H^{(i)}|$, there exists $I\subset \{1,2,\dotsc,Q\}$ (which could be empty) such that $V\res T^{\BC}_{1/2,6/16}(0) = \sum_{i\in I}V_i$, where each $V_i$ is a stationary integral varifold in $T^{\BC}_{1/2,6/16}(0)$ for which there exists a domain $\Omega_i\subset H^{(i^\prime)}\cap T^{\BC}_{1/2,6/16}(0)$ for some $i^\prime = i^\prime(i)$ and $u_i\in C^\infty(\Omega_i; (P^{(i^\prime)})^\perp)$ such that
		\begin{enumerate}
			\item [\textnormal{(i)}] $V_i = |\graph(u_i)|$;
			\item [\textnormal{(ii)}] $\|u_i\|_{C^4} \leq CE_{V,\BC}(T^{\BC}_{1/2,7/16}(0))$.
		\end{enumerate}
		\item [\textnormal{(B)}] $\sfrak_0<n-1$. Then, writing $\BC = \sum^{\pfrak_0}_{i=1}q_i|P^{(i)}|$, there exists $I\subset\{1,2,\dotsc,Q\}$ (which could be empty) such that $V \res T_{1/2,6/16}^{\BC}(0) = \sum_{i\in I}V_i$ where for each $i$ there exists a domain $\Omega_i\subset P^{(i^\prime)}\cap T^{\BC}_{1/2,6/16}(0)$ for some $i^\prime = i^\prime(i)$ and a smooth, minimal, function $u_i:\Omega_i\to (P^{(i^\prime)})^\perp$ such that
		\begin{enumerate}
			\item [\textnormal{(i)}] $V_i = |\graph(u_i)|$;
			\item [\textnormal{(ii)}] $\|u_i\|_{C^4} \leq CE_{V,\BC}(T^{\BC}_{1/2,7/16}(0))$.
		\end{enumerate}
	\end{enumerate}
	Here, $C = C(n,Q,\BC^{(0)})\in (0,\infty)$, and if $I = \emptyset$ we simply mean $V\res T^{\BC}_{1/2,6/16}(0) = 0$.
\end{lemma}

\begin{proof}
	We prove this by contradiction. If it were false, then we could find sequences $V_j\in \V_Q$, $\BC_j\in \CC$, which satisfy assumptions (i) -- (iii) of the lemma with $\eps_j$ in place of $\eps_0$, where $\eps_j\downarrow 0$, yet the conclusions of the lemma fail.
	
	Our assumptions give that $V_j\res T^{\BC_j}_{1/2,7/16}(0)$ converge to some stationary integral varifold $W$, which by (ii) and (iii) must obey $\spt\|W\|\subset \spt\|\BC^{(0)}\|\cap T^{\BC^{(0)}}_{1/2,7/16}(0)$ (notice that this uses $S(\BC_j)\to S(\BC^{(0)})$ so that $T^{\BC_j}_{1/2,7/16}(0)\to T^{\BC^{(0)}}_{1/2,7/16}(0)$). If $W = 0$, then set $I = \emptyset$ and the proof is complete. Thus, we may assume $W\neq 0$.
	
	Suppose now $\sfrak_0 = n-1$. Since $W$ is stationary, the constancy theorem implies that each regular component of $W$ has constant multiplicity. Thus, as $\sfrak_0 = n-1$, we must have, writing $\BC^{(0)} = \sum^{\pfrak_0}_{i=1}q^{(0)}_i|H_{0}^{(i)}|$, where $H^{(i)}_j\to H^{(i)}_0$,
	$$W = \sum_{i\in J} \theta_i |H_{0}^{(i)}\cap T^{\BC^{(0)}}_{1/2,7/16}(0)|,$$
	for some subset $J\subseteq\{1,\dotsc,\pfrak_0\}$, where each $\theta_i$ is a positive integer. Now fix a small $\eta>0$ so that the regions $B_\eta(H^{(i)}_0)\cap T_{1/2,7/16}^{\BC^{(0)}}(0)$ are disjoint. Then, as $V_j\res T^{\BC_j}_{1/2,7/16}(0)$ converge to $W$ locally in Hausdorff distance in $T^{\BC^{(0)}}_{1/2,7/16}(0)$, we know that for all sufficiently large $j$,
	$$\spt\|V_j\|\cap T^{\BC^{(0)}}_{1/2,13/32}(0)\subset \bigcup_{i\in I}B_{\eta}(H_0^{(i)})\cap T^{\BC^{(0)}}_{1/2,7/16}(0),$$
	and hence if $V_j^{(i)}:= V^j\res B_{\eta}(H_0^{(i)})\cap T^{\BC^{(0)}}_{1/2,13/32}(0)$, each $V_j^{(i)}$ is a stationary integral varifold. Our mass hypothesis (i) combined with monotonicity of the mass ratio guarantees that $\theta_i\in \{1,2,\dotsc,Q\}$ for each $i$. We may therefore apply Theorem \ref{thm:bellettini} to deduce that (A) holds, completing the proof in the case $\sfrak_0 = n-1$. When instead $\sfrak_0 < n-1$, the argument is completely analogous, just using the slightly different form of the cones and that varifolds in $\V_Q$ do not contain any non-immersed classical singularities of density $<Q$.
\end{proof}

Next, we prove the inductive version of Lemma \ref{lemma:one-sided-1}, namely when $\BC^{(0)}\prec \BC$. Again, as we saw in Theorem \ref{thm:main-L^2-estimate}, an important point is to show that good density points must accumulate around the spine of the finer cone under a suitable hypothesis guaranteeing that the finer cone has comparatively smaller excess to all coarser cones. As such, we define
$$\mathcal{E}_{V,\BC}(T^{\BC}_{\rho,r}(\zeta)):= \inf\{E_{V,\BC^\prime}(T^{\BC}_{\rho,r}(\zeta)): \BC^\prime\prec \BC\text{ are aligned}\}.$$

\begin{lemma}\label{lemma:one-sided-2}
	There exist $\eps,\beta\in (0,1)$ depending only on $n,Q,\BC^{(0)}$ such that the following is true. Suppose that $V\in \V_Q$ and $\BC\in \mathfrak{C}_{i_0}$ satisfy:
	\begin{enumerate}
		\item [\textnormal{(1)}] $(\w_n 2^n)^{-1}\|V\|(B^{n+1}_2(0))\leq Q+\frac{1}{4}$;
		\item [\textnormal{(2)}] $\BC^{(0)}\prec \BC$ are aligned;
		\item [\textnormal{(3)}] Hypothesis I holds;
		\item [\textnormal{(4)}] $d(\BC\res B_1,\BC^{(0)}\res B_1) < \eps$;
		\item [\textnormal{(5)}] $E_{V,\BC}(T_{1/2,7/16}^{\BC}(0)) < \eps$;
		\item [\textnormal{(6)}] $E_{V,\BC}(T_{1/2,7/16}^{\BC}(0)) < \beta \mathcal{E}_{V,\BC}(T_{1/2,7/16}^{\BC}(0))$.
	\end{enumerate}
	Then, one of the following conclusions holds depending on the value of $\sfrak \equiv \dim S(\BC)$.
	\begin{enumerate}
		\item [\textnormal{(C)}] $\sfrak = n-1$. Then, working with half-hyperplane notation, write $\BC^{(0)} = \sum^{\pfrak_0}_{i=1}q_i^{(0)}|H^{(i)}_0|$ and $\BC = \sum^{\pfrak_0}_{i=1}\sum_{k=1}^{\pfrak_i}q_{i,k}|H^{(i,k)}|$, where $H^{(i)}_0$ is the closest half-hyperplane in $\BC^{(0)}$ to $H^{(i,k)}$. Then there exists $I\subset\{1,2,\dotsc,Q\}$ such that $V\res T^{\BC}_{1/2,6/16}(0) = \sum_{i\in I}V_i$, where each $V_i$ is a stationary integral varifold in $T^{\BC}_{1/2,6/16}(0)$ for which there exists a domain $\Omega_i\subset H^{(i^\prime,k^\prime)}\cap T^{\BC}_{1/2,6/16}(0)$ for some $(i^\prime,k^\prime)$ depending on $i$ and $u_i\in C^\infty(\Omega_i; (P^{(i^\prime,k^\prime)})^\perp)$ which is minimal such that
		\begin{enumerate}
			\item [\textnormal{(i)}] $V_i = |\graph(u_i)|$;
			\item [\textnormal{(ii)}] $\|u_i\|_{C^4}\leq CE_{V,\BC}(T^{\BC}_{1/2,7/16}(0))$.
		\end{enumerate}
		\item [\textnormal{(D)}] $\sfrak<n-1$. Then we work with hyperplane notation, so $\BC^{(0)} = \sum_{i=1}^{\pfrak_0}q_i^{(0)}|P_0^{(i)}|$ and $\BC = \sum^{\pfrak_0}_{i=1}\sum^{\pfrak_i}_{k=1}q_{i,k}|P^{(i,k)}|$, with $P_0^{(i)}$ the closest plane in $\BC^{(0)}$ to $P^{(i,k)}$. Then there exists $I\subset \{1,2,\dotsc,Q\}$ such that $V\res T^{\BC}_{1/2,6/16}(0) = \sum_{i\in I}V_i$, where for each $i$ there exists a domain $\Omega_i\subset P^{(i^\prime,k^\prime)}\cap T^{\BC}_{1/2,6/16}(0)$ for some $(i^\prime,k^\prime)$ depending on $i$ and $u_i\in C^\infty(\Omega_i;(P^{(i^\prime,k^\prime)})^\perp)$ which is minimal such that
		\begin{enumerate}
			\item [\textnormal{(i)}] $V_i = |\graph(u_i)|$;
			\item [\textnormal{(ii)}] $\|u_i\|_{C^4}\leq CE_{V,\BC}(T^{\BC}_{1/2,7/16}(0))$;
		\end{enumerate}
	\end{enumerate}
	Here, $C = C(n,Q,\BC^{(0)})\in (0,\infty)$. Again, if $I = \emptyset$, then the sum over $I$ is taken to be $0$.
\end{lemma}

\begin{proof}
	We argue by contradiction. Thus, we suppose that there exist sequences $(V_j)_j\subset\V_Q$, $(\BC_j)_j\subset \mathfrak{C}_{i_0}$ obeying the assumptions of the lemma with parameters $\eps_j,\beta_j\downarrow 0$ in place of $\eps,\beta$. In particular, setting $\mathcal{T}_j:= T^{\BC_j}_{1/2,7/16}(0)$, we have
	$$E_{V_j,\BC_j}(\mathcal{T}_j) <\beta_j\mathcal{E}_{V_j,\BC_j}(\mathcal{T}_j) < \beta_j\eps_j.$$
	We want to prove that the conclusions hold along some subsequence.
	
	Note that by assumption $\sfrak := \dim(S(\BC_j))$ is fixed. By passing to a subsequence, we may assume that
	\begin{itemize}
		\item $S(\BC_j)\to A$ for some subspace $A\subseteq S(\BC^{(0)})$;
		\item $V_j\res \mathcal{T}_j\weakly W$ locally as varifolds in $T^A_{1/2,7/16}(0)$, where $W$ is a stationary integral varifold in $T^A_{1/2,7/16}(0)$.
	\end{itemize}
	Assumptions (4) and (5) ensure that $\spt\|W\|\subseteq \spt\|\BC^{(0)}\|\cap T^A_{1/2,7/16}(0)$, and thus by the constancy theorem $W$ must be formed of some collection of half-hyperplanes in $\BC^{(0)}$ with certain multiplicities. If $W=0$, then one can take $I=\emptyset$ and the proof is complete, so we may assume $W\neq 0$.
	
	If $\sfrak = n-1$, then necessarily $W$ is supported within half-hyperplanes. Furthermore, as $S(\BC_j) \equiv S(\BC^{(0)})$ for all $j$, we have $\mathcal{T}_j\cap S(\BC^{(0)}) = \emptyset$, and consequently that, for each fixed $j$, $\mathcal{T}_j\cap H^{(i)}_0$ are pairwise disjoint for distinct $i$. Thus, we have
	$$W = \sum_{i \in J}\theta_i |H_0^{(i)}\cap T^{\BC^{(0)}}_{1/2,7/16}(0)|$$
	for some non-empty $J\subseteq\{1,\dotsc,\pfrak_0\}$ and where $\theta_i\in\{1,2,\dotsc,Q\}$ from our mass bounds.
	
	If instead $\sfrak<n-1$, then $W$ is stationary and supported on $\bigcup_i P_0^{(i)}$. We first claim that $W$ must be a sum of full hyperplanes. Indeed, if a half-hyperplane in $W$ has multiplicity $Q$, since $A\subsetneq S(\BC^{(0)})$, stationarity of $W$ in $T^{\BC}_{1/2,7/16}(0)$ combined with the mass upper bound forces $W$ to coincide with the corresponding full hyperplane with multiplicity $Q$. If no half-hyperplane in $W$ has multiplicity $Q$, then by upper semi-continuity of the mass ratio we must have $\Theta_{V_j}<Q$ in $T^{\BC}_{1/2,6/16}(0)$ for all sufficiently large $j$, which combined with Lemma \ref{lemma:gaps} and Theorem \ref{thm:HLW}, implies that $W$ must be a sum of full hyperplanes.
	
	In either case, we need not have $\spt\|W\| = \spt\|\BC^{(0)}\|\cap T^A_{1/2,7/16}(0)$, and as such $F_{V_j,\BC^{(0)}}(\mathcal{T}_j)$ need not be small. Instead, there is a sub-cone $\widetilde{\BC}^{(0)}$ of $\BC^{(0)}$, namely a cone comprised of a certain subset of the (half-)hyperplanes in $\BC^{(0)}$, for which $F_{V_j,\widetilde{\BC}^{(0)}}(\mathcal{T}_j)\to 0$. We must have that $\widetilde{\BC}^{(0)}$ is ``simpler'' than $\BC^{(0)}$; indeed, one can take $\widetilde{\BC}^{(0)}$ such that $\widetilde{\BC}^{(0)}\res T^A_{1/2,7/16}(0) = W$. We then know that $\Theta_{\widetilde{\BC}^{(0)}}(0) \leq Q = \Theta_{\BC^{(0)}}$, and we also know that $S(\BC^{(0)})\subseteq S(\widetilde{\BC}^{(0)})$, and furthermore than the number of distinct half-hyperplanes in $\widetilde{\BC}^{(0)}$ must be at most that of $\BC^{(0)}$. This suggests an inductive argument based on these three measures of the complexity of $\widetilde{\BC}^{(0)}$.
	
	Indeed, firstly if $\Theta_{\widetilde{\BC}^{(0)}}(0)<Q$, then $\Theta_{\widetilde{\BC}^{(0)}}<Q$ everywhere, and so by upper semi-continuity of the density we must have $\Theta_{V_j}<Q$ on $\mathcal{T}_j$ for all $j$ sufficiently large. In this case, by Lemma \ref{lemma:gaps}, we can then apply one of Theorem \ref{thm:bellettini} (if $\widetilde{\BC}^{(0)}$ is supported on a single hyperplane or has spine dimension $n-1$) or Theorem \ref{thm:HLW} (if $\widetilde{\BC}^{(0)}$ is supported on at least two hyperplanes and has spine dimension $\leq n-1$) in order to prove the desired conclusion. We note that the estimate on the functions, which would initially be in terms of $E_{V_j,\widetilde{\BC}^{(0)}}$, can be improved to $E_{V_j,\BC_j}$ using the assumption (6) in the same manner as seen in Section \ref{sec:graphical-rep-proof}.
	
	Thus we may assume that $\Theta_{\widetilde{\BC}^{(0)}}(0) = Q$. Again, if $\widetilde{\BC}^{(0)}$ is supported on a single hyperplane, one can apply Theorem \ref{thm:bellettini} (cf.~$(\S3\text{b})_Q$) to conclude. We also claim that, if $\widetilde{\BC}^{(0)}$ has spine dimension $n-1$ and has one half-hyperplane occurring with multiplicity $Q$, then either
	\begin{enumerate}
		\item [(a)] $\sfrak = n-1$; or
		\item [(b)] $\widetilde{\BC}^{(0)}$ is a single hyperplane with multiplicity $Q$.
	\end{enumerate}
	Indeed, if $\sfrak\leq n-2$, then $W$ needs to be stationary in $\mathcal{T} := T^A_{1/2,7/16}(0)$. As $\dim(A)\leq n-2$ and $\dim(S(\widetilde{\BC}^{(0)}))=n-1$, $W$ must be stationary on a region containing its spine. As one half-hyperplane in $W$ has multiplicity $Q$, the only way for $W$ to be stationary across its spine, as $\Theta_W(0) = Q$, is for it to be a single hyperplane with multiplicity $Q$, i.e.~(b) must hold. Thus, we have proven this dichotomy. The claim in either case here follows in the same manner as the inductive case above; indeed, in case (b) one can again apply Theorem \ref{thm:bellettini} to the $V_j$, and in case (a) we know that the regions $\mathcal{T}_j$ disconnect the half-hyperplanes in the cone, and so again one can apply Theorem \ref{thm:bellettini} to conclude, even if a half-hyperplane has multiplicity $Q$.
	
	Hence, the only cases left to analyse are when $\Theta_{\widetilde{\BC}^{(0)}}(0) = Q$ and:
	\begin{enumerate}
		\item [(I)] $\widetilde{\BC}^{(0)}$ has spine dimension $\leq n-2$ (and so is a sum of hyperplanes);
		\item [(II)] $\widetilde{\BC}^{(0)}$ has spine dimension $n-1$ and all half-hyperplanes have multiplicity $<Q$.
	\end{enumerate}
	In particular, we must have $\widetilde{\BC}^{(0)}\in \CC$; in particular, our conditions give $\widetilde{\BC}^{(0)}\preceq \BC^{(0)}$. As $\widetilde{\BC}^{(0)}$ is never finer than $\BC^{(0)}$, we are therefore able to argue in an analogous manner to that we saw in Section \ref{sec:good-density-points} and Section \ref{sec:graphical-rep-proof}, but with $\widetilde{\BC}^{(0)}$ in place of $\BC^{(0)}$ in order to get the desired conclusions.
	
	Indeed, just as in the proof of \eqref{E:concentration-repeated} (now treating $\widetilde{\BC}^{(0)}$ as the base cone), we can find $\tilde{\eps}_j,\tilde{\beta}_j\downarrow 0$ and $(\tilde{\eps}_j,\tilde{\beta}_j)$-good representatives $\Dbf_j$ for $\mathcal{F}_{V_j,\BC_j}(\mathcal{T}_j)$; we therefore have
	\begin{itemize}
		\item $F_{V_j,\Dbf_j}\leq \tilde{\beta}_j\mathcal{F}_{V_j,\Dbf_j}(\mathcal{T}_j)$;
		\item $F_{V_j,\Dbf_j}(\mathcal{T}_j) \leq C\mathcal{F}_{V_j,\BC_j}(\mathcal{T}_j)$, where $C$ is a constant independent of $j$.
	\end{itemize}
	In particular $\Dbf_j$ must converge to $W$ and
	$$E_{V_j,\BC_j}(\mathcal{T}_j)<\beta_j\mathcal{E}_{V_j,\BC_j}(\mathcal{T}_j) \leq \beta_jE_{V_j,\Dbf_j},$$
	so that $E_{V_j,\Dbf_j}(\mathcal{T}_j)^{-1}E_{V_j,\BC_j}(\mathcal{T}_j)\to 0$.
	
	For any choice of $\tau>0$ we can now apply Theorem \ref{thm:main-L^2-estimate} to $V_j,\Dbf_j$ and $\widetilde{\BC}^{(0)}$ in place of $V,\BC$, and $\BC^{(0)}$, respectively, provided $j$ is sufficiently large. Notice that this requires our inductive assumption on its validity not only through Hypothesis I, but also inductively on the level of the base cone $\BC^{(0)}$ as well (as the support of $\widetilde{\BC}^{(0)}$ need not be close to that of $\BC^{(0)}$.)
	
	We claim that, for each half-hyperplane $\widehat{H}_i$ in $\Dbf_j$, there is a half-hyperplane $H^\prime_i$ in $\BC_j$ for which
	\begin{equation}\label{E:one-sided-1}
		\dist_\H(\widehat{H}_i\res B_1,H^\prime_i\res B_1) \leq CE_{V_j,\Dbf_j}(\mathcal{T}_j).
	\end{equation}
	To see this, note that for any $\sigma\in (6/16,7/16)$ we have for all $j$ sufficiently large, writing $\ell^\prime_i$ for the linear functions representing the half-hyperplanes in $\BC_j$ over the corresponding half-hyperplane in $\Dbf_j$ (if such a half-hyperplane exists, which is not clear at present as only the one-sided excess is small),
	\begin{align*}
		E^2_{V_j,\BC_j}(\mathcal{T}_j) & \geq \int_{T^{\widetilde{\BC}^{(0)}}_{1/2,\sigma}(0)\cap \{r_{\Dbf_j}>1/8\}}\dist^2(x,\BC_j)\, \ext\|V_j\|(x)\\
		& \geq C^{-1}\sum_i\int_{T^{\widetilde{\BC}^{(0)}}_{1/2,\sigma}(0)\cap \{r_{\Dbf_j}>1/8\}\cap \widehat{H}_i}\min_{i^\prime}|u_i(x)-\ell^\prime_{i^\prime}(x)|^2\, \ext x\\
		& \geq C^{-1}\min_{i^\prime}\dist^2_\H(\widehat{H}_i\res B_1,H^\prime_{i^\prime}\res B_1) - CE_{V_j,\Dbf_j}^2(\mathcal{T}_j)
	\end{align*}
	which since $E_{V_j,\BC_j}(\mathcal{T}_j)\leq \beta_j E_{V_j,\Dbf_j}(\mathcal{T}_j)$ gives \eqref{E:one-sided-1}.
	
	Let $\widetilde{\BC}_j$ denote the cone formed from only taking those half-hyperplanes in $\BC_j$ which satisfy \eqref{E:one-sided-1} for some half-hyperplane in $\Dbf_j$. We now proceed as in the proof of \eqref{E:concentration-repeated}. We perform a fine blow-up of $V_j$ and $\widetilde{\BC}_j$ relative to $\Dbf_j$ in the region $\mathcal{T}_j$ (this is done in exactly the same way but now just locally in the region $T^A_{1/2,7/16}(0)$). Our convergence combined with $E_{V_j,\Dbf_j}(\mathcal{T}_j)^{-1}E_{V_j,\BC_j}(\mathcal{T}_j)\to 0$ guarantees that the fine blow-up of $V_j$ is consists of linear functions, each of which agrees with one of the linear functions formed when blowing-up $\widetilde{\BC}_j$ relative to $\Dbf_j$. In fact, every linear function formed when blowing-up $\widetilde{\BC}_j$ relative to $\Dbf_j$ must be realised by a linear function in the blow-up of $V_j$ off $\Dbf_j$, else one could remove a plane from $\widetilde{\BC}_j$ without significantly changing the one-sided excess, which would contradict $E_{V_j,\BC_j}(\mathcal{T}_j)<\beta_j\mathcal{E}_{V_j,\BC_j}(\mathcal{T}_j)$ (this uses that we can verify strong $L^2$ convergence globally in $\mathcal{T}_j$, rather than just locally, cf.~the proof of \eqref{E:concentration-repeated}).
	
	This convergence can then be used to again verify, just as in Section \ref{sec:good-density-points}, that the set of good density points must accumulate around $S(\BC_j)$, else one gets a contradiction to $\mathcal{E}_{V_j,\BC_j}(\mathcal{T}_j)^{-1}E_{V_j,\BC_j}(\mathcal{T}_j)\to 0$. Once one has the accumulation of good density points to $S(\BC_j)$, the proof is completed as in Lemma \ref{lemma:fine-rep}, leading to the conclusions of the lemma depending on the form of $\BC_j$.
\end{proof}

The final lemma needed is a small extension of Lemma \ref{lemma:one-sided-2} which removes hypothesis (6) therein.

\begin{lemma}\label{lemma:one-sided-3}
	Fix $\tilde{\eps}\in (0,1)$. Then, there exists $\eps = \eps(n,Q,\BC^{(0)},\tilde{\eps})\in (0,1)$ such that the following holds. Suppose that $V\in \V_Q$ and $\BC\in \mathfrak{C}_{i_0}$ satisfy:
	\begin{enumerate}
		\item [(1)] $(\w_n 2^n)^{-1}\|V\|(B^{n+1}_2(0))\leq Q+\frac{1}{4}$;
		\item [(2)] $\BC^{(0)}\prec \BC$ are aligned;
		\item [(3)] Hypothesis I holds;
		\item [(4)] $d(\BC\res B_1,\BC^{(0)}\res B_1) < \eps$;
		\item [(5)] $E_{V,\BC}(T^{\BC}_{1/2,7/16}(0))<\eps$.
	\end{enumerate}
	Then, either the conclusion of Lemma \ref{lemma:one-sided-2} holds, or we have:
	\begin{enumerate}
		\item [(E)] $V$ is equally well approximated by a coarser cone, i.e.~there exists $\BC^{(0)}\preceq\Dbf\prec \BC$, which are all aligned, with
		\begin{enumerate}
			\item [(i)] $E_{V,\Dbf}(T^{\BC}_{1/2,7/16}(0))\leq CE_{V,\BC}(T^{\BC}_{1/2,7/16}(0))$, for constant $C = C(n,k,\BC^{(0)})\in (0,\infty)$;
			\item [(ii)] $E_{V,\Dbf}(T^{\BC}_{1/2,7/16}(0))<\tilde{\eps}$ and $\dist_\H(\Dbf\res B_1,\BC^{(0)}\res B_1)<\tilde{\eps}$.
		\end{enumerate}
	\end{enumerate}
\end{lemma}
\begin{proof}
	If one has $E_{V,\BC}(T^{\BC}_{1/2,7/16}(0))<\beta\mathcal{E}_{V,\BC}(T^{\BC}_{1/2,7/16}(0))$ for sufficiently small $\beta = \beta(n,Q,\BC^{(0)})\in (0,1)$, the result follows from Lemma \ref{lemma:one-sided-2}. If this were to fail, then one can simply choose $\Dbf\prec \BC$ with $E_{V,\Dbf}(T^{\BC}_{1/2,7/16}(0)) \leq \frac{3}{2}\mathcal{E}_{V,\BC}(T^{\BC}_{1/2,7/16}(0))$, from which (E) follows, provided $\eps$ is sufficiently small depending also on $\tilde{\eps}$.
\end{proof}

Armed with Lemma \ref{lemma:one-sided-1}, Lemma \ref{lemma:one-sided-2}, and Lemma \ref{lemma:one-sided-3}, the proofs of Theorem \ref{thm:main-L^2-estimate} and Corollary \ref{cor:main-L^2-estimate} are now extremely similar to the arguments seen in \cite[Section 3]{BKMW25} (cf.~\cite[Section 6]{BK17}). As such, in the interest of length and avoiding unnecessary repetition, we omit the details and refer the reader to \cite[Section 3]{BKMW25}.

As such, by induction, we have now proved Theorem \ref{thm:main-L^2-estimate} and Corollary \ref{cor:main-L^2-estimate} by induction, and so can assume their validity for any $\BC\in \CC$.

\bigskip

\part{The Fine Regularity Theorem}\label{part:regularity}

In this final part of the paper we will prove the following \emph{fine} $\eps$-regularity theorem, which directly implies Theorem \ref{thm:main-reg}.
\begin{thmx}\label{thm:fine-reg}
	There exist $\eps_0,\beta_0\in (0,1)$ depending only on $n,Q,\BC^{(0)}$ such that if $V\in \V_Q$ and $\BC\in \CC$ satisfy:
	\begin{itemize}
		\item [(i)] $\Theta_V(0)\geq Q$;
		\item [(ii)] $d(V\res B^{n+1}_2(0), \BC^{(0)}\res B^{n+1}_2(0))<\eps_0$;
		\item [(iii)] $d(\BC\res B^{n+1}_1(0), \BC^{(0)}\res B^{n+1}_1(0))<\eps_0$, and $\BC$ is aligned with $\BC^{(0)}$;
		\item [$(\dagger)$] $F_{V,\BC}<\beta_0\mathcal{F}_{V,\BC}$;
	\end{itemize}
	then, there is a cone $\widetilde{\BC} \succeq \BC$ aligned and a rotation $\Gamma:\R^{n+1}\to \R^{n+1}$ with
	$$d(\widetilde{\BC}\res B_1,\BC\res B_1) \leq CF_{V,\BC} \qquad \text{and}\qquad \|\Gamma-\id\|\leq C\mathcal{F}_{V,\BC}^{-1}F_{V,\BC},$$
	such that $\widetilde{\BC}$ is the unique tangent cone to $\Gamma_\#^{-1}V$ at $0$, and
	$$\sigma^{-n-2}\int_{B_{\sigma}^{n+1}(0)}\dist^2(x,\widetilde{\BC})\, \ext\|\Gamma_\#^{-1}V\|(x) \leq C\sigma^{2\alpha}F_{V,\BC}^2 \qquad \text{for all }\sigma\in (0,1/4).$$
	Furthermore, the corresponding conclusions to those in Theorem \ref{thm:main-reg} hold, depending on the form of $\BC$, now with the estimates in terms of $F_{V,\BC}$ rather than $E_{V,\BC^{(0)}}$. Here, $C = C(n,Q,\BC^{(0)})\in (0,\infty)$ and $\alpha = \alpha(n,Q,\BC^{(0)})\in (0,1)$.
\end{thmx}

\textbf{Note:} The case of Theorem \ref{thm:fine-reg} where $\BC$ and $\BC^{(0)}$ have the same level directly implies Theorem \ref{thm:main-reg}, since then one can take $\BC = \BC^{(0)}$; indeed, since $\mathcal{F}_{V,\BC^{(0)}}\geq c(\BC^{(0)})>0$, assumption $(\dagger)$ is automatically true if $\eps_0 = \eps_0(n,Q,\BC^{(0)})$ is sufficiently small. In particular, it suffices to prove Theorem \ref{thm:fine-reg} to complete the proof of Theorem \ref{thm:main-compactness}. Moreover, under the assumptions of Theorem \ref{thm:fine-reg} recall that from Theorem \ref{thm:main-L^2-estimate} we know that $F_{V,\BC}\leq CE_{V,\BC}$, and so one can alternatively replace instances of $F_{V,\BC}$ in the conclusions of Theorem \ref{thm:fine-reg} by $E_{V,\BC}$.

\medskip

Our proof of Theorem \ref{thm:fine-reg} will be by \emph{downwards} induction on the level $i$ of $\BC$; as such, the finest cones are the base case. Recalling Figure \ref{fig:0}, it is perhaps useful to bear in mind that the finest cones will always be in either $\P_{\max\{0,n+1-Q\}}$ or $\H_{n-1}(2Q)$, and so in particular are comprised of sums of \emph{multiplicity one} (half-)hyperplanes.

As such, the following inductive hypothesis will be used:
\begin{itemize}
	\item \textbf{Hypothesis F:} Theorem \ref{thm:fine-reg} holds for all $\BC\in\mathfrak{C}_{i}$ with $i>i_0$.
\end{itemize}
In the base case of Theorem \ref{thm:fine-reg}, when $\BC$ is finest, as all (half-)hyperplanes in $\BC$ have multiplicity one, we will be able to prove Theorem \ref{thm:fine-reg} directly using the properties of blow-ups established in Section \ref{sec:blow-up}. This is consistent with Hypothesis F being vacuous in this case.

\section{$C^{1,\alpha}$ Regularity of Blow-Ups}\label{sec:reg}

In this section, we will use Theorem \ref{thm:main-L^2-estimate}, Corollary \ref{cor:main-L^2-estimate}, and Hypothesis F to prove $C^{1,\alpha}$ regularity of the blow-ups constructed in Section \ref{sec:blow-up}, namely those in the class $\mathfrak{B}_{i_0}$. As such, our blow-up functions arise from sequences $(V_j)_j\subset \V_Q$ and $(\BC_j)_j\subset\mathfrak{C}_{i_0}$ which satisfy, for some sequences $\eps_j,\beta_j\downarrow 0$,
\begin{itemize}
	\item Hypothesis $H(\eps_j)$;
	\item Hypothesis $(\dagger)$: $F_{V_j,\BC_j} \leq \beta_j\mathcal{F}_{V_j,\BC_j}$.
\end{itemize}

\begin{remark}
	We stress that one does not need the regularity results of this section in order to conclude the inductive proofs of Theorem \ref{thm:main-L^2-estimate} and Corollary \ref{cor:main-L^2-estimate}. In fact, if desired one could verify the $C^{0,\alpha}$ regularity of blow-ups within the inductive proof of Theorem \ref{thm:main-L^2-estimate} and Corollary \ref{cor:main-L^2-estimate}.
\end{remark}

\smallskip

Recall that $\D$ denoted the Hausdorff limit of the sets $\D_j := \{z\in B_1:\Theta_{V_j}(z)\geq Q\}$ of good density points. Depending on the form of the cones $\BC_j$ and the set $\D$, as described in the two cases for constructing blow-ups in Section \ref{sec:blow-up}, regularity of blow-ups takes a different meaning in each case:
\begin{itemize}
	\item If $\D$ is a proper subset of an $(n-1)$-dimensional subspace in $B_1$, our blow-ups are comprised of functions defined on subsets of the plane $P_*$. Initially, there are only known to be $C^2$ away from $\D$. We will in fact show that are in $C^{1,\alpha}(P_*\cap B_{1/2})$.
	\item If $\D$ is equal to an $(n-1)$-dimensional subspace (namely, $S(\BC^{(0)})$) in $B_1$, then our blow-ups are comprised of functions defined on the half-hyperplane $H_*$. In this case, we will show that they are in fact in $C^{1,\alpha}(\overline{H_*}\cap B_{1/2})$.
\end{itemize}

\subsection{$L^2$ convergence along the blow-up sequence}

Here we justify that one has strong convergence in $L^2$ along the blow-up sequence, namely all the way up to the spine. We prove this for $v\in \FB_{i_0}$. Let $(V_j)_j\subset\V$, $(\BC_j)_j\subseteq \mathfrak{C}_{i_0}$ be the sequences generating $v$, and let us write
$$v_j^{(i,k,s)} := E_{V_j,\BC_j}^{-1}u_j^{(i,k,s)}$$
for the sequence obeying $v_j^{(i,k,s)}\to v^{(i,k,s)}$, where $u_j^{(i,k,s)}$ is as in Section \ref{sec:blow-up}.

To prove that the convergence is in $L^2_{\text{loc}}(B_1)$, by which we mean in $L^2(U)$ for any open set $U$ which is contained within a compact subset of $\overline{\Omega_*}\cap B_1$, notice that Theorem \ref{thm:main-L^2-estimate} gives strong $L^2$ convergence on open subsets of $\Omega_*\setminus \D$. Thus, in fact all we need to show is the following $L^2$ non-concentration estimate about $\D$.

\begin{lemma}\label{lemma:non-con-prelim}
	Fix $\tau\in (0,1/2056)$. Then, there exist $\eps_0,\beta_0$ depending only on $n,Q,\BC^{(0)},\tau$ such that the following is true. If $V,\BC$ satisfy the assumptions of Theorem \ref{thm:main-L^2-estimate} with $\eps_0,\beta_0$ therein, then
	$$\int_{B_\tau(\D_V)\cap B_{15/16}^{n+1}(0)}\dist^2(x,\BC)\, \ext\|V\|(x) \leq C\tau^{n-\sfrak-1/4}E^2_{V,\BC}.$$
	Here, $\D_V = \{z\in \spt\|V\|\cap B_1:\Theta_V(z)\geq Q\}$, $\sfrak = \dim S(\BC)$, and $C = C(n,Q,\BC^{(0)})\in (0,\infty)$ is independent of $\tau$.
\end{lemma}

\begin{proof}
	For any $\tau>0$, provided $\eps_0,\beta_0$ are sufficiently small depending on $n,Q,\BC^{(0)},\tau$, by Theorem \ref{thm:main-L^2-estimate}(1) we know that $\D_V\subset B_{\tau}(S(\BC))$. Since $S(\BC)$ is a subspace of dimension $\sfrak$, we may choose a finite cover $\{B_{2\tau}(y_j)\}_{j=1}^N$ where $y_j\in S(\BC)$ for all $j$ such that
	\begin{itemize}
		\item $B_{\tau}(S(\BC))\subset \bigcup_{i=1}^N B_{2\tau}(y_j)$; in particular, $\D_V\subset \bigcup_{i=1}^N B_{2\tau}(y_j)$;
		\item $N \leq C\cdot\frac{\tau^{n+1-\sfrak}}{\tau^{n+1}} = C\tau^{-\sfrak}$, where $C = C(n)$.
	\end{itemize}
	For each $j$ such that $\D_V\cap B_{3\tau}(y_j)\neq \emptyset$, choose $z_j \in \D_V\cap B_{3\tau}(y_j)$; write $\mathcal{J}$ for the set of such $j$. Then we have $B_{\tau}(\D_V)\subset \bigcup_{j\in \mathcal{J}}B_{4\tau}(z_j)$, and $|\mathcal{J}|\leq N \leq C\tau^{-\sfrak}$.
	
	For any $z\in \D(V)\cap B_{125/126}^{n+1}(0)$, if one takes $\eps_0,\beta_0$ sufficiently small in Corollary \ref{cor:main-L^2-estimate}(2) with $\rho = 1/256$ therein, we get
	$$\int_{B_{1/512}(z)}\frac{\dist^2(x,(\tau_z)_\#\BC)}{|x-z|^{n+7/4}}\, \ext\|V\|(x) \leq CE^2_{V,\BC}$$
	where $C = C(n,Q,\BC^{(0)})$. Thus, we have
	\begin{align*}
	\int_{B_{1/512}(z)}\frac{\dist^2(x,\BC)}{|x-z|^{n-1/4}}&\, \ext\|V\|(x)\\
	& \leq \int_{B_{1/512}(z)}\frac{\dist^2(x,(\tau_z)_\#\BC)}{|x-z|^{n-1/4}}\, \ext\|V\|(x) + C\dist^2_\H(\BC\res B_1,(\tau_z)_\#\BC\res B_1)\\
	& \leq C\int_{B_{1/512}}\frac{\dist^2(x,(\tau_z)_\#\BC)}{|x-z|^{n+7/4}}\, \ext\|V\|(x) + C\dist^2_\H(\BC\res B_1,(\tau_z)_\#\BC\res B_1)\\
	& \leq CE_{V,\BC}^2 + C\dist^2_\H(\BC\res B_1, (\tau_z)_\#\BC\res B_1).
	\end{align*}
	Writing $\xi:= \pi^\perp_{S(\BC)}(z)$, recalling that (cf.~\cite[(3.4)--(3.5)]{BK17})
	$$\dist_\H(\BC\res B_1,(\tau_z)_\#\BC\res B_1) \leq C\max_i\{|\pi^\perp_{P^{(i)}}(\xi)| + E_{V,\BC^{(0)}}|\pi_{P^{(i)}}(\xi)|\}$$
	and so by Corollary \ref{cor:main-L^2-estimate}(1) we have
	$$\dist_\H(\BC\res B_1, (\tau_z)_\#\BC\res B_1) \leq CE_{V,\BC},$$
	and so combining we have
	$$\int_{B_{1/512}(z)}\frac{\dist^2(x,\BC)}{|x-z|^{n-1/4}}\, \ext\|V\|(x) \leq CE_{V,\BC}^2.$$
	In particular, for any $j\in \mathcal{J}$, we have
	$$\int_{B_{4\tau}(z_j)}\dist^2(x,\BC)\, \ext\|V\|(x) \leq C\tau^{n-1/4}E_{V,\BC}^2.$$
	If we sum this up over $j$, we get
	\begin{align*}
	\int_{B_\tau(\D_V)\cap B^{n+1}_{15/16}(0)}\dist^2(x,\BC)\, \ext\|V\|(x) & \leq \sum_{j\in \mathcal{J}}\int_{B_{4\tau}^{n+1}(z_j)}\dist^2(x,\BC)\, \ext\|V\|(x)\\
	& \leq N\cdot C\tau^{n-1/4}E_{V,\BC}^2\\
	& \leq C\tau^{n-\sfrak-1/4}E_{V,\BC}^2,
	\end{align*}
	which is the claimed result.
\end{proof}	

As already explained, as a direct consequence of Lemma \ref{lemma:non-con-prelim} and Theorem \ref{thm:main-L^2-estimate}(1) we get the full $L^2$ non-concentration result about $S(\BC)$.

\begin{lemma}\label{lemma:non-con}
	Fix $\tau\in (0,1/2056)$. Then, there exist $\eps_0,\beta_0$ depending only on $n,Q,\BC^{(0)},\tau$ such that the following is true. If $V,\BC$ satisfy the assumptions of Theorem \ref{thm:main-L^2-estimate} with $\eps_0,\beta_0$ therein, then
	$$\int_{B_\tau(S(\BC))\cap B_{15/16}^{n+1}(0)}\dist^2(x,\BC)\, \ext\|V\|(x) \leq C\tau^{n-\sfrak-1/4}E^2_{V,\BC}.$$
	Here, $\sfrak = \dim S(\BC) \leq n-1$ and $C = C(n,Q,\BC^{(0)})\in (0,\infty)$ is independent of $\tau$.
\end{lemma}

\subsection{$C^{0,\alpha}$ regularity of blow-ups}

We now establish the first preliminary regularity result for coarse blow-ups, which is their Hölder regularity.

We start with the main $L^2$ estimate. Here, we recall the notation $m^{(i,k)}$, $\lambda^{(i,k)}(z)$, and $\widetilde{\lambda}^{(i,k)}(z)$ for $z\in \D\cap B_1$ from Section \ref{sec:first-blow-up} (which we stress depends on the specific blow-up $v\in \FB_{i_0}$). By construction, $\widetilde{\lambda}^{(i,k)}$ obeys $|\widetilde{\lambda}^{(i,k)}(z)|\leq C(n,Q,\BC^{(0)})$. To simplify the notation, we will write $\kappa^{(i,k)}:= \widetilde{\lambda}^{(i,k)}$.

\begin{lemma}\label{lemma:holder-start}
	Let $v = (v^{(i,k,s)})_{i,k,s}\in \FB_{i_0}$. Then for each $z\in \D\cap B_{1/2}(0)$ and $\rho\in (0,1/8)$ we have
	\begin{align*}
		\sum_{i,k,s}\int_{\Omega_*\cap B_{5\rho/8}(z)}\frac{|v^{(i,k,s)}(x)-\kappa^{(i,k)}(z)|^2}{|x-z|^{n+7/4}}&\, \ext x \leq C\rho^{-n-7/4}\sum_{i,k,s}\int_{\Omega_*\cap B_\rho(z)}|v^{(i,k,s)}(x)-\kappa^{(i,k)}(z)|^2\, \ext x.
	\end{align*}
	Moreover, $\kappa$ is independent of the choice of good density points $z_j$ which give rise to it, and satisfies
	$$|\kappa^{(i,k)}(z)|^2\leq C\sum_{i,k,s}\int_{\Omega_*\cap B_1(0)}|v^{(i,k,s)}|^2.$$
	Here, $C = C(n,Q,\BC^{(0)})$.
\end{lemma}

\begin{proof}
	Fix $v\in \FB_{i_0}$ and let $(V_j)_j\subset\V_Q$, $(\BC_j)_j\subset\mathfrak{C}_{i_0}$ be the sequences which give rise to $v$. Fix also $z\in \D\cap B_{1/2}(0)$, and let $z_j\in \D_j$ be the sequence of good density points of $V_j$ with $z_j\to z$. If one takes the estimate in Corollary \ref{cor:main-L^2-estimate}(2) and writes the left-hand side in terms of the graphical representation provided by Theorem \ref{thm:main-L^2-estimate}, for a suitable sequence $(\tau_j)_j$ with $\tau_j\downarrow 0$ sufficiently slowly, we get
	\begin{align*}
		&\sum_{i,k,s}\int_{\Omega^{(i,k)}\cap B_{5\rho/8}(\pi_{P^{(i,k)}}(z_j))\setminus B_{2\tau_j}(\sing(\BC_j))}\frac{|u_j^{(i,k,s)}(x) - (\pi^\perp_{P^{(i)}}(\xi^{(i,k)}_j) - D\ell_j^{(i,k)}\pi_{P^{(i)}}(\xi^{(i,k)}_j))|^2}{|(x,u_j^{(i,k,s)}(x)) - z_j|^{n+7/4}}\, \ext x\\
		& \hspace{35em}  \leq CE_{V_j,\BC_j}^2,
	\end{align*}
	where we have also used Corollary \ref{cor:main-L^2-estimate}(1) in order to control the right-hand side of Corollary \ref{cor:main-L^2-estimate}(2) by $E_{V_j,\BC_j}^2$. Notice that on the left-hand side, in order to know that
	$$\dist((x,u^{(i,k,s)}_j(x)), (\tau_{z_j})_\#\BC_j) = |u_j^{(i,k,s)}(x) - (\pi^\perp_{P^{(i)}}(\xi^{(i,k)}_j) - D\ell_j^{(i,k)}\pi_{P^{(i)}}(\xi^{(i,k)}_j))|,$$
	one must stay away from a neighbourhood of $\sing(\BC_j)$, which is why this is removed on the left-hand side. If we now divide both sides by $E_{V_j,\BC_j}^2$ and take $j\to\infty$ we get, using the strong $L^2$ convergence provided by Lemma \ref{lemma:non-con},
	$$\sum_{i,k,s}\int_{\Omega_*\cap B_{5\rho/8}(z)}\frac{|v^{(i,k,s)}(x) - \kappa^{(i,k)}(z)|^2}{|x-z|^{n+7/4}}\, \ext x \leq C.$$
	This integral shows that the limit $\kappa^{(i,k)}(z)$ is independent of the choice of good density points $z_j$: indeed, if one chose a different sequence $\tilde{z}_j\in \D_j$ with $\tilde{z}_j\to z$, with corresponding limit $\tilde{\kappa}^{(i,k)}(z)$, we would similarly have
	$$\sum_{i,k,s}\int_{\Omega_*\cap B_{5\rho/8}(z)}\frac{|v^{(i,k,s)}(x) - \tilde{\kappa}^{(i,k)}(z)|^2}{|x-z|^{n+7/4}}\, \ext x \leq C$$
	and so combining the two inequalities we would get
	$$\sum_{i,k,s}\int_{\Omega_*\cap B_{5\rho/8}(z)}\frac{|\kappa^{(i,k)}(z) - \tilde{\kappa}^{(i,k)}(z)|^2}{|x-z|^{n+7/4}}\, \ext x \leq C.$$
	Since the numerator of the integrand is a constant and $1/|x|^{n+7/4}$ is not integrable on $B_1(0)$, this implies we need $\kappa^{(i,k)}(z) = \tilde{\kappa}^{(i,k)}(z)$, proving the claimed independence.
	
	To prove the claimed inequality, one instead directly blows-up Corollary \ref{cor:main-L^2-estimate}(2) without bounding the right-hand side by $E_{V_j,\BC_j}^2$. Indeed, we have
	\begin{align*}
		&\sum_{i,k,s}\int_{\Omega^{(i,k)}\cap B_{5\rho/8}(\pi_{P^{(i,k)}}(z_j))\setminus B_{2\tau_j}(\sing(\BC_j))}\frac{|u_j^{(i,k,s)}(x) - (\pi^\perp_{P^{(i)}}(\xi^{(i,k)}_j) - D\ell_j^{(i,k)}\pi_{P^{(i)}}(\xi^{(i,k)}_j))|^2}{|(x,u_j^{(i,k,s)}(x)) - z_j|^{n+7/4}}\, \ext x\\
		& \hspace{20em}  \leq C\rho^{-n-7/4}\int_{B_\rho(z_j)}\dist^2(x,(\tau_{z_j})_\#\BC_j)\, \ext\|V_j\|(x).
	\end{align*}
	Writing the right-hand side in the same manner as the left-hand side (on the appropriate graphical region; the remaining error term will still $\to 0$ as $j\to\infty$, even after dividing by $E_{V_j,\BC_j}^2$), then dividing by $E^2_{V_j,\BC_j}$ and taking $j\to\infty$, we get
	\begin{align*}
		\sum_{i,k,s}\int_{\Omega_*\cap B_{5\rho/8}(z)}\frac{|v^{(i,k,s)}(x)-\kappa^{(i,k)}(z)|^2}{|x-z|^{n+7/4}}&\, \ext x \leq C\rho^{-n-7/4}\sum_{i,k,s}\int_{\Omega_*\cap B_\rho(z)}|v^{(i,k,s)}(x)-\kappa^{(i,k)}(z)|^2\, \ext x.
	\end{align*}
	as claimed. Finally, to see the claimed bound on $\kappa^{(i,k)}$, if we apply the same blow-up procedure but with $\widetilde{V}_j:= (\eta_{0,7/8})_\#V_j$ in place of $V_j$ (which we may do, cf.~Theorem \ref{thm:blow-up-properties})(2)), then Corollary \ref{cor:main-L^2-estimate} would instead give $|\pi^{\perp}_{P^{(i)}}(\xi_j)| + E_{V_j,\BC^{(0)}}|\pi_{P^{(i)}}(\xi_j)| \leq CE_{(\eta_{0,7/8})_\#V_j,\BC_j}$, which if we use in our construction of $\kappa^{(i,k)}$ results in
	$$|\kappa^{(i,k)}(z)|^2 \leq C\sum_{i,k,s}\int_{\Omega_*\cap B_{7/8}(0)}|v^{(i,k,s)}|^2$$
	which gives the result.
\end{proof}

We now use Lemma \ref{lemma:holder-start} to prove the Hölder regularity of blow-ups.

\begin{theorem}\label{thm:holder}
	If $v\in \FB_{i_0}$, then in fact $v^{(i,k,s)}\in C^{0,\alpha}(\overline{\Omega_*}\cap B_1)$ for all $i,k,s$, for some $\alpha = \alpha(n,Q,\BC^{(0)})\in (0,1)$. Furthermore, we have
	$$\|v^{(i,k,s)}\|_{C^{0,\alpha}(\overline{\Omega_*}\cap B_{1/4}(0))} \leq C\|v\|_{L^2(B_{1/2}(0))}.$$
	Here, $\|v\|_{L^2(B_{1/2}(0))}^2:= \sum_{i,k,s}\|v^{(i,k,s)}\|_{L^2(\Omega_*\cap B_{1/2}(0))}^2$ and $C = C(n,Q,\BC^{(0)})\in (0,\infty)$.
\end{theorem}

\begin{proof}
	From Lemma \ref{lemma:holder-start}, we get that for each $z\in \D\cap B_{1/4}(0)$ and $0<\sigma<5\rho/8<5/64$ we have
	$$\sum_{i,k,s}\sigma^{-n}\int_{\Omega_*\cap B_\sigma(z)}|v^{(i,k,s)}(x)-\kappa^{(i,k)}(z)|^2\, \ext x \leq C\left(\frac{\sigma}{\rho}\right)^{7/4}\sum_{i,k,s}\rho^{-n}\int_{\Omega_*\cap B_\rho(z)}|v^{(i,k,s)}(x)-\kappa^{(i,k)}|^2\, \ext x.$$
	For $z\in B_{1/4}(0)\setminus \D$, we know from Theorem \ref{thm:blow-up-properties} and Theorem \ref{thm:main-L^2-estimate}(1) that each $v^{(i,k,s)}$ is harmonic on $\Omega_*\cap B_\rho(z)$ for $0<\rho<\dist(z,\D)$. Hence, by harmonic estimates, we have for any choice of constants $\bar{\kappa}^{(i,k)}$,
	$$\sum_{i,k,s}\sigma^{-n}\int_{\Omega_*\cap B_\sigma(z)}|v^{(i,k,s)}(x) - v^{(i,k,s)}(z)|^2\, \ext x \leq C\left(\frac{\sigma}{\rho}\right)^2\sum_{i,k,s}\rho^{-n}\int_{\Omega_*\cap B_\rho(z)}|v^{(i,k,s)}(x) - \bar{\kappa}^{(i,k)}|^2\, \ext x$$
	for any $0<\sigma\leq\rho<\dist(z,\D)$. It follows from the above two inequalities by degenerate Campanato estimates (cf.~\cite[Lemma 12.1 \& Lemma 4.3]{Wic14}) that we in fact have, setting $v^{(i,k,s)}(z):= \kappa^{(i,k)}(z)$ for $z\in \D\cap B_{1/4}(0)$,
	$$\sum_{i,k,s}\sigma^{-n}\int_{\Omega_*\cap B_\sigma(z)}|v^{(i,k,s}(x) - v^{(i,k,s)}(z)|^2\, \ext x \leq C\sigma^{2\alpha}\sum_{i,k,s}\int_{\Omega_*\cap B_{1/2}(0)}|v^{(i,k,s)}(x)|^2\, \ext x$$
	for any $\sigma\in (0,r]$, for suitable $r>0$ and $\alpha\in (0,1)$ depending only on $n,Q,\BC^{(0)}$. Here, we have also used that $|\kappa^{(i,k)}(z)| \leq C\|v\|_{B_{1/2}^n(0)}$ for $z\in \D\cap B_{1/4}(0)$ (such an estimate follows analogously to the corresponding estimate to that in Lemma \ref{lemma:holder-start}, making advantage of the smaller ball $\D\cap B_{1/4}(0)$ which the $z$ lie in). It is then standard to conclude from the above estimate (using Campanato theory) that $v^{(i,k,s)}\in C^{0,\alpha}(\overline{\Omega_*}\cap B_{1/4}(0))$ with the claimed estimate.
\end{proof}

\subsection{Further properties of blow-ups}

In this subsection we give two further properties of blow-ups which will be used in the proof of their $C^{1,\alpha}$ regularity.

\begin{remark}\label{remark:reg-simple}
It is worth noticing at this stage that if the spine dimension $\sfrak$ of the sequence of cones $(\BC_j)_j$ along which we created our blow-up $v$ satisfies $\sfrak \leq n-2$, then as by Theorem \ref{thm:holder} each $v^{(i,k,s)}$ is an $C^{0,\alpha}$ function on $P_*\cap B_1$ which is harmonic on $(P_*\setminus S)\cap B_1$, where $S$ is a subspace of dimension $\leq n-2$, standard removable singularity results for harmonic functions imply that $v$ must be harmonic on all of $P_*\cap B_1$, and hence $v$ is automatically smooth (and in fact harmonic). Thus, in such cases the regularity of blow-ups is substantially easier, and so we may primarily focus on the case where $\sfrak = n-1$
\end{remark}

\smallskip

The first additional observation is that, if $v\in\mathfrak{B}_{i_0}$ and $\kappa$ denotes the function $\mathcal{D}\cap B_{1/2}\to S^\perp$ which determines the values of $v$ on $\mathcal{D}\subseteq S$ (cf.~Lemma \ref{lemma:holder-start}), then $\kappa$ is the restriction of a smooth function defined on $S$; furthermore, this function must be linear is $v$ is homogeneous of degree one.

\begin{lemma}\label{lemma:kappa}
	Let $v\in \mathfrak{B}_{i_0}$. Then, $\kappa: \D\cap B_{1/2}(0)\to S^\perp$ is the restriction of a smooth function $\kappa_* :\Omega_*\cap B_{1/2}(0)\to S^\perp$ to $\D$. Furthermore, if $v$ is homogeneous of degree one, $\kappa_*$ must be linear.
\end{lemma}

\begin{proof}
	If $\sfrak\leq n-2$ the claim is trivial, since then each $v^{(i,k,s)}$ is a continuous function on $P_*\cap B_1$ (by Theorem \ref{thm:holder}) which is harmonic on $P_*\cap B_1\setminus \D$ (by Theorem \ref{thm:blow-up-properties}(1)). By standard removability results for harmonic functions, this implies that $v^{(i,k,s)}$ is a harmonic function on $P_*\cap B_1$, and so in particular is smooth. The claims then follow immediately from this.
	
	So, we may suppose $\sfrak=n-1$; in particular, we must have $\BC^{(0)}\in \CC_{n-1}$. In this situation, $\kappa$ is determined by the projection of the good density points to a fixed, $2$-dimensional cross-section. Since $\BC^{(0)}$ is not flat, the rays determining its cross-section span $\R^2$, and as such one can use the averaging argument as seen originally in \cite[Proof of Lemma 1]{Sim93} (cf.~\cite[Section 4.3]{Min21a}) to prove the result.
\end{proof}

The second additional property of blow-ups we need is the following $\eps$-regularity property, which arises from Hypothesis F, namely the inductive assumption of the fine $\eps$-regularity theorem. For this, we introduce the notion of a \emph{dehomogeniser}. For $v\in \FB_{i_0}$, $z\in \D$ with $\kappa(z) = 0$, and $\rho\in (0,1]$, we say that $\pi\in \mathfrak{L}_{i_0}$ \emph{dehomogenises} $v$ in $B_\rho(z)$ if
\begin{align*}
\sum_{i,k,s}\int_{B_\rho(z)\cap \Omega_*}&\left|v^{(i,k,s)}(x) - \pi^{(i,k)}\left(\frac{x-z}{\rho}\right)\right|^2\, \ext x\\
& \hspace{5em} = \inf_{\tilde{\pi}\in \mathfrak{L}_{i_0}}\sum_{i,k,s}\int_{B_\rho(z)\cap \Omega_*}\left|v^{(i,k,s)}(x)-\tilde{\pi}^{(i,k)}\left(\frac{x-z}{\rho}\right)\right|^2\, \ext x.
\end{align*}
Notice that our dehomogenisers only ever come from $\mathfrak{L}_{i_0}$. Dehomogenisers always exist by compactness of $\mathfrak{L}_{i_0}$ (as elements of $\mathfrak{L}_{i_0}$ are linear and always have a uniform $L^2$ bound). We write $\pi^v_{z,\rho}(x):= \pi\left(\frac{x-z}{\rho}\right)$ for the (recentered) dehomogeniser of $v$ in $B_\rho(z)$. Moreover, when $\kappa(z)\neq 0$, we can define the dehomogeniser of $v$ at $z$ to be the dehomogeniser of $\hat{v}$ as in Theorem \ref{thm:blow-up-properties}(3) for $\kappa$ therein as that here (given by Lemma \ref{lemma:holder-start}), since then $\hat{v}(z) = 0$.
\begin{lemma}[$\eps$-Regularity for Blow-Ups]\label{lemma:blow-up-eps-reg}
	There exists a constant $\eps = \eps(n,Q,\BC^{(0)})\in (0,1)$ such that the following is true. Suppose $v\in \mathfrak{B}_{i_0}$ has $v(0) = 0$, $v$ is dehomogenised by $0$ in $B_1(0)$, and $\|v\|_{L^2(\Omega_*\cap B_1)} = 1$. If $\psi\in \mathfrak{L}_\ell$ for some $\ell>i_0$ obeys
	$$\sum_{i,k,s}\int_{B_1\cap \Omega_*}|v^{(i,k,s)}-\psi^{(i,k,s)}|^2 < \eps,$$
	then we have $v|_{B_{1/2}\cap \Omega_*} \in C^{1,\alpha}(B_{1/2}\cap\overline{\Omega_*})$, and there exists $\widetilde{\psi}\in \mathfrak{L}_{\tilde{\ell}}$ for some $\tilde{\ell}>i_0$ with
	$$\sum_{i,k,s}\sigma^{-n-2}\int_{B_\sigma\cap \Omega_*}|v^{(i,k,s)}-\widetilde{\psi}^{(i,k,s)}|^2 \leq C\sigma^{2\alpha}.$$
	Here, $\alpha\in (0,1)$ and $C\in (0,\infty)$ depend only on $n,Q,\BC^{(0)}$.
\end{lemma}

\begin{proof}
	The proof is by contradiction, and similar to those in \cite[Theorem 3.4]{MW24} and \cite[Corollary 7.3]{Min21a}, so we only sketch the details. If it were to fail, we could take a sequence of blow-ups $v_j\in \mathfrak{B}_{i_0}$ obeying the hypotheses of the lemma for suitable $\psi_j\in \mathfrak{L}_{\ell}$ with $\ell>i_0$ ($\ell$ can be fixed by passing to a subsequence) yet
	$$\sum_{i,k,s}\int_{B_1\cap\Omega_*}|v^{(i,k,s)}_j-\psi^{(i,k,s)}_j|^2 \to 0.$$
	By passing to a subsequence, using Theorem \ref{thm:blow-up-properties}(5), we know $v_j\to v_*$ where $v_* \in\mathfrak{L}_{\ell_*}$ for some $\ell\geq\ell_*>i_0$ ($\ell_*$ cannot equal $i_0$ as $v_*$ is dehomogenised by $0$, yet has $L^2$ norm $1$). It suffices to prove the result for all sufficiently large $j$. For this, let $(V_j^p)_p$ and $(\BC_j^p)_p$ be the sequence of varifolds whose blow-up is $v_j$. It suffices to show that Theorem \ref{thm:fine-reg} is applicable to $V_j^p$ and $\BC_j^p$ for all $p$ and $j$ sufficiently large, as then the result follows from blowing-up the conclusions of Theorem \ref{thm:fine-reg}. Thus, by taking a suitable sequence $\widetilde{V}_j:= V_j^{p_j}$ and $\widetilde{\BC}_j:= \BC_j^{p_j}$, whose blow-up can be assumed to be $v_*$, it suffices to show that for all $j$ sufficiently large Theorem \ref{thm:fine-reg} is applicable to $\widetilde{V}_j$ and $\widetilde{\BC}_j$. One finds that
	$$E_{\widetilde{V}_j,\widetilde{\BC}_j} \leq \frac{3}{2}\inf_{\Dbf}E_{\widetilde{V}_j,\Dbf}$$
	where here the infimum is taken over cones $\Dbf$ which are the same level as $\widetilde{\BC}_j$ and are aligned. In particular, using Theorem \ref{thm:main-L^2-estimate}, we know that
	\begin{equation}\label{E:blow-up-eps-1}
	F_{\widetilde{V}_j,\widetilde{\BC}_j} \leq C\inf_{\Dbf}F_{\widetilde{V}_j,\Dbf},
	\end{equation}
	where the infimum is over the same class of $\Dbf$. One the other hand, since the blow-up of $\widetilde{V}_j$ relative to $\widetilde{\BC}_j$ is $v_*$, this means that there is a sequence of cones $\widehat{\BC}_j$ which are finer than $\widetilde{\BC}_j$ for which
	$$F_{\widetilde{V}_j,\widehat{\BC}_j} < \gamma_j F_{\widetilde{V}_j,\widetilde{\BC}_j}$$
	for some sequence $\gamma_j\downarrow 0$. Combined with \eqref{E:blow-up-eps-1}, this gives
	$$F_{\widetilde{V}_j,\widehat{\BC}_j} < C\gamma_j \inf_{\Dbf}F_{\widetilde{V}_j,\Dbf}$$
	where again the infimum is taken over cones which are finer than $\widetilde{\BC}_j$. This is not enough to apply Theorem \ref{thm:fine-reg}, since we would need the right-hand side to be $\mathcal{F}_{\widetilde{V}_j,\widehat{\BC}_j}$. Of course, if $F_{\widetilde{V}_j,\widehat{\BC}_j}<\beta_0\mathcal{F}_{\widetilde{V}_j,\widehat{\BC}_j}$ for infinitely many $j$, where $\beta_0 = \beta_0(n,Q,\BC^{(0)})\in (0,1)$ is sufficiently small, we can apply Theorem \ref{thm:fine-reg} to conclude. If this were to fail, then we would have
	$$F_{\widetilde{V}_j,\widehat{\BC}_j}\geq \beta_0\mathcal{F}_{\widetilde{V}_j,\widehat{\BC}_j} \qquad \text{for all $j$ sufficiently large}$$
	then choose $\BC^*_j$ coarser than $\widehat{\BC}_j$ with $F_{\widetilde{V}_j,\BC_j^*} \leq \frac{3}{2}\mathcal{F}_{\widetilde{V}_j,\widehat{\BC}_j}$. Then we have
	$$F_{\widetilde{V}_j,\BC_j^*} \leq \frac{3}{2}\beta_0^{-1}\cdot C\gamma_j\inf_{\Dbf}F_{\widetilde{V}_j,\Dbf}$$
	and so in particular $\BC_j^*$ must be finer than $\widetilde{\BC}_j$ for all $j$ sufficiently large. In particular, we may repeat this process finitely many times in order to see that Theorem \ref{thm:fine-reg} is eventually applicable with some finer cone to $\widetilde{\BC}_j$, by Hypothesis F (this process is similar to choosing a replacement). We also know that the cones we choose in this process are always finer than $\widetilde{\BC}_j$ due to \eqref{E:blow-up-eps-1} and by the significant improvement of the fine excess relative to $F_{V_j,\widetilde{\BC}_j}$.
	This completes the proof.
\end{proof}

\subsection{Classification of homogeneous degree one blow-ups}

In this section we will classify all homogeneous degree one elements of $\FB_{i_0}$. Note that we have already done this when the spine dimension of the sequence of cones $\BC_j$ generating the blow-up is $\sfrak\leq n-2$, since then as we saw in Remark \ref{remark:reg-simple}, in this situation $v^{(i,k,s)}$ are harmonic functions on $P_*\cap B_1$, and so if they are homogeneous of degree one then they must be linear. Thus we may focus on the situation where $\sfrak=n-1$. In this situation, $\D$ is a subset of an $(n-1)$-dimensional subspace $S$, which we may assume without loss of generality is $\{0\}^2\times\R^{n-1}$.

For ease of notation, if $v$ is homogeneous of degree one we do not distinguish between $v$ and its homogeneous (degree one) extension to all of $P_*$ or $H_*$ (depending on the form of $\Omega_*$).

\begin{theorem}\label{thm:classification}
	Suppose $v = (v^{(i,k,s)})_{i,k,s}\in \FB_{i_0}$ is homogeneous of degree one, and $S = \{0\}^2\times\R^{n-1}$. Then, $v\in \mathfrak{L}_{\ell}$ for some $\ell\geq i_0$; in particular, $v$ satisfies one of the following depending on the form of $\D$.
	\begin{enumerate}
		\item [\textnormal{(i)}] If $\D$ is a proper subset of $S\cap B_1$, then each $v^{(i,k,s)}$ is given by a linear function on $P_*$;
		\item [\textnormal{(ii)}] If $\dim S(\BC^{(0)}) = n-1$ and $\D = S(\BC^{(0)})\cap B_1$, then each $v^{(i,k,s)}$ is the restriction of a linear function on $P_*$ to $H_*$.
	\end{enumerate}
\end{theorem}
\begin{remark}
	It is important, especially in Theorem \ref{thm:classification}(ii), to know that along $\D$ the values of the individual $v^{(i,k,s)}$ are compatible with each other in the sense that they arise from $\kappa\in \R^2\times\{0\}^{n-1}$ which is in turn determined by $\lambda^{(i,k)}$ and $m^{(i,k)}$. Notice that in this situation, as $\sfrak=n-1$, Corollary \ref{cor:main-L^2-estimate} actually gives $|\xi|\leq CE_{V,\BC}$ for $\xi = z^{\perp_{S(\BC)}}$, where $z$ is any good density point.
\end{remark}

\begin{proof}
	Throughout, we fix $v = (v^{(i,k,s)})_{i,k,s}\in \FB_{i_0}$ which is homogeneous of degree one. We identify $v$ with its homogeneous degree one extension (namely, $v^{(i,k,s)}$ is extended to all of $P_*$ of $H_*$ depending on the form of its domain). We will therefore use notation not restricted to $B_1$.
	
	The proof is inspired by those seen in \cite{BK17, MW24, Min21b, BKMW25}. Let $(V_j)_j\subset\V_Q$ and $(\BC_j)_j\subset\CC$ be the sequences of varifolds and cones which give rise to $v$. By assumption we know $\sfrak\equiv \dim(S(\BC_j))=n-1$. In fact, if $\H^{n-2}(\D\cap B_1)<\infty$, then $\D$ has vanishing $2$-capacity, and so each $v^{(i,k,s)}$ is a continuous function on $P_*\cap B_1$ which is harmonic on $P_*\cap B_1\setminus \D$, and by standard removability theorems for harmonic functions we would have $v^{(i,k,s)}$ is harmonic on all of $P_*\cap B_1$, and so is linear, completing the proof. So, we may assume that not only $\sfrak=n-1$, but also $\H^{n-2}(\D\cap B_1)=\infty$.
	
	At the opposite extreme, if $\D = S(\BC^{(0)})$ then the proof is also standard. Indeed, as $\BC_j\in \CC_{n-1}$ for all $j$, the proof in this case is identical to that in \cite[Theorem 16.7]{Wic14}. In more detail, by Lemma \ref{lemma:kappa} one knows that the boundary values of each $v^{(i,k,s)}$ along $\D\cap B_1 = S\cap B_1$ are given by a linear function, and so standard boundary regularity for harmonic functions implies that $v^{(i,k,s)}$ are $C^{\infty}$ up to the boundary on each half-hyperplane. But then $Dv^{(i,k,s)}$ is a continuous, homogeneous degree zero function on each half-hyperplane, which implies it must be constant, and thus $v^{(i,k,s)}$ is linear.
	
	So, we are left with the case where $\sfrak=n-1$, $\H^{n-2}(\D\cap B_1)=\infty$, yet $\D\subsetneq S$. Notice then that by Lemma \ref{lemma:no-gaps} we must have $\BC_j\in \P_{n-1}$ for all $j$, and hence $\BC^{(0)}\in \P_{n-1}$. We break the proof down into several steps.
	
	\textbf{Step 1:} \emph{Induction setup.} For any $v$ as above, define
	$$S(v) := \{z\in S:v(x+z) = v(x) \text{ for all }x\},$$
	i.e.~the set of points in $S$ under which each $v^{(i,k,s)}$ is translation invariant. Using homogeneity of $v$ it is easy to verify that $S(v)$ is a linear subspace of $S$, and thus has a dimension $d:= \dim S(v)$ which belongs to $\{0,1,\dotsc,n-1\}$. Clearly if $d=n-1$ then the proof is complete (indeed, since $\D\subsetneq S\cap B_1$, we know that each $v^{(i,k,s)}$ is defined on the full hyperplane $P_*$ and is harmonic on some ball centered on $S$).
	
	So, we may assume that $d\leq n-2$. We will prove the theorem by working inductively \emph{downwards} on the variable $d$. Thus, we may fix $d\in \{n-2,n-3,\dotsc,0\}$ and assume that we have shown the result for all larger values of $d$. In particular, as $S(v)$ is a subspace of dimension at most $n-2$ and $\H^{n-2}(\D\cap B_1)=\infty$, we must have $\D\setminus S(v)\neq\emptyset$.
	
	\textbf{Step 2:} \emph{The reverse Hardt--Simon dichotomy.} For $z\in \D\cap B^n_{1/4}(0)$ and $\rho\in (0,1/4)$, recall that $\pi^v_{z,\rho}\equiv \pi_{z,\rho}\in\mathfrak{L}_{i_0}$ is the dehomogeniser of $v$. If $v\equiv v(z) + \pi_{z,\rho}$ on $B_\rho(z)$, then necessarily $v\in\mathfrak{L}_{i_0}$ (by unique continuation for harmonic functions) and so there is nothing to prove. Hence, we may suppose (looking for a contradiction) that this fails for each $z\in \D\cap B^n_{1/4}(0)$. Fix any compact subset $K$ of $B_{1/4}^n(0)\setminus S(v)$. We then claim the following \emph{reverse Hardt--Simon dichotomy}:
	
	\textbf{Claim:} There exists $\eps = \eps(v,K)\in (0,1)$ such that for any $z\in \D\cap K$ and any $\rho\in (0,\eps]$, we either have:
	\begin{enumerate}
		\item [(a)] the \emph{reverse Hardt--Simon inequality}
	\begin{equation}\label{E:RHS}
		\int_{B_\rho(z)\setminus B_{\rho/2}(z)}R_z^{2-n}\left|\frac{\del}{\del R_z}\left(\frac{v-v(z)}{R_z}\right)\right|^2 \geq \eps\rho^{-n-2}\int_{B_\rho(z)}|v - (v(z)+\pi_{z,\rho})|^2,
	\end{equation}
	where here $R_z\equiv R_z(x):=|x-z|$, and for simplicity we have suppressed notation which involves summing over $i,k,s$;
		\item [(b)] or the conclusions of Lemma \ref{lemma:blow-up-eps-reg} hold on $B_{3\rho/8}(z)$; in particular, there exists $\psi_z \in \mathfrak{L}_{\ell_z}$ for some $\ell_z>i_0$ such that for all $0<\sigma\leq 3\rho/8$,
	\begin{equation}\label{E:RHS-alt}
		\sigma^{-n-2}\int_{B_\sigma(z)}|v-(v(z)+\psi_z)|^2 \leq C\left(\frac{\sigma}{\rho}\right)^{2\alpha}\cdot\rho^{-n-2}\int_{B_\rho(z)}|v-(v(z)+\psi_z)|^2.
	\end{equation}
	Here, $C = C(n,Q,\BC^{(0)})\in (0,\infty)$ and $\alpha = \alpha(n,Q,\BC^{(0)})\in (0,1)$.
	\end{enumerate}
	
	(We remark that, using Lemma \ref{lemma:kappa} and Theorem \ref{thm:blow-up-properties} at the start of the proof, we could without loss of generality assume that $\kappa\equiv 0$, giving that $v\equiv 0$ on $\D$.)
	
	Indeed, if this claim were false, then there would exist a sequence of points $z_j\in \D\cap K$ and radii $\rho_j\downarrow 0$ with $v\not\equiv v(z_j)+\pi_{z_j,\rho_j}$ on $B_{\rho_j}(z_j)$ and a sequence $\eps_j\downarrow 0$ such that, setting $\pi_j := \pi_{z_j,\rho_j}$ (and recalling that $\pi_{z_j,\rho_j}$ is homogeneous of degree one centered at $z_j$),
	$$\int_{B_{\rho_j}(z_j)\setminus B_{\rho_j/2}(z_j)}R_{z_j}^{2-n}\left|\frac{\del}{\del R_{z_j}}\left(\frac{v-(v(z_j)+\pi_j)}{R_{z_j}}\right)\right|^2 < \eps_j\rho_j^{-n-2}\int_{B_{\rho_j}(z_j)}|v-(v(z_j)+\pi_j)|^2.$$
	By Theorem \ref{thm:blow-up-properties}(2)--(4), we have that
	$$w_j := \frac{v(z_j+\rho_j(\cdot))-(v(z_j) + \pi_j(z_j+\rho_j(\cdot)))}{\|v(z_j+\rho_j(\cdot)) - (v(z_j)+\pi_j(z_j+\rho_j(\cdot)))\|_{L^2(B_1)}}$$
	obeys $w_j\in \mathfrak{B}_{i_0}$. Notice here that we are using that $z_j\in \D$ and so $v^{(i,k,s)}(z_j) = \kappa^{(i,k)}(z_j)$ take the desired form for the subtraction in Theorem \ref{thm:blow-up-properties}(3). We may then apply the compactness property, Theorem \ref{thm:blow-up-properties}(5), to see that, up to passing to a subsequence, $w_j\to w_*\in \mathfrak{B}_{i_0}$, where the convergence is locally smoothly on $\Omega_*\setminus \D$ and locally uniformly on $\Omega_*$ (by Theorem \ref{thm:holder}). By construction, we then know that $w_*(0) = 0$, and $w_*$ is dehomogenised by $0$ in $B_1(0)$. Moreover, it follows from the construction of $w_*$ and an 1-dimensional integration-along-rays argument that $w_*\not\equiv 0$ (cf.~\cite[(5.25)]{BK17}, for instance). Moreover, our contradiction assumption reads
	$$\int_{B_1(0)\setminus B_{1/2}(0)}R^{2-n}\left|\frac{\del}{\del R}\left(\frac{w_j}{R}\right)\right|^2 < \eps_j,$$
	where $R \equiv R(x):= |x|$, which implies that $w_*$ must be homogeneous of degree one on $B_1(0)\setminus B_{1/2}(0)$, and so by unique continuation of harmonic functions must be homogeneous of degree one on $B_1(0)$. Thus, if we denote $\sfrak_*$ and $\D_*$ the corresponding quantities for $w_*$, if we have $\sfrak_*\leq n-2$, then we know that $w_*$ consists of linear functions. The same holds if $\sfrak_* = n-1$ and $\H^{n-2}(\D_*\cap B_1)<\infty$ or $\D_* = B_1\cap S$; in all of these cases, we therefore have $w_*\in \mathfrak{L}_{\ell_*}$ for some $\ell_*\geq i_0$. But if $w_*\in \mathfrak{L}_{i_0}$ then $w_*$ would be dehomogenised by itself, which at $w_*$ is dehomogenised by $0$ would imply $w_*\equiv 0$, a contradiction. Thus, we must have $\ell_*>i_0$. But then we may apply Lemma \ref{lemma:blow-up-eps-reg} to $w_j$ for all $j$ sufficiently large to see that \eqref{E:RHS-alt} holds, again giving a contradiction to the failure of (b) along the sequence.
	
	Thus, we must have $\sfrak_*=n-1$, $\H^{n-2}(\D_*\cap B_1)=\infty$, and $\D_*\subsetneq S\cap B_1^n(0)$. But by construction, we know that $S(v)\subseteq S(w_*)$, and furthermore from the blow-up process we must also have $z\in S(w_*)$, giving that $S(v)\subsetneq S(w_*)$. Hence, $\dim(S(w_*))>\dim(S(v))$. By our inductive assumption, we then must have $w_*\in \mathfrak{L}_{\ell_*}$ for some $\ell_*\geq i_0$. We then reach the contradiction in exactly the same manner as in the other cases above. This contradiction therefore proves the reverse Hardt--Simon dichotomy.
	
	\textbf{Step 3:} \emph{$C^1$ regularity away from $S(v)$ and conclusion.} From the reverse Hardt--Simon dichotomy and Theorem \ref{thm:blow-up-properties} we then get the following dichotomy (by ``hole filling''): for each $z\in \D\cap K$ and $\rho\in (0,\eps]$, then either:
	\begin{enumerate}
		\item [(a$^\prime$)] we have \eqref{E:RHS} holds and
		$$\int_{B_{\rho/2}(z)}R_z^{2-n}\left|\frac{\del}{\del R_z}\left(\frac{v-v(z)}{R_z}\right)\right|^2 \leq \theta\int_{B_\rho(z)}R_z^{2-n}\left|\frac{\del}{\del R_z}\left(\frac{v-v(z)}{R_z}\right)\right|^2$$
		for some $\theta = \theta(v,K)\in (0,1)$;
		\item [(b$^\prime$)] conclusion (b) holds.
	\end{enumerate}
	If one iterates this dichotomy analogously to that in \cite[Proof of Theorem 3.11]{MW24}, we end up with the following dichotomy: for each $z\in \D\cap K$, either:
	\begin{enumerate}
		\item [(I)] for each $0<\sigma\leq\rho/2\leq\eps/4$, we have
		$$\sigma^{-n-2}\int_{B_\sigma(z)}|v-(v(z)+\pi_{z,\sigma})|^2 \leq C\left(\frac{\sigma}{\rho}\right)^{2\mu}\cdot\rho^{-n-2}\int_{B_\rho(z)}|v-(v(z)+\pi_{z,\rho})|^2;$$
		\item [(II)] or the conclusions of Lemma \ref{lemma:blow-up-eps-reg} hold on some neighbourhood of $z$, and moreover there exists $\psi_z\in \mathfrak{L}_{\ell_z}$ for some $\ell_z>i_0$ such that for all $0<\rho<3\eps/8]$,
		$$\rho^{-n-2}\int_{B_\rho(z)}|v-(v(z)+\psi_z)|^2 \leq C\rho^{2\mu}\int_{B_\eps(z)}|v|^2.$$
	\end{enumerate}
	Here, $C = C(v,K)\in (0,\infty)$ and $\mu = \mu(v,K)\in (0,1)$. It is straightforward to then check that in (I) one can take $\pi_{z,\sigma}$ to be independent of $\sigma$, i.e.~there exists $\pi_{z}\in \mathfrak{L}_{i_0}$ for which
	\begin{enumerate}
		\item [(I$^\prime$)] for each $0<\sigma\leq\rho/2\leq\eps/4$, we have
			$$\sigma^{-n-2}\int_{B_\sigma(z)}|v-(v(z)+\pi_{z})|^2 \leq C\left(\frac{\sigma}{\rho}\right)^{2\mu}\cdot\rho^{-n-2}\int_{B_\rho(z)}|v-(v(z)+\pi_{z})|^2.$$
	\end{enumerate}
	Thus, we have (I$^\prime$) or (II) at every $z\in \D\cap K$. At every point $y\in K\setminus \D$, we know that $v$ is harmonic. It then follows from the above estimates, harmonic estimates, and degenerate Campanato-style arguments (cf.~\cite[Lemma 4.3]{Wic14} or \cite{Min21b}) then on the interior of $K$ we have that $v$ must be $C^{1,\beta}$ on $\text{int}(K)\cap H_*\setminus S(v)$ for some $\beta = \beta(v,K)\in (0,1)$. In particular, $v$ must be $C^1$ up to the boundary on each connected component of $P_*\setminus S$. (As $S$ is $(n-1)$-dimensional, this is the same as saying on each half-hyperplane of $P_*$ determined by the complement of $S$ in $P_*$ we have $C^1$ regularity up to the boundary on each half-hyperplane. In particular, we only know that the derivatives parallel to $S$ are continuous at present, as some elements of $\mathfrak{L}_{\ell}$, with $\ell\geq i_0$, could consist of distinct linear functions on the components of $P_*\setminus S$.)
	
	Thus, our next aim is to show that $v$ must in fact be $C^1$ on $P_*\setminus S(v)$. To show this, we just need to show continuity along $S\setminus S(v)$ for each normal derivative of $v^{(i,k,s)}$. Fix $v^{(i,k,s)}$, which we know is a continuous function $P_*\to \R$. Write coordinates on $P_*$ as $(x,y)\in S^\perp\times S$. Set $r = |x|$ and define $\phi(r,y):= v^{(i,k,s)}(r,y) - v^{(i,k,s)}(-r,y)$. Then $\phi$ is a $C^1$ harmonic function on $H_*$ (a connected component of $P_*\setminus S$) with zero boundary values, and so must be of the form $\phi(r,y) = ar$ for some constant $a>0$. But then since $\D\subsetneq S\cap B_1$, there must exist a point $y\in S\cap B_1$ about which $v^{(i,k,s)}$ is harmonic, and thus smooth, and which point we would have $\del_r\phi(z)= 0$, and so necessarily $a=0$. But then this implies $\del_r\phi\equiv 0$, proving that $v^{(i,k,s)}$ (and hence $v$) is indeed $C^1$ on all of $P_*\setminus S(v)$.
	
	At present, we only know that $v$ is harmonic away from $\D$. We now need to argue that $v$ is harmonic on $P_*\setminus S(v)$. From Lemma \ref{lemma:kappa}, we know that for each $v^{(i,k,s)}$ there is a linear function $l$ for which $v^{(i,k,s)} - l$ vanishes on $\D$.  But then $v^{(i,k,s)}-l$ is a $C^1$ function on $P_*\setminus S(v)$ which is harmonic on $P_*\setminus (S(v)\cup \D)$. Since $\D\setminus S(v)$ is contained within $\{v^{(i,k,s)}-l =0\}$, and since the zero set of removable for a $C^1$ harmonic function, we see that $v^{(i,k,s)}-l$ must be harmonic on $P_*\setminus S(v)$, implying that $v^{(i,k,s)}$ must be harmonic here (as $l$ is linear). Hence, $v$ is harmonic on $P_*\setminus S(v)$. But then as $v$ is continuous on $P_*$, translation invariant along $S(v)$, and $\dim(S(v))\leq n-2$, it follows that $v$ must be harmonic on all of $P_*$, which implies it must be linear. This completes the proof of the classification.
\end{proof}

\subsection{$C^{1,\alpha}$ regularity of blow-ups}

Now that we have classified the possible homogeneous degree one blow-ups of $\mathfrak{B}_{i_0}$, we are ready to prove the $C^{1,\alpha}$ regularity of general blow-ups in $\mathfrak{B}_{i_0}$.

\begin{theorem}\label{thm:blow-up-reg}
	There exists $\alpha = \alpha(n,Q,\BC^{(0)})\in (0,1)$ such that if $v\in \mathfrak{B}_{i_0}$, then $v^{(i,k,s)}\in C^{1,\alpha}(\overline{\Omega_*}\cap B_{1/2}^n(0))$ for all $i,k,s$. Moreover, if $v\in \mathfrak{B}_{i_0}$ and $0\in \D$, then there exists $\phi\in \cup_{\ell\geq i_0}\mathfrak{L}_{\ell}$ such that for every $0<\sigma\leq \rho/2\leq 3/16$ we have
	$$\sigma^{-n-2}\int_{B_\sigma(0)}|v-(v(0) + \phi)|^2\leq C\left(\frac{\sigma}{\rho}\right)^{2\alpha}\cdot\rho^{-n-2}\int_{B_\rho(0)}|v|^2.$$
	Here, $C = C(n,Q,\BC^{(0)})\in (0,1)$.
\end{theorem}

\begin{proof}
	In view of Theorem \ref{thm:blow-up-properties} it suffices to prove the estimate when $\rho=3/8$.
	
	First note that one can repeat the argument which lead to the reverse Hardt--Simon dichotomy in Step 2 of the proof of Theorem \ref{thm:classification}, except now using the classification of homogeneous degree one blow-ups provided by Theorem \ref{thm:classification} to get the contradiction (via Lemma \ref{lemma:blow-up-eps-reg}) to prove the following dichotomy: there exists $\eps = \eps(n,Q,\BC^{(0)})\in (0,1)$ such that for each $v\in \mathfrak{B}_{i_0}$, $z\in \D\cap B_{1/4}(0)$, and $\rho\in (0,\eps]$, we have either:
	\begin{enumerate}
		\item [(a)] the reverse Hardt--Simon inequality
		$$\int_{B_\rho(z)\setminus B_{\rho/2}(z)}R_z^{2-n}\left|\frac{\del}{\del R_z}\left(\frac{v-v(z)}{R_z}\right)\right|^2 \geq \eps\rho^{-n-2}\int_{B_\rho(z)}|v - (v(z)+\pi_{z,\rho})|^2;$$
		\item [(b)] or the conclusions of Lemma \ref{lemma:blow-up-eps-reg} hold on $B_{3\rho/8}(z)$; in particular, there exists $\psi_z\in \mathfrak{L}_{\ell_z}$ for some $\ell_z>i_0$ such that for all $0<\sigma\leq3\rho/8$,
		\begin{equation}
		\sigma^{-n-2}\int_{B_\sigma(z)}|v-(v(z)+\psi_z)|^2 \leq C\left(\frac{\sigma}{\rho}\right)^{2\alpha}\cdot\rho^{-n-2}\int_{B_\rho(z)}|v-(v(z)+\psi_z)|^2.
	\end{equation}
	\end{enumerate}
	We may then follow the argument in Step 3 of the proof of Theorem \ref{thm:classification} to get that either (I$^\prime$) or (II) holds at each $z\in \D\cap B_{1/4}(0)$, where now all the constants depend only on $n,Q,\BC^{(0)}$. But then we know $v$ is harmonic on $B_{1/4}(0)\setminus$ by Theorem \ref{thm:blow-up-properties}(1), and so once again we can combine the above estimates centered on $\D\cap B_{1/4}(0)$ with harmonic estimates away from $\D$ in a degenerate Camapanato style in order to prove the claimed regularity statement and estimate.
\end{proof}

\section{Excess Decay and Conclusion}\label{sec:excess-decay}

In this section, we now complete the proof of Theorem \ref{thm:fine-reg} by induction on the level of the cone $\BC$, working down the order $\{\mathfrak{C}_{i}\}_{i=1}^N$. For this, we use the $C^{1,\alpha}$ regularity of blow-ups established in Theorem \ref{thm:blow-up-reg} and, in all situations except the base case (namely, when the cones are finest), Theorem \ref{thm:fine-reg} for finer cones (which is by Hypothesis F).

\begin{proof}[Proof of Theorem \ref{thm:fine-reg}]
	The proof is by downwards induction on $i_0$, where $\BC\in \mathfrak{C}_{i_0}$. As such, we may assume Hypothesis F in our proof. The base case, when $i_0$ is maximal, is strictly simpler than the argument given in the inductive step (as in that situation one only has decay to a cone of the same level), and so we do not write it separately.
	
	Here, for a point $y\in B^{n+1}_1(0)$ and $\rho>0$, we will write $E_{V,\BC}(B_\rho(y))$ for the scale-invariant $L^2$ height-excess on the ball $B_\rho(y)$ of $V$ relative to $\BC$, namely
	$$E^2_{V,\BC}(B_\rho(y)) = \rho^{-n-2}\int_{B_\rho^{n+1}(y)}\dist^2(x,\BC)\, \ext\|V\|(x).$$
	We will also write $F_{V,\BC}(B_\rho(y))$, and so forth for the corresponding quantities.
	
	We first claim the following: there exists $\eps = \eps(n,Q,\BC^{(0)})\in (0,1)$, $\beta = \beta(n,Q,\BC^{(0)})\in (0,1)$, and $\theta = \theta(n,Q,\BC^{(0)})\in (0,1)$ such that if $V\in \V_Q$, $\BC\in \mathfrak{C}_{i_0}$ obey the assumptions of Theorem \ref{thm:fine-reg} with these choices of $\eps,\beta$ therein, then either:
	\begin{enumerate}
		\item [(i)] there is a cone $\widetilde{\BC}\in\mathfrak{C}_{i_0}$ and an orthogonal rotation $R:\R^{n+1}\to \R^{n+1}$ with
		\begin{itemize}
			\item $\displaystyle \dist_\H(\BC\cap B_1, \widetilde{\BC}\cap B_1) \leq CE_{V,\BC}$;
			\item $\displaystyle |R-\id_{\R^{n+1}}| \leq C\mathcal{F}_{V,\BC}^{-1}F_{V,\BC}$;
			\item $E_{R_\#V,\widetilde{\BC}}(B_\theta) \leq \frac{1}{2}E_{V,\BC}$;
			\item $\displaystyle \frac{F_{R_\#V,\widetilde{\BC}}(B_\theta)}{\mathcal{F}_{R_\#V,\widetilde{\BC}}(B_\theta)}\leq \frac{1}{2}\cdot\frac{F_{V,\BC}}{\mathcal{F}_{V,\BC}}$;
		\end{itemize}
		or,
		\item [(ii)] there is a cone $\widehat{\BC}\in \mathfrak{C}_\ell$ for some $\ell>i_0$ and a rotation $\Gamma:\R^{n+1}\to \R^{n+1}$ with
		\begin{itemize}
			\item $\displaystyle \dist_\H(\BC\cap B_1,\widehat{\BC}\cap B_1) \leq CE_{V,\BC}$;
			\item $\displaystyle |R-\id_{\R^{n+1}}| \leq C\mathcal{F}_{V,\BC}^{-1}F_{V,\BC}$;
			\item $\displaystyle E_{R_\#V,\widehat{\BC}}(B_\rho) \leq C\rho^{\mu}E_{V,\BC}$ for all $\rho\in (0,\theta/8]$.
		\end{itemize}
	\end{enumerate}
	Here, $C = C(n,Q,\BC^{(0)})\in (0,\infty)$ and $\mu = \mu(n,Q,\BC^{(0)})\in (0,1)$. We stress that alternative (ii) does not occur in the base case, and in that situation only alternative (i) occurs.
	
	We prove this claim by contradiction. Indeed, if it were false we could find sequences $\eps_j,\beta_j\downarrow 0$ such that there are $V_j\in \V_Q$, and $\BC_j\in \mathfrak{C}_{i_0}$ which obey the assumptions of Theorem \ref{thm:fine-reg} with $\eps_j,\beta_j$ therein yet fail the above dichotomy for all $j$. By passing to a subsequence, we can construct a blow-up $v\in \mathfrak{B}_{i_0}$ of $V_j$ relative to $\BC_j$. As $\Theta_{V_j}(0)\geq Q$ for all $j$ by assumption, we know that $v(0) = 0$. Thus, Theorem \ref{thm:blow-up-reg} gives that there exists $\phi\in \cup_{\ell\geq i_0}\mathfrak{L}_{\ell}$ such that for all $0<\sigma\leq 3/8$ we have
	$$\sigma^{-n-2}\int_{B_{\sigma}(0)}|v-\phi|^2 \leq C\sigma^{2\alpha}\int_{B_{1/2}(0)}|v|^2.$$
	One can now follow the argument in \cite[Proof of Theorem A]{MW24}, except now taking the corresponding average of $\phi$ to be for each $(i,k)$ the average of $(\phi^{(i,k,s)})_{s}$; denote this by $\phi_a = (\phi^{(i,k)}_a)_{i,k}$. The outline of this argument is as follows. One sees that, if $v$ is sufficiently close to $\phi_a$ on $B_{2\theta}(0)$ for sufficiently chosen $\theta = \theta(n,Q,\BC^{(0)})\in (0,1)$, then one can prove for all $j$ sufficiently large that conclusion (i) must hold in the standard excess-decay manner, for a sufficiently chosen scale $\theta = \theta(n,Q,\BC^{(0)})\in (0,1)$; here the new cone $\widetilde{\BC}_j$ is determined by perturbing each (half-)hyperplane $P^{(i,k)}$ in $\BC_j$ by $E_{V_j,\BC_j}\phi^{(i,k)}_a$, with the rotation $R_j$ coming from the need to align this new cone with $\BC^{(0)}$ (as the perturbation may modify the spine of the cone). If instead $v$ is sufficiently far from $\phi_a$ on $B_{2\theta}(0)$ (cf.~\cite[Case 2 in Proof of Theorem A]{MW24}), then one can verify the assumptions of Theorem \ref{thm:fine-reg} hold in $B^{n+1}_{2\theta}(0)$ in an analogous way to the proof of Lemma \ref{lemma:blow-up-eps-reg}, which gives conclusion (ii) must hold. This completes the proof of the above dichotomy.
	
	We now iterate the dichotomy from (i) and (ii): whenever (i) holds, we can reapply it with $(R\circ\eta_{0,\theta})_\#V$ and $\widetilde{\BC}$ in place of $V$ and $\BC$, respectively. The iteration stops whenever (ii) holds. Combining the resulting dichotomy in an analogous manner to that in Theorem \ref{thm:blow-up-reg} (cf.~Theorem \ref{thm:classification}) allows us to prove that
	\begin{enumerate}
		\item [(A)] there is (a unique) $\BC_0 \in \cup_{\ell\geq i_0}\mathfrak{C}_{\ell}$ and rotation $R_0:\R^{n+1}\to \R^{n+1}$ with:
		\begin{itemize}
			\item $\displaystyle \dist_\H(\BC_0\cap B_1, \BC\cap B_1) \leq CE_{V,\BC}$;
			\item $\displaystyle |R_0-\id_{\R^{n+1}}|\leq \mathcal{F}_{V,\BC}^{-1}F_{V,\BC}$;
			\item $\displaystyle E_{(R_0)_\#V,\BC_0}(B_\rho(0)) \leq C\rho^{\beta}E_{V,\BC}$ for all $\rho\in (0,\theta/8]$.
		\end{itemize}
	\end{enumerate}
	Here, $\beta = \beta(n,Q,\BC^{(0)})\in (0,1)$. One can then repeat this at any good density point $z\in B_{1/4}^{n+1}(0)$ by applying (A) to $(\eta_{z,1/2})_\#V$ and $\BC$ (this is possible provided $\eps,\beta$ are sufficiently small, as one must verify the hypotheses of Theorem \ref{thm:fine-reg} for these new choices, and this is done in analogous ways to those in the proofs of Theorem \ref{thm:blow-up-properties}(6)). This gives the uniqueness of tangent cones (with tangent cones in $\cup_{\ell\geq i_0}\mathfrak{C}_{\ell}$) for at every good density point in $B_{1/4}^{n+1}(0)$. Away from the good density set, we know from Lemma \ref{lemma:gaps} that $\dim_\H(\sing_*(V)\cap \{\Theta_V<Q\})\leq n-7$, and so one may instead apply Theorem \ref{thm:HLW} in order to get the desired result. Combining these conclusions then completes the proof.
\end{proof}

\bibliographystyle{alpha} 
\bibliography{references-1}

\end{document}